\documentclass[11pt]{article}
\usepackage{subfiles}
\usepackage{amssymb}
\usepackage{pifont}

\usepackage{amsfonts}
\usepackage{amsmath, amsthm}
\usepackage{amssymb, amsxtra}

\usepackage{graphicx}

\usepackage{amscd}
\usepackage{multicol}
\usepackage{epstopdf}
\usepackage{color}
\usepackage{xparse}

\usepackage{epsf}
\usepackage{float}
\usepackage{extarrows}
\usepackage[a4paper, text={.8\paperwidth,.83\paperheight}, ratio=1:1]{geometry}
\usepackage[format=hang,font=small,textfont=it]{caption}
\usepackage{authblk}

\usepackage[english]{babel}
\usepackage{latexsym}
\usepackage{nicefrac}

\usepackage{tikz}
\usetikzlibrary{calc}
\usepackage[tight]{subfigure}
\usepackage{afterpage}
\usepackage{hyperref}
\usepackage{color}
\definecolor{greenbean}{RGB}{199,237,204}
\newtheorem{thm}{Theorem}[section]
\newtheorem{lemma}[thm]{Lemma}
\newtheorem{prop}[thm]{Proposition}

\newtheorem{Qst}{Question}[section]
\newtheorem{conjeceture}{Conjecture}[section]

\theoremstyle{definition}
\newtheorem{Def}[thm]{Definition}
\newtheorem{eg}{Example}[section]

\def\C{{\mathbb C}}

\def\P{{\mathbb P}}

\def \ta{\tau}

\def \ta1{\tau_1}

\def \G{\Gamma}

\newcommand\Gal[1]{{{#1}_{\operatorname{Gal}}}}

\newcommand\set[1]{{\{{#1}\}}}

\newcommand{\Xgal}{X_{\text{Gal}}}

\newcommand{\ug}[1]{\Gamma_#1}
\newcommand{\Ggal}{\widetilde{G}}

\makeatletter
\newcommand{\uGammaSq}[2]{\begin{equation}\label{#1}
	\ug{#2}^2\uGammaSqChecknextarg}
\newcommand{\uGammaSqChecknextarg}{\@ifnextchar\bgroup{\uGammaSqGobblenextarg}{ = e,
		\end{equation} }}
\newcommand{\uGammaSqGobblenextarg}[1]{ = \ug{#1}^2\@ifnextchar\bgroup{\uGammaSqGobblenextarg}{  = e
\end{equation}}}
\makeatother
\newcommand{\mbb}[1]{\mathbb{#1}}
\newcommand{\todo}[1]{\textcolor{red}{[TODO: #1]}}

\newcommand{\ubegineq}[1]{\begin{equation}\label{#1}}
\newcommand{\uendeq}{\end{equation}}
\newcommand{\trip}[2] {{\langle \ug{#1}, \ug{#2} \rangle}}

\newcommand{\comm}[2] {{[\ug{#1}, \ug{#2}]}}

\newcommand{\uOnePoint}[3]{%
\noindent
Vertex $ #1 $ is a $1$-point that gives rise to braid  $Z_{#2 \; #2^{'}}$ and to the following relation in $G$: %
\begin{equation}\label{#3} %
\Gamma_#2=\Gamma_#2'. %
\end{equation} %
}

\newcommand{\uTwoPoint}[4]{
Vertex $ #1 $ is a $2$-point. The braid monodromy corresponding to this point is:
\begin{equation}\label{#4-1}
\widetilde{\Delta_#1} = (Z_{#3\; #3'})^{Z_{#2 \; #2', #3}^2}\cdot Z_{#2 \; #2',#3}^3. \nonumber
\end{equation}
$\widetilde{\Delta_#1}$ gives rise to the following relations:
\begin{equation}\label{#4-2}
\langle \Gamma_#2, \Gamma_#3 \rangle = \langle \Gamma_#2', \Gamma_#3 \rangle = \langle \Gamma_#2^{-1}\Gamma_#2'\Gamma_#2, \Gamma_#3 \rangle = e
\end{equation}
\begin{equation}\label{#4-3}
\Gamma_#3' = \Gamma_#3\Gamma_#2'\Gamma_#2\Gamma_#3\Gamma_#2^{-1}\Gamma_#2'^{-1}\Gamma_#3^{-1}.
\end{equation}
}
\newcommand{\uThreePointOuter}[5]{
Vertex $ #1 $ is an outer $3$-point. The braid monodromy corresponding to this $3$-point is:
\begin{eqnarray} \label{#5-1}
{\widetilde{\Delta_#1}} = Z_{#2 \; #2', #4 \; #4'}^2 \cdot Z_{#2',#3\;#3'}^3\cdot(Z_{#2\;#2'})^{Z_{#2',#3\;#3'}^2}\cdot(Z_{#3\;#3',#4}^3)^{Z_{#2',#3\;#3'}^2}
\cdot(Z_{#4\;#4'})^{Z_{#3\;#3',#4}^2Z_{#2',#3\;#3'}^2}. \nonumber
\end{eqnarray}
${\widetilde{\Delta_#1}}$ thus gives rise to the following relations:
\begin{equation}\label{#5-2}
\langle\Gamma_#2',\Gamma_#3\rangle=\langle\Gamma_#2',\Gamma_#3'\rangle=\langle\Gamma_#2',\Gamma_#3^{-1}\Gamma_#3'\Gamma_#3\rangle=e
\end{equation}
\begin{equation}\label{#5-3}
\Gamma_#2=\Gamma_#3'\Gamma_#3\Gamma_#2'\Gamma_#3^{-1}{\Gamma_#3'}^{-1}
\end{equation}
\begin{equation}\label{#5-4}
\begin{split}
\langle\Gamma_#4,\Gamma_#3'\Gamma_#3\Gamma_#2'\Gamma_#3{\Gamma_#2'}^{-1}\Gamma_#3^{-1}{\Gamma_#3'}^{-1}\rangle=&
\langle\Gamma_#4,\Gamma_#3'\Gamma_#3\Gamma_#2'\Gamma_#3'{\Gamma_#2'}^{-1}\Gamma_#3^{-1}{\Gamma_#3'}^{-1}\rangle= \\
=& \langle\Gamma_#4,\Gamma_#3'\Gamma_#3\Gamma_#2'\Gamma_#3^{-1}\Gamma_#3'\Gamma_#3{\Gamma_#2'}^{-1}\Gamma_#3^{-1}{\Gamma_#3'}^{-1}\rangle=e
\end{split}
\end{equation}
\begin{equation}\label{#5-5}
\Gamma_#4'=\Gamma_#4\Gamma_#3'\Gamma_#3\Gamma_#2'\Gamma_#3'\Gamma_#3{\Gamma_#2'}^{-1}\Gamma_#3^{-1}{\Gamma_#3'}^{-1}\Gamma_#4
\Gamma_#3'\Gamma_#3\Gamma_#2'\Gamma_#3^{-1}{\Gamma_#3'}^{-1}{\Gamma_#2'}^{-1}\Gamma_#3^{-1}{\Gamma_#3'}^{-1}\Gamma_#4^{-1}
\end{equation}
\begin{equation}\label{#5-6}
[\Gamma_#2,\Gamma_#4]=[\Gamma_#2,\Gamma_#4']=[\Gamma_#2',\Gamma_#4]=[\Gamma_#2',\Gamma_#4']=e.
\end{equation}
}
\newcommand{\uThreePointInner}[5]{
Vertex $ #1 $ is an inner $3$-point. Its braid monodromy is
\begin{eqnarray}
\begin{split} \label{#5-1}
{\widetilde{\Delta_#1}} = & Z_{#2',#3\;#3'}^3 \cdot Z_{#2,#3\;#3'}^3 \cdot (Z_{#3\;#4'})^{Z_{#3\;#3'}^2 Z_{#2',#3\;#3'}^2} \cdot (Z_{#3'\;#4})^{Z_{#3\;#3'}^2 Z_{#4\;#4'}^2 Z_{#2',#3\;#3'}^2} \cdot \\
& (Z_{#3\;#4'})^{Z_{#3\;#3'}^2 Z_{#2,#3\;#3'}^2} \cdot (Z_{#3'\;#4})^{Z_{#3\;#3'}^2 Z_{#4\;#4'}^2 Z_{#2,#3\;#3'}^2}. \nonumber
\end{split}
\end{eqnarray}
These braids give rise to the following relations in $G$:
\begin{equation}\label{#5-1}
\langle\Gamma_#2',\Gamma_#3\rangle=\langle\Gamma_#2',\Gamma_#3'\rangle=\langle\Gamma_#2',\Gamma_#3^{-1}\Gamma_#3'\Gamma_#3\rangle=e
\end{equation}
\begin{equation}\label{#5-2}
\langle\Gamma_#2,\Gamma_#3\rangle=\langle\Gamma_#2,\Gamma_#3'\rangle=\langle\Gamma_#2,\Gamma_#3^{-1}\Gamma_#3'\Gamma_#3\rangle=e
\end{equation}
\begin{equation}\label{#5-3}
\Gamma_#3'\Gamma_#3\Gamma_#2'\Gamma_#3\Gamma_#2'^{-1}\Gamma_#3^{-1}{\Gamma_#3'}^{-1} = \Gamma_#4'
\end{equation}
\begin{equation}\label{#5-4}
\Gamma_#3'\Gamma_#3\Gamma_#2'\Gamma_#3'\Gamma_#2'^{-1}\Gamma_#3^{-1}{\Gamma_#3'}^{-1} = \Gamma_#4'\Gamma_#4{\Gamma_#4'}^{-1}
\end{equation}
\begin{equation}\label{#5-5}
\Gamma_#3'\Gamma_#3\Gamma_#2\Gamma_#3\Gamma_#2^{-1}\Gamma_#3^{-1}{\Gamma_#3'}^{-1} = \Gamma_#4'
\end{equation}
\begin{equation}\label{#5-6}
\Gamma_#3'\Gamma_#3\Gamma_#2\Gamma_#3'\Gamma_#2^{-1}\Gamma_#3^{-1}{\Gamma_#3'}^{-1} = \Gamma_#4'\Gamma_#4\Gamma_#4'^{-1}.
\end{equation}
}
\newcommand{\uFourPointInner}[6]{
	Vertex $#1$ is an inner $4$-point.
	Its braid monodromy appears in \cite{Ogata}.
	The corresponding relations in $G$ are:
	\begin{equation}\label{#6-1} \langle\Gamma_#2',\Gamma_#3\rangle=\langle\Gamma_#2',\Gamma_#3'\rangle=\langle\Gamma_#2',\Gamma_#3^{-1}\Gamma_#3'\Gamma_#3\rangle=e
	\end{equation}
	\begin{equation}\label{#6-2} \langle\Gamma_#4,\Gamma_#5\rangle=\langle\Gamma_#4',\Gamma_#5\rangle=\langle\Gamma_#4^{-1}\Gamma_#4'\Gamma_#4,\Gamma_#5\rangle=e
	\end{equation}
	\begin{equation}\label{#6-3}
	[\Gamma_#3'\Gamma_#3\Gamma_#2'\Gamma_#3^{-1}{\Gamma_#3'}^{-1},\Gamma_#5] = e
	\end{equation}
	\begin{equation}\label{#6-4} [\Gamma_#3'\Gamma_#3\Gamma_#2'\Gamma_#3^{-1}{\Gamma_#3'}^{-1},\Gamma_#4^{-1}{\Gamma_#4'}^{-1}\Gamma_#5^{-1}\Gamma_#5'\Gamma_#5\Gamma_#4'\Gamma_#4] = e
	\end{equation}
	\begin{equation}\label{#6-5} \langle\Gamma_#2,\Gamma_#3\rangle=\langle\Gamma_#2,\Gamma_#3'\rangle=\langle\Gamma_#2,\Gamma_#3^{-1}\Gamma_#3'\Gamma_#3\rangle=e
	\end{equation}
	\begin{equation}\label{#6-6} \langle\Gamma_#4,\Gamma_#5^{-1}\Gamma_#5'\Gamma_#5\rangle=\langle\Gamma_#4',\Gamma_#5^{-1}\Gamma_#5'\Gamma_#5\rangle=\langle\Gamma_#4^{-1}
\Gamma_#4'\Gamma_#4,\Gamma_#5^{-1}\Gamma_#5'\Gamma_#5\rangle=e
	\end{equation}
	\begin{equation}\label{#6-7}
	[\Gamma_#3'\Gamma_#3\Gamma_#2\Gamma_#3^{-1}{\Gamma_#3'}^{-1},\Gamma_#5^{-1}\Gamma_#5'\Gamma_#5] = e
	\end{equation}
	\begin{equation}\label{#6-8}
	[\Gamma_#3'\Gamma_#3\Gamma_#2\Gamma_#3^{-1}{\Gamma_#3'}^{-1}, \Gamma_#4^{-1}{\Gamma_#4'}^{-1}\Gamma_#5^{-1}{\Gamma_#5'}^{-1}\Gamma_#5\Gamma_#5'\Gamma_#5\Gamma_#4'\Gamma_#4] = e
	\end{equation}
	\begin{equation}\label{#6-9}
	\Gamma_#3'\Gamma_#3\Gamma_#2'\Gamma_#3\Gamma_#2'^{-1}\Gamma_#3^{-1}{\Gamma_#3'}^{-1} = \Gamma_#5\Gamma_#4'\Gamma_#5^{-1}
	\end{equation}
	\begin{equation}\label{#6-10}
	\Gamma_#3'\Gamma_#3\Gamma_#2'\Gamma_#3'\Gamma_#2'^{-1}\Gamma_#3^{-1}{\Gamma_#3'}^{-1} = \Gamma_#5\Gamma_#4'\Gamma_#4{\Gamma_#4'}^{-1}\Gamma_#5^{-1}
	\end{equation}
	\begin{equation}\label{#6-11}
	\Gamma_#3'\Gamma_#3\Gamma_#2\Gamma_#3\Gamma_#2^{-1}\Gamma_#3^{-1}{\Gamma_#3'}^{-1} = \Gamma_#5^{-1}\Gamma_#5'\Gamma_#5\Gamma_#4'\Gamma_#5^{-1}{\Gamma_#5'}^{-1}\Gamma_#5
	\end{equation}
	\begin{equation}\label{#6-12}
	\Gamma_#3'\Gamma_#3\Gamma_#2\Gamma_#3'\Gamma_#2^{-1}\Gamma_#3^{-1}{\Gamma_#3'}^{-1} = \Gamma_#5^{-1}\Gamma_#5'\Gamma_#5\Gamma_#4'\Gamma_#4\Gamma_#4'^{-1}\Gamma_#5^{-1}{\Gamma_#5'}^{-1}\Gamma_#5.
	\end{equation}
}
\newcommand{\uFourPointOuter}[6]{
	Vertex $ #1 $ is a $4$-point. Its braid monodromy is
	\begin{align*}\label{#6-1}
	\widetilde{\Delta}_{#1} &=  Z_{#4\; #4', #5}^3 \cdot (Z_{#2\; #2', #5}^2)^{Z_{#3\; #3', #5}^{-2}} \cdot Z_{#3\; #3', #5}^2 \cdot {\bar{Z}}_{#2\; #2', #5'}^2 \cdot {\bar{Z}}_{#3\; #3', #5'}^2 \cdot (Z_{#5\; #5'})^{Z_{#4\; #4', #5}^2} \cdot Z_{#2', #3 \; #3'}^3 \cdot
	(Z_{#2\; #2'})^{Z_{#2', #3\; #3'}^2} \cdot \\
	& (Z_{#3\; #3', #4}^3)^{Z_{#2', #3\; #3'}^2 Z_{#4\; #4', #5}^2} \cdot
	(Z_{#4\; #4'})^{Z_{#3\; #3', #4}^2 Z_{#2', #3\; #3'}^2 Z_{#4\; #4', #5}^2} \cdot (Z_{#2\; #2', #4\; #4'}^2)^{Z_{#4\; #4' , #5}^2}.
	\end{align*}
	It gives rise to the following relations:
	\begin{equation}\label{#6-2}
		\trip{#4}{#5} = \trip{#4'}{#5} = \langle \ug{#4}^{-1}\ug{#4'}\ug{#4}, \ug{#5} \rangle =e
	\end{equation}
	\begin{equation}\label{#6-4}
	[\ug{#3'}\ug{#3}\ug{#2}\ug{#3}^{-1}\ug{#3'}^{-1}, \ug{#5}] = [\ug{#3'}\ug{#3}\ug{#2'}\ug{#3}^{-1}\ug{#3'}^{-1}, \ug{#5} ] = e
	\end{equation}
	\begin{equation}\label{#6-5}
		\comm{#3}{#5} = \comm{#3'}{#5} = e
	\end{equation}
	\begin{equation}\label{#6-6}
	[\ug{#4'}\ug{#4}\ug{#3'}\ug{#3}\ug{#2}\ug{#3}^{-1}\ug{#3}'^{-1}\ug{#4}^{-1}\ug{#4'}^{-1}, \ug{#5}^{-1}\ug{#5'}\ug{#5} ] = 	 [\ug{#4'}\ug{#4}\ug{#3'}\ug{#3}\ug{#2'}\ug{#3}^{-1}\ug{#3}'^{-1}\ug{#4}^{-1}\ug{#4'}^{-1}, \ug{#5}^{-1}\ug{#5'}\ug{#5} ] = e
	\end{equation}
	\begin{equation}\label{#6-7}
	[\ug{#4'}\ug{#4}\ug{#3}\ug{#4}^{-1}\ug{#4'}^{-1}, \ug{#5}^{-1}\ug{#5'}\ug{#5}] = 	 [\ug{#4'}\ug{#4}\ug{#3'}\ug{#4}^{-1}\ug{#4'}^{-1}, \ug{#5}^{-1}\ug{#5'}\ug{#5}] = e
	\end{equation}
	\begin{equation}\label{#6-8}
	\ug{#5}\ug{#4'}\ug{#4}\ug{#5}\ug{#4}^{-1}\ug{#4'}^{-1}\ug{#5}^{-1} = \ug{#5'}
	\end{equation}
	\begin{equation}\label{#6-9}
	\trip{#2'}{#3} = \trip{#2'}{#3'} = \langle \ug{#2'}, \ug{#3}^{-1}\ug{#3'}\ug{#3} \rangle =e
	\end{equation}
	\begin{equation}\label{#6-10}
	\ug{#2} = \ug{#3'}\ug{#3}\ug{#2'}\ug{#3}^{-1}\ug{#3'}^{-1}
	\end{equation}
	\begin{equation}\label{#6-11}
	\begin{split}
	\langle \ug{#3'}\ug{#3}\ug{#2'} \ug{#3} \ug{#2'}^{-1}\ug{#3}^{-1}\ug{#3'}^{-1}, \ug{#5}\ug{#4}\ug{#5}^{-1} \rangle = \langle \ug{#3'}\ug{#3}\ug{#2'} \ug{#3'} \ug{#2'}^{-1}\ug{#3}^{-1}\ug{#3'}^{-1}, \ug{#5}\ug{#4}\ug{#5}^{-1} \rangle = \\
	=  \langle \ug{#3'}\ug{#3}\ug{#2'} \ug{#3}^{-1}\ug{#3'}\ug{#3} \ug{#2'}^{-1}\ug{#3}^{-1}\ug{#3'}^{-1}, \ug{#5}\ug{#4}\ug{#5}^{-1} \rangle = e
	\end{split}
	\end{equation}
	\begin{equation}\label{#6-12} \ug{#3'}\ug{#3}\ug{#2'}\ug{#3}^{-1}\ug{#3'}^{-1}\ug{#2'}^{-1}\ug{#3}^{-1}\ug{#3'}^{-1}\ug{#5}\ug{#4}^{-1}
\ug{#4'}\ug{#4}\ug{#5}^{-1}\ug{#3'}\ug{#3}\ug{#2'}\ug{#3'}\ug{#3}\ug{#2'}^{-1}\ug{#3}^{-1}\ug{#3'}^{-1} = \ug{#5}\ug{#4}\ug{#5}^{-1}
	\end{equation}
	\begin{equation}\label{#6-13}
	[\ug{#2}, \ug{#5}\ug{#4}\ug{#5}^{-1}] = [\ug{#2'}, \ug{#5}\ug{#4}\ug{#5}^{-1}] = [\ug{#2}, \ug{#5}\ug{#4'}\ug{#5}^{-1}] =
[\ug{#2'}, \ug{#5}\ug{#4'}\ug{#5}^{-1}] = e.
	\end{equation}
} % End of 4-point type QQH
\newcommand{\uFivePointInner}[7]{
Vertex $ #1 $ is a $5$-point. According to \cite[Corollary\;2.5]{FT08},
the braid monodromy corresponding to it yields the following relations in $G$:
\begin{equation}\label{#7-1}
[\Gamma_#4,\Gamma_#5]=[\Gamma_#4',\Gamma_#5]=e
\end{equation}
\begin{equation}\label{#7-2}
\langle\Gamma_#5',\Gamma_#6\rangle=\langle\Gamma_#5',\Gamma_#6'\rangle=\langle\Gamma_#5',\Gamma_#6^{-1}\Gamma_#6'\Gamma_#6\rangle=e
\end{equation}
\begin{equation}\label{#7-3}
\langle\Gamma_#3,\Gamma_#5\rangle=\langle\Gamma_#3',\Gamma_#5\rangle=\langle\Gamma_#3^{-1}\Gamma_#3'\Gamma_#3,\Gamma_#5\rangle=e
\end{equation}
\begin{equation}\label{#7-4}
[\Gamma_#5\Gamma_#4\Gamma_#5^{-1},\Gamma_#6'\Gamma_#6\Gamma_#5'\Gamma_#6^{-1}{\Gamma_#6'}^{-1}]=
[\Gamma_#5\Gamma_#4'\Gamma_#5^{-1},\Gamma_#6'\Gamma_#6\Gamma_#5'\Gamma_#6^{-1}{\Gamma_#6'}^{-1}]=e
\end{equation}
\begin{equation}\label{#7-5}
\Gamma_#5\Gamma_#3'\Gamma_#3\Gamma_#5\Gamma_#3^{-1}{\Gamma_#3'}^{-1}\Gamma_#5^{-1}=\Gamma_#6'\Gamma_#6\Gamma_#5'\Gamma_#6^{-1}{\Gamma_#6'}^{-1}
\end{equation}
\begin{equation}\label{#7-6}
[\Gamma_#2,\Gamma_#5]=[\Gamma_#2',\Gamma_#5]=
[\Gamma_#2,\Gamma_#6'\Gamma_#6\Gamma_#5'\Gamma_#6^{-1}{\Gamma_#6'}^{-1}]=[\Gamma_#2',\Gamma_#6'\Gamma_#6\Gamma_#5'\Gamma_#6^{-1}{\Gamma_#6'}^{-1}]=e
\end{equation}
\begin{equation}\label{#7-7}
\langle\Gamma_#2',\Gamma_#3\rangle=\langle\Gamma_#2',\Gamma_#3'\rangle=\langle\Gamma_#2',\Gamma_#3^{-1}\Gamma_#3'\Gamma_#3\rangle=e
\end{equation}
\begin{equation}\label{#7-8}
\langle\Gamma_#5\Gamma_#4\Gamma_#5^{-1},\Gamma_#6\rangle=\langle\Gamma_#5\Gamma_#4'\Gamma_#5^{-1},\Gamma_#6\rangle=\langle\Gamma_#5\Gamma_#4^{-1}\Gamma_#4'\Gamma_#4\Gamma_#5^{-1},\Gamma_#6\rangle=e
\end{equation}
\begin{equation}\label{#7-9}
\Gamma_#3'\Gamma_#3\Gamma_#2'\Gamma_#3\Gamma_#2'^{-1}\Gamma_#3^{-1}{\Gamma_#3'}^{-1}=\Gamma_#5^{-1}\Gamma_#6\Gamma_#5\Gamma_#4'
\Gamma_#5^{-1}\Gamma_#6^{-1}\Gamma_#5
\end{equation}
\begin{equation}\label{#7-10}
\Gamma_#3'\Gamma_#3\Gamma_#2'\Gamma_#3'\Gamma_#2'^{-1}\Gamma_#3^{-1}{\Gamma_#3'}^{-1}=
\Gamma_#5^{-1}\Gamma_#6\Gamma_#5\Gamma_#4'\Gamma_#4\Gamma_#4'^{-1}\Gamma_#5^{-1}\Gamma_#6^{-1}\Gamma_#5
\end{equation}
\begin{equation}\label{#7-11}
[\Gamma_#3'\Gamma_#3\Gamma_#2'\Gamma_#3^{-1}{\Gamma_#3'}^{-1},\Gamma_#5^{-1}\Gamma_#6\Gamma_#5]=e
\end{equation}
\begin{equation}\label{#7-12}
[\Gamma_#4'\Gamma_#4\Gamma_#3'\Gamma_#3\Gamma_#2'\Gamma_#3^{-1}{\Gamma_#3'}^{-1}\Gamma_#4^{-1}\Gamma_#4'^{-1},
\Gamma_#5^{-1}\Gamma_#6^{-1}\Gamma_#6'\Gamma_#6\Gamma_#5]=e
\end{equation}
\begin{equation}\label{#7-13}
\langle\Gamma_#2,\Gamma_#3\rangle=\langle\Gamma_#2,\Gamma_#3'\rangle=\langle\Gamma_#2,\Gamma_#3^{-1}\Gamma_#3'\Gamma_#3\rangle=e
\end{equation}
\begin{equation} \label{#7-14}
\begin{split}
\langle\Gamma_#5\Gamma_#4\Gamma_#5^{-1},\Gamma_#6^{-1}\Gamma_#6'\Gamma_#6\rangle&=\langle\Gamma_#5\Gamma_#4'\Gamma_#5^{-1},
\Gamma_#6^{-1}\Gamma_#6'\Gamma_#6\rangle\\
&=\langle\Gamma_#5\Gamma_#4^{-1}\Gamma_#4'\Gamma_#4\Gamma_#5^{-1},\Gamma_#6^{-1}\Gamma_#6'\Gamma_#6\rangle=e
\end{split}
\end{equation}
\begin{equation}\label{#7-15}
\Gamma_#3'\Gamma_#3\Gamma_#2\Gamma_#3\Gamma_#2^{-1}\Gamma_#3^{-1}{\Gamma_#3'}^{-1}=
\Gamma_#5^{-1}\Gamma_#6^{-1}\Gamma_#6'\Gamma_#6\Gamma_#5\Gamma_#4'\Gamma_#5^{-1}\Gamma_#6^{-1}{\Gamma_#6'}^{-1}\Gamma_#6\Gamma_#5
\end{equation}
\begin{equation}\label{#7-16}
\Gamma_#3'\Gamma_#3\Gamma_#2\Gamma_#3'\Gamma_#2^{-1}\Gamma_#3^{-1}{\Gamma_#3'}^{-1}=
\Gamma_#5^{-1}\Gamma_#6^{-1}\Gamma_#6'\Gamma_#6\Gamma_#5\Gamma_#4'\Gamma_#4\Gamma_#4'^{-1}\Gamma_#5^{-1}\Gamma_#6^{-1}{\Gamma_#6'}^{-1}\Gamma_#6\Gamma_#5
\end{equation}
\begin{equation}\label{#7-17}
[\Gamma_#3'\Gamma_#3\Gamma_#2\Gamma_#3^{-1}{\Gamma_#3'}^{-1},\Gamma_#5^{-1}\Gamma_#6^{-1}\Gamma_#6'\Gamma_#6\Gamma_#5]=e,
\end{equation}
\begin{equation}\label{#7-18}
[\Gamma_#4'\Gamma_#4\Gamma_#3'\Gamma_#3\Gamma_#2\Gamma_#3^{-1}{\Gamma_#3'}^{-1}\Gamma_#4^{-1}\Gamma_#4'^{-1},\Gamma_#5^{-1}\Gamma_#6^{-1}
{\Gamma_#6'}^{-1}\Gamma_#6\Gamma_#6'\Gamma_#6\Gamma_#5]=e.
\end{equation}
}

\newcommand{\uFivePointOuter}[7]{
	\todo{}
}

\newcommand{\uParasit}[3]{
\begin{equation}\label{#3}
\comm{#1}{#2} = \comm{#1'}{#2} = \comm{#1}{#2'} = \comm{#1'}{#2'} = e
\end{equation}}
\makeatletter
\newcommand{\uProjRel}[2]{
\begin{equation}\label{#1}
\ug{#2'}\ug{#2}\uProjRelChecknextarg}
\newcommand{\uProjRelChecknextarg}{\@ifnextchar\bgroup{\uProjRelGobblenextarg}{ = e
	\end{equation}}}
\newcommand{\uProjRelGobblenextarg}[1]{\ug{#1'}\ug{#1}\@ifnextchar\bgroup{\uProjRelGobblenextarg}{  = e. \end{equation}}}
\makeatother
\newcommand{\uChernSummary}[7]{
	In this case, the number of branch points is $#1$, the number of cusps is $#3$, the number of nodes is $#2$ and $\deg S=#4$.
	So the first Chern number is $ c_1^2 = #5\cdot 6! $, the second Chern number is $ c_2 = #6\cdot 6! $, and the signature is $ \chi = #7 6! $.
}
\begin{document}
		\definecolor{circ_col}{rgb}{0,0,0}
	
	\title{Fundamental group of Galois covers of degree $6$ surfaces\footnotetext{Email addresses: Meirav Amram (corresponding author): meiravt@sce.ac.il; Cheng Gong: cgong@suda.edu.cn; \\ Uriel Sinichkin: sinichkin@mail.tau.ac.il; Wan-Yuan Xu: wanyuanxu@fudan.edu.cn\\
			Sheng-Li Tan: sltan@math.ecnu.edu.cn; Michael Yoshpe: mikeking8890@gmail.com \\
			2010 Mathematics Subject Classification. 14D05, 14D06, 14H30, 14J10, 20F36. \\{\bf Key words}: Degeneration, generic projection, Galois cover, braid monodromy, fundamental group
	}}
	\author[1]{Meirav Amram}
	\author[2]{Cheng Gong}
	\author[3]{Uriel Sinichkin}
	\author[4]{Sheng-Li Tan}
	\author[5]{Wan-Yuan Xu}
	\author[6]{Michael Yoshpe}
	
	\affil[1, 6]{\small{Department of Mathematics, Shamoon College of Engineering, Ashdod, Israel}}
	\affil[2]{\small{Department of Mathematics, Soochow University, Suzhou 215006, Jiangsu, P. R. China}}
	\affil[3]{\small{School of Mathematical Sciences, Tel Aviv University, Tel Aviv, Israel}}
	\affil[4]{\small{School of Mathematical Sciences, Shanghai Key Labaratory of PMMP, East China Normal University, Shanghai 200241, P. R. China}}
	\affil[5]{\small{Department of Mathematics, Shanghai Normal University, Shanghai 200234, P. R. China}}
	
	\date{}
	
	\maketitle

	\abstract{
		In this paper we consider the Galois covers of algebraic surfaces of degree 6, with all associated planar degenerations.
		We compute the fundamental groups of those Galois covers, using their degeneration.
		We show that for 8 types of degenerations the fundamental group of the Galois cover is non-trivial and for 20 types it is trivial.
		Moreover, we compute the Chern numbers of all the surfaces with this type of degeneration and prove that the signatures of all their Galois covers are negative.
		We formulate a conjecture regarding the structure of the fundamental groups of the Galois covers based on our findings.
	}
	
	\newpage
	
	\tableofcontents

	\newpage
	
\section{Introduction}\label{outline}

Classification of algebraic surfaces has been one of the most significant problems guiding the development of algebraic geometry throughout its history.
In modern mathematics this classification has often been studied through its moduli space (see, for example, the works of Catanese \cite{C1,C2}).

Recent progress in this area is due to the study of invariants related to the Galois covers of such surfaces.
To an algebraic surface $ X $ of degree $ n $ embedded in a projective space $ \mbb{CP}^N $, we attach a Galois cover.
The Galois cover ${X}_{\text{Gal}}$ is the Zariski closure of the fibered
product, with respect to a generic projection to $\C\P^2$, of $ n $ copies of $X$, where the generalized diagonal is excluded. For a detailed definition of the Galois cover, see Section \ref{section:method}. 

Galois covers and their related fundamental groups were studied by Gieseker \cite{Gie}, Liedtke \cite{Li08}, Moishezon-Teicher \cite{MoTe87}, A.-Goldberg \cite{AG04}, A.-Teicher-Vishne \cite{ATV08}. Works on  surfaces of degrees $4$ and $5$ and their Galois covers were done by A.-Lehman-Shwartz-Teicher \cite{A-R-T} and A.-G.-Teicher-X. \cite{AGT}, respectively.
Lately, the fundamental group of Galois covers were computed for surfaces with Zappatic singularity by A.-G.-T.-Teicher-X. in \cite{ZAPP} and for embeddings of $ \mbb{CP}^1\times T $ of higher degrees by A.-T.-X.-Y. in \cite{CP1*T}.
In all those works, the Galois covers provide a way to construct important invariants of the base surfaces and contribute to their classification.
They are also primary tool used to construct an example of a simply-connected surface of positive index \cite{MoTe87}, thus disproving Bogomolov's watershed conjecture.

In this paper we study $29$ surfaces of degree $6$ and their Galois covers.
It is a continuation of the work on degenerations of degree $5$ \cite{AGT}. We study braid monodromy, the fundamental groups of the complements of the branch curves and of the Galois covers of degree $6$ surfaces with nice planar degenerations. Degenerations were studied by Calabri-Ciliberto-Flamini-Miranda in \cite{CCFM}, Ciliberto-Lopez-Miranda in \cite{CLM}, A.-G.-T.-Teicher-X. in \cite{ZAPP}, and in other papers such as \cite{pillow,agst,AG04}. The braid monodromy algorithm was presented by Moishezon-Teicher in \cite{BGT2,19}.
Fundamental groups of the complements of branch curves were studied for their own sake by Zariski in \cite{Z1} and by Auroux-Donaldson-Katzarkov-Yotov in \cite{ADKY}.

Below we describe briefly the process taken in this paper; for full details see Section \ref{section:method}.	
	The computation of the Galois cover and its invariants begins with the inspection of the branch curve $ S $.
	In general, it is difficult to describe $ S $ for a given projection $X$ onto  $\mathbb{CP}^2$.
	In order to overcome this obstacle, we degenerate $ X $ to a union of planes $ X_0 $, which then  degenerates the branch curve into an arrangement of lines $ S_0 $.
	Once we have this line arrangement, we can use the reverse process of regeneration, which is described by the so called "regeneration Lemmas" from \cite{BGT2}, to recover partial information regarding the branch curve $S$ of $X$.
	
	This suffices to compute the braid monodromy of $S$ by the Moishezon-Teicher algorithm.
	Having the braid monodromy of $S$, one can apply the van-Kampen Theorem \cite{vk} to calculate the  fundamental group $\pi_1(\mathbb{CP}^2-S)$ of the complement of $S$ in $\mathbb{CP}^2$.
	Then it is possible to calculate the fundamental group $\pi_1(\Xgal)$ of the Galois cover ${X}_{\text{Gal}}$ of $X$ by considering an exact sequence (see \cite{MoTe87}), 
	\begin{equation}\label{M-T}
	0 \rightarrow \pi_1(\Xgal) \rightarrow \widetilde{G} \rightarrow S_n \rightarrow 0,
	\end{equation}
	where $ \Ggal:=\pi_1(\mathbb{CP}^2-S)/\langle \Gamma_i^2 \rangle $.
	The group $\pi_1(\Xgal)$ is the kernel of the natural projection $\widetilde{G} \rightarrow S_n$.
	The above process computes the correct fundamental group $ \pi_1(\Xgal) $, because  Moishezon-Teicher showed in \cite{MoTe87} that when the complex structure of $X$ changes continuously, $\pi_1(\Xgal)$ does not change.

Other important invariants of surfaces are their Chern numbers $ c_1^2 $ and $ c_2 $ (see, for  example, Manetti's work \cite{Ma}). 
Both $ c_1^2(\Xgal) $ and $ c_2(\Xgal) $ can be computed from the braid monodromy (see \cite{MoTe87}).

We have considered surfaces that have degeneration with planar representation (see Definition \ref{def_planar_representation}), of which there are 29 types (see \cite[Appendix A]{6degree}).
The main result of this paper is that $\pi_1(\Gal{X})$ is not trivial in eight out of those 29 types, trivial in 20 other cases, and for one case the question whether $ \pi_1(\Gal{X}) $ is trivial or not remains open (see Table \ref{table_all_invariants} for details).
As an application, we compute the Chern numbers for all 29 Galois covers.
We find degenerations of surfaces that have the same Chern numbers, i.e., the related surfaces are in the same component of the moduli spaces.
Moreover, the signatures of all 29 types, that are computed using the Chern number of those Galois covers, are all negative.

This paper is organized as follows:  In Section \ref{section:method}, we explain important methods and present the fundamental and necessary background  that we use in this paper.
	In Section \ref{sec:main} we state and prove auxiliary lemmas we used in the calculations and provide the full computation of three of the cases we compute, to show an example of our methods.
	In Section \ref{sec:results} we present the results of all our calculations, including the computation of the Chern numbers and signatures of the Galois covers of the surfaces we consider.
	We include an appendix with a proof of the classification of the 29 cases of surfaces degenerating to degree 6 objects. Moreover, we give full calculations of the fundamental groups of the Galois covers.

	\paragraph{Acknowledgements:} The authors  are grateful  to   Guo  Zhiming for
	useful discussions about his work on the classification of the degree six surfaces.
	This research was supported by the ISF-NSFC joint research program (Grant No. 2452/17) .

\section{Preliminaries}\label{section:method}
	
	In this section we give the background needed for the computations that appear in later sections.
	Additionally, we introduce all the notations used in the paper.

	Let $ X $ be a projective algebraic surface embedded in projective space $ \mathbb{CP}^N $, for some $ N $.
	Consider a generic projection  $f:\mathbb{CP}^N\to\mathbb{CP}^2$.
	The restriction of $ f|_X $ is branched along a curve $ S\subseteq \mathbb{CP}^2 $.
	The branch curve $ S $ can tell a lot about $ X $, but it is difficult to describe it explicitly.
	To tackle this problem we consider degeneration of $ X $, defined as follows.

	\begin{Def}
		Let $\Delta$ be the unit disc,
		and let $X, Y$ be projective algebraic surfaces.
		Let $p: Y \rightarrow \mathbb{CP}^2$ and $p': X \rightarrow \mathbb{CP}^2$
		be projective embeddings.
		We say that $p'$ is a \emph{projective degeneration} of $p$
		if there exists a flat family $\pi: V \rightarrow \Delta$
		and an embedding $F:V\rightarrow \Delta \times \mathbb{CP}^2$,
		such that $F$ composed with the first projection is $\pi$,
		and:
		\begin{itemize}
			\item[(a)] $\pi^{-1}(0) \simeq X$;
			\item[(b)] there is a $t_0 \neq 0$ in $\Delta$ such that
			$\pi^{-1}(t_0) \simeq Y$;
			\item[(c)] the family $V-\pi^{-1}(0) \rightarrow \Delta-{0}$
			is smooth;
			\item[(d)] restricting to $\pi^{-1}(0)$, $F = {0}\times p'$
			under the identification of $\pi^{-1}(0)$ with $X$;
			\item[(e)] restricting to $\pi^{-1}(t_0)$, $F = {t_0}\times p$
			under the identification of $\pi^{-1}(t_0)$ with $Y$.
		\end{itemize}
	\end{Def}

	We construct a degeneration of $ X $ into a union of planes, as a sequence of \emph{partial degenerations} $X: =X_r \leadsto X_{r-1} \leadsto \cdots X_{r-i} \leadsto
	X_{r-(i+1)} \leadsto \cdots \leadsto X_0$.
	Consider generic projections $\pi^{(i)} : X_{i} \rightarrow \mathbb{CP}^2$ with the branch curve $S_i$
	for $0 \leq i \leq r$. Note that $S_{i-1}$ is a degeneration of
	$S_{i}$.
	
	Because $ X_0 $ is a union of planes, its projection $ S_0 $ is a line arrangement.
	Locally around each singular point, $ S_0 $ is defined by the multiplicity of this singularity.
	We call a singular point of multiplicity $ k $, a \emph{$ k $-point}.
	A $ 1 $-point always comes from the intersection of $ 2 $ planes in $ X_0 $.
	Similarly, a $ 2 $-point is always the projection of the intersection of $ 3 $ planes $ P_1,P_2,P_3 $ in $ X_0 $, with $ P_1 $ and $ P_3 $ intersecting only in the singular point (and not in a line).
	
	Next, for a $3$-point in $ S_0 $, two options can occur.
	Either the $3$-point is an intersection of three planes in $ X_0 $, with every pair of planes intersecting in a line, or it is the intersection of four planes $ P_1,\dots,P_4 $ in $ X_0 $, s.t. $ P_i $ and $ P_j $ intersect in a line only if $ |i-j|=1 $.
	We call the first option an \emph{inner $3$-point} and the second an \emph{outer $3$-point}.
	For singularities of higher multiplicities additional types of singularities can occur in $ S_0 $, but in the arrangements we consider in the current work, only 2 types can happen for each multiplicity.
	
	\begin{Def}\label{def:inner_and_outer}
	We call a $ k $-point that is the intersection of $ k+1 $ planes $ P_1,\dots,P_{k+1} $, s.t. $ P_i $ intersect $ P_j $ in a line if and only if $ |i-j|=1 $, an \emph{outer $ k $-point}.
	
	We call a $ k $-point that is the intersection of $ k $ planes $ P_1,\dots,P_{k} $, s.t. $ P_i $ intersects $ P_j $ in a line if $ |i-j|=1 $ and additionally $ P_1 $ intersects $ P_k $, an \emph{inner $ k $-point}.
	\end{Def}

	Note that those singularities were considered in \cite{CCFM}, where the inner $ k $-point is denoted by $ E_k $ and the outer $ k $-point is denoted by $ R_k $.

	\begin{eg}
		We explain the above notions on the degeneration using Figure \ref{fig_degeneration_example}.
		The figure is a schematic representation of $ X_0 $.
		The components of $ X_0 $ are represented by triangles and their lines (common edges).
		So $ P_1 $ and $ P_2 $ intersect in a line indexed $ 1 $, $ P_2 $ and $ P_3 $ intersect in a line indexed $ 2 $, $ P_1 $ and $ P_4 $ are disjoint, and so on.
		
		The vertices of the diagram that have inner edges connected to them are the singular points of $ X_0 $.
		Vertex $ 6 $ is a $ 1 $-point, vertices $ 2 $ and $ 3 $ are $ 2 $-points, vertex $ 4 $ is an inner $ 3 $-point and vertex $ 7 $ is an outer $ 4 $-point.
	\end{eg}

\begin{figure}[H]
	\begin{center}
		
		\definecolor{circ_col}{rgb}{0,0,0}
		\begin{tikzpicture}[x=1cm,y=1cm,scale=3]
		
		\draw [fill=circ_col] (0.6,1.03) node [anchor=west] {1} circle (1pt);
		\draw [fill=circ_col] (-0.86, -0.5) node [below] {2} circle (1pt);
		\draw [fill=circ_col] (0.86, -0.5) node [below] {3} circle (1pt);
		\draw [fill=circ_col] (0,0) node [anchor=east] {4} circle (1pt);
		\draw [fill=circ_col] (-0.86, 0.5) node [above] {5} circle (1pt);
		\draw [fill=circ_col] (0.86, 0.5) node [above] {6} circle (1pt);
		\draw [fill=circ_col] (0,1) node [above] {7} circle (1pt);
		
		\draw [-, line width = 1pt] (-0.86, -0.5) -- (0.86, -0.5) -- (0.86, 0.5)  -- (0.6,1.03)-- (0,1) -- (-0.86, 0.5) -- (-0.86, -0.5);
		
		\draw [-, line width = 1pt] (-0.86,-0.5) -- node [above] {4} (0, 1);
		\draw [-, line width = 1pt] (0,0) -- node [anchor=west] {3} (0, 1);
		\draw [-, line width = 1pt] (0.86,-0.5) -- node [anchor=west] {2} (0, 1);
		\draw [-, line width = 1pt] (0,0) -- node [above] {5} (-0.86, -0.5);
		\draw [-, line width = 1pt] (0,0) -- node [above] {6} (0.86, -0.5);
		\draw [-, line width = 1pt] (0.86,0.5) -- node [below] {1} (0,1);

		\draw (-0.625, 0.375) node {\footnotesize $ P_6 $};
		\draw (-0.25, 0.15) node {\footnotesize $ P_5 $};
		\draw (0, -0.3) node {\footnotesize $ P_4 $};
		\draw (0.25, 0.15) node {\footnotesize $ P_3 $};
		\draw (0.625, 0.375) node {\footnotesize $ P_2 $};
		\draw (0.5, 0.85) node {\footnotesize $ P_1 $};
		\end{tikzpicture}
	\end{center}
	\setlength{\abovecaptionskip}{-0.15cm}
	\caption{An example of degeneration into a union of planes.}\label{fig_degeneration_example}
\end{figure}
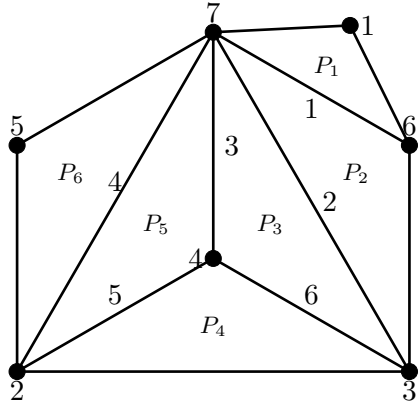
All the degenerations considered in the current work appear in Figure \ref{fig_all_cases_compact}.

	We note that $1$- and $2$-points were considered in \cite{pillow,Ogata,MoTe87}, $3$-points were considered in \cite{pillow}, $4$-points were considered in \cite{pillow,Ogata}, and $5$-points were considered in \cite{FT08}.
	The regeneration process shown below can be quite difficult for a large $k$, but work has been done for some specific values (see \cite{FT08,ZAPP,agst} for $5$-, $6$-, and $8$-points, respectively).	
	
	One of the principal tools we use is a reverse process of degeneration, called \emph{regeneration}.
	Using this tool, which was described in  \cite{BGT2} as regeneration Lemmas, we can recover $ S_i $ from $ S_{i-1} $.
	Applying it multiple times we can recover the original branch curve $ S $	from the line arrangement $ S_0 $.
	In the following diagram, %(Figure \ref{Degeneration-Regeneration}),
	we illustrate this process.

	\[\begin{CD}
	X\subseteq \mathbb{CP}^N  @>\text{degeneration}>> X_0\subseteq \mathbb{CP}^M \\
	@V\text{generic~ projection}VV                      @VV\text{generic~ projection}V \\
	S\subseteq \mathbb{CP}^2 @<\phantom{regeneration}<\text{regeneration}< S_0\subseteq \mathbb{CP}^2
	\end{CD}\]
	
	\vspace{0.2cm}	
	
	A line in $ S_0 $ regenerates either to a conic or a double line.
	The resulting components of the partial regeneration are tangent to each other.
	To get a transversal intersection of components, we regenerate further, and this gives us three cusps for each tangency point (see \cite{BGT2, 19} for more details). Therefore, the regenerated branch curve $ S $ is a cuspidal curve.

Using the notations of \cite{19} as well as the additional notation $ Z_{i\; i', j\; j'} $ introduced here, we can denote the braids related to $S$ as follows:
	\begin{enumerate}
		\item for a branch point, $Z_{j \;j'}$ is a counterclockwise
		half-twist of $j$ and $j'$ along a path below the real
		axis,
		\item for nodes, $Z^2_{i, j \;j'}=Z_{i\;j}^2 \cdot Z_{i\;j'}^2$ and $Z^2_{i \;i', j \;j'}=Z_{i'\;j'}^2 \cdot Z_{i\;j'}^2 \cdot Z_{i'\;j}^2 \cdot Z_{i\;j}^2$,
		\item for cusps, $Z^3_{i, j \;j'}=Z^3_{i\; j} \cdot (Z^3_{i\; j})^{Z_{j\;j'}} \cdot (Z^3_{i \; j})^{Z^{-1}_{j \; j'}}$.
	\end{enumerate}
	A singular point of $ S $ gives rise to braids $ (1)-(3) $ only locally.
	To get the final braids, one needs to perform a conjugation, which we denote as  $a^b = b^{-1}ab$.
	In several places we use the notation $ \bar{a} $ where $ a $ is a braid that means the same braid as $ a $ but above the real axis.
	
	Denote $ G:=\pi_1(\mathbb{CP}^2-S) $ and its standard generators as $\G_1, \G_{1}', \dots, \G_{2m}, \G_{2m}'$. By the van Kampen
	Theorem \cite{vk} we can get a presentation of $G$ by means of generators $\{\G_{j}, \G_{j}'\}$ and relations of the types:
	\begin{enumerate}
		\item for a branch point, $ Z_{j\; j'} $ corresponds to the relation $\G_{j} = \G_{j}'$,
		\item for nodes, $ Z_{i\; j}^2 $ corresponds to   $[\G_{i},\G_{j}]=\G_{i}\G_{j}\G_{i}^{-1}\G_{j}^{-1}=e$,
		\item for cusps, $ Z_{i\; j}^3 $ corresponds to  $\langle\G_{i},\G_{j}\rangle=\G_{i}\G_{j}\G_{i}\G_{j}^{-1}\G_{i}^{-1}\G_{j}^{-1}=e$.
	\end{enumerate}
	To get all the relations, we write the braids in a product and collect all the relations that correspond to the different factors.
	See \cite{BGT2, 19} for full treatment of the subject.
	
	To each list of relations we add the projective relation $\prod\limits_{j=m}^1 \G_{j}'\G_{j}=e$. Moreover, in some cases in the paper, we have parasitic intersections that induce commutative relations. These intersections come from lines in $X_0$ that do not intersect, but when projecting $X_0$ onto $\C\P^2$, they will intersect (see details in \cite{MoTe87}).

	\bigskip
	
	Our techniques also allow us to compute fundamental groups of Galois covers. We recall from \cite{MoTe87} that if
	$f : X\rightarrow \mathbb{CP}^2$ is a generic projection of degree $n$, then ${X}_{\text{Gal}}$, the Galois cover, is defined as follows:
	$${X}_{\text{Gal}}=\overline{(X \times_{\mathbb{CP}^{2}} \ldots \times_{\mathbb{CP}^{2}} X)-\triangle},$$
	where the product is taken $n$ times, and $\triangle$ is the diagonal.
	To apply a theorem of Moishezon-Teicher \cite{MoTe87}, we define $ \Ggal:=G/\langle \Gamma_i^2 \rangle $.
	Then, there is an exact sequence
	\begin{equation}\label{M-T}
	0 \rightarrow \pi_1(\Xgal) \rightarrow \widetilde{G} \rightarrow S_n \rightarrow 0,
	\end{equation}
	where the second map takes the generators $ \Gamma_i $ of $G$ to transpositions in the symmetric group $S_n$ according to the order of the lines in the degenerated surface.
	We thus obtain a presentation of the fundamental group of the Galois cover.
	We next simplify the relations to produce a canonical presentation that identifies with $\pi_1(\Xgal)$, using the theory of Coxeter covers of the symmetric
	groups. For more details see the proof of Theorem \ref{thm_computation_U_0_5_3_paper}.

	\section{Explicit fundamental group computation}\label{sec:main}
	
	In this section, we define the degenerations we are studying in this paper.
	The total number of such degenerations of degree 6 is 29, as shown in \cite[Appendix A]{6degree}, and as depicted in Figure \ref{fig_all_cases_compact}.
	We next compute the fundamental groups of their Galois covers.
	
	The full calculations for all the cases are quite long and would not fit in the current paper.
	We present here three examples of possible choices for $ X_0 $: $ U_{0,5,3},U_{0,6,1} $ and $ U_{4\cup 3,1 } $\footnote{The labeling of the possible choices for $ X_0 $ came from the combinatorial classification (see \cite[Appendix A]{6degree}).}.
	Those examples indicate the techniques we are using.
	The computations of $ \pi_1(\Gal{X}) $ for the 26 additional cases, as well as the computation of the Chern numbers of all 29 cases, are presented in \cite[Appendix B]{6degree}.

	\subsection{The degenerations of interest}
	
	Here we will explain the types of surface degenerations we have considered and the way to enumerate them.
	
	\begin{Def}\label{def_planar_representation}
		A degeneration of smooth toric surface $ X $ into a union of planes $ X_0 $ is said to have a \emph{planar representation} if:
		\begin{enumerate}
			\item
			No three planes in $ X_0 $ intersect in a line.
			
			\item
			There exists a simplicial complex with connected interior, embedded in $ \mathbb{R}^2 $, s.t. its two dimensional cells correspond bijectively to irreducible components of $ X_0 $ and this bijection preserves an incidence relation.
		\end{enumerate}
	\end{Def}
	
	This choice is not the most general that can be made, but most degenerations of surfaces that are classically of interest are of this form.
	It is also the choice made in \cite{AGT, A-R-T}.
	
	Note that degenerations that have planar representation of bounded degree can be enumerated recursively, because the planar simplicial complexes can be enumerated.
	Such enumeration for degree 6 is presented in \cite[Appendix A]{6degree}.
	
	The list of degenerations we get from this enumeration, and thus the degenerations we are considering in the paper, are presented in Figure \ref{fig_all_cases_compact}.
	
\begin{figure}[H]
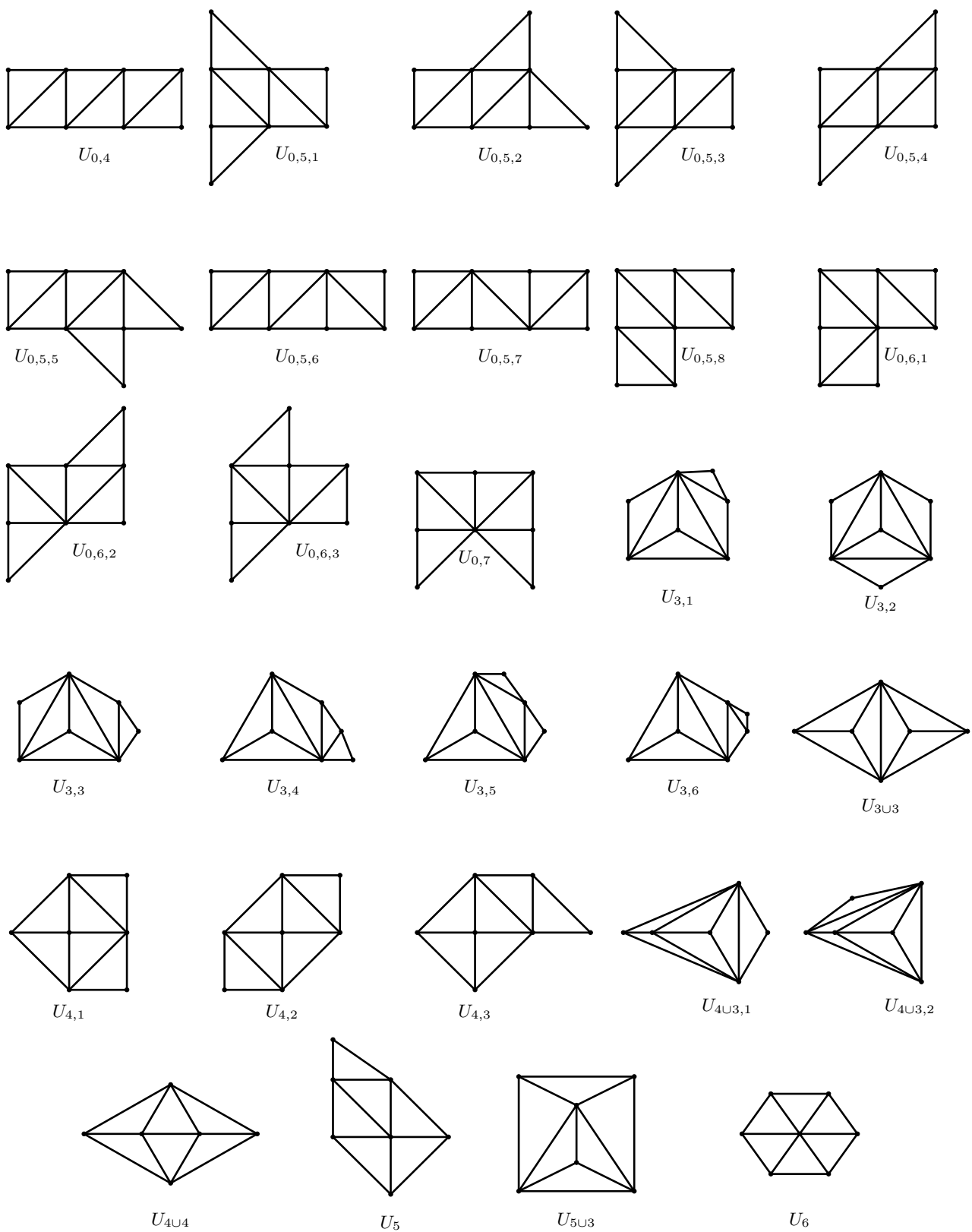


\begin{center}

\end{center}
		
\caption{The collection of all the degenerations we inspect in the current work.}\label{fig_all_cases_compact}
\end{figure}

\subsection{General setup and useful lemmas}
In this subsection we prove couple of results we will use later in the computation.

We will work with a dual graph $ T $ of $ X_0 $, which is defined when vertices of $ T $ are in bijection with the planes in $ X_0 $, and the vertices corresponding to the planes $ P_1 $ and $ P_2 $ are connected by an edge if $ P_1 $ and $ P_2 $ intersect in a line.

First, we will provide the principal tool we will use to prove that $ \pi_1(\Xgal) $ is not trivial in certain cases.

\begin{lemma}\label{not-trivial}
	Let $ \hat{G}:=\Ggal/\langle \Gamma_i = \Gamma_i' \rangle $.
	If $ T $ has a vertex of valency of at least 3 that is not part of a cycle in $ T $, then the kernel of the natural projection $ \hat{G}\to S_6 $ is not trivial.
	In particular, $ \pi_1(\Xgal) $ is not trivial in this case.
\end{lemma}

\begin{proof}
	Let $ i,j,k $ be three distinct edges connected to the vertex $ p $ of $ T $ that satisfy the properties required in the Lemma.
	Because $ X_0 $ has a planar representation (Definition \ref{def_planar_representation}), $ T $ is a planar graph.
	Thus, the three lines in $ X_0 $ corresponding to $ i,j,k $ cannot meet in a point.
	By \cite{RTV}, the kernel of the natural homomorphism $ \hat{G}\to S_6 $ contains the relation $ [\Gamma_i, \Gamma_j\Gamma_k\Gamma_j]=e $, so we are left to show that this relation is not trivial in $ \hat{G} $.
	
	To that end, we take the quotient of $ \hat{G} $ by the group normally generated by $ \{ \Gamma_l | l\notin \{i,j,k\} \} $.
	We now can inspect what can happen to the relations that arise from branch points, nodes, and cusps.
	Note that in every such relation, at most two instances of $ \Gamma_i, \Gamma_j, \Gamma_k $ can appear, because the lines $ i,j,k $ in $ X_0 $ do not pass through one point.
	
	\begin{description}
		\item[Branch points:]
		A branch point relation with only two of the generators $ \Gamma_i, \Gamma_j, \Gamma_k $ is either redundant or equivalent to a commutation relation between two of the generators.
		Because no two such generators can commute (otherwise it would contradict the existence of the natural homomorphism $ \widetilde{G}\to S_6 $, all those relations are trivial in $ \hat{G} $.
		
		\item[Nodes:]
		This case is identical to the previous - the only commutation relations with only two generators are redundant or a commutation between two generators that do not hold in $ \hat{G} $.
		
		\item[Cusps:]
		Cusps give rise to triple relations $ \langle w_1, w_2 \rangle  $ for some $ w_1,w_2\in \hat{G} $, which obviously cannot lead to the fork relation $ [\Gamma_i, \Gamma_j\Gamma_k\Gamma_j]=e $.
	\end{description}
\end{proof}

We next list some results that will allow us to show $ \Gamma_i=\Gamma_i' $ in $ \widetilde{G} $.
This will be the principal tool at our disposal to show that the group $ \pi_1(\Gal{X}) $ is trivial.

\begin{lemma}\label{2pt-equal}
	Let $ p $ be a $2$-point in $ X_0 $ with lines $ i $ and $ j $.
	Then $ \Gamma_i=\Gamma_i' $ in $ \widetilde{G} $ iff $ \Gamma_j=\Gamma_j' $ in $ \widetilde{G} $.
\end{lemma}

\begin{proof}
	\uTwoPoint{p}{i}{j}{lemma2pt_vert_p}
	
	Now, if $ \Gamma_i=\Gamma_i' $, relation \eqref{lemma2pt_vert_p-3} becomes $ \Gamma_j=\Gamma_j' $.
	On the other hand, if $ \Gamma_j=\Gamma_j' $, we can deduce from \eqref{lemma2pt_vert_p-2} and \eqref{lemma2pt_vert_p-3} that
	\begin{equation*}
	\Gamma_j\Gamma_i'\Gamma_j = \Gamma_i'\Gamma_j\Gamma_i'=\Gamma_i\Gamma_j\Gamma_i=\Gamma_j\Gamma_i\Gamma_j
	\end{equation*}
	and so we get $ \Gamma_i=\Gamma_i' $, as needed.
\end{proof}

\begin{lemma}\label{3pt-in-bigmid}
	Let $ p $ be an inner $3$-point (see Definition \ref{def:inner_and_outer}) in $ X_0 $ with lines $ i < j < k $ (see Figure \ref{fig_lemma_3pt_inner_bigmid}).
	If either $ \Gamma_j=\Gamma_j' $ or $ \Gamma_k=\Gamma_k' $ in $ \widetilde{G} $ then $ \Gamma_l=\Gamma_l' $ for all $ l\in \{i,j,k\} $.
	
	Moreover, $ \Gamma_i=\Gamma_i' $ is always present in this case.
\end{lemma}

\begin{figure}[H]
	\begin{center}
		
		\definecolor{circ_col}{rgb}{0,0,0}
		\begin{tikzpicture}[x=1cm,y=1cm,scale=2]
		\draw [fill=circ_col] (0,0) node [below] {$ p $} circle (1pt);
		
		\draw [-, line width = 1pt] (0,0) -- node [anchor=west] {$ i $} (0, 1);
		\draw [-, line width = 1pt] (0,0) -- node [above] {$ j $} (-0.86, -0.5);
		\draw [-, line width = 1pt] (0,0) -- node [above] {$ k $} (0.86, -0.5);
		
		\end{tikzpicture}
	\end{center}
	\setlength{\abovecaptionskip}{-0.15cm}
	\caption{The vertex $ p $ in Lemma \ref{3pt-in-bigmid}.}\label{fig_lemma_3pt_inner_bigmid}
\end{figure}
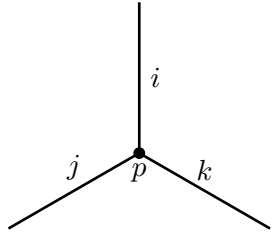

\begin{proof}
	\uThreePointInner{p}{i}{j}{k}{lemma3pt_inner_vert_p}
	
	First, we will use \eqref{lemma3pt_inner_vert_p-3}	 and \eqref{lemma3pt_inner_vert_p-4} to get that $ \Gamma_j=\Gamma_j' $ if and only if $ \Gamma_k=\Gamma_k' $.
	Indeed, if $ \Gamma_k=\Gamma_k' $ then we get that the right sides of those equations are equal, so by equating the left sides, we get $ \Gamma_j=\Gamma_j' $.
	Conversely, if $ \Gamma_j=\Gamma_j' $, we get that the left sides of \eqref{lemma3pt_inner_vert_p-3}	 and \eqref{lemma3pt_inner_vert_p-4} are equal, so by equating the right sides we get that $ \Gamma_k=\Gamma_k' $.
	
	For the second part of the statement, we equate the left sides of \eqref{lemma3pt_inner_vert_p-3} and \eqref{lemma3pt_inner_vert_p-5} and then use \eqref{lemma3pt_inner_vert_p-1} and \eqref{lemma3pt_inner_vert_p-2}, similarly to the proof of Lemma \ref{2pt-equal}.
\end{proof}

\begin{lemma}\label{3pt-out-smallmid}
	Let $ p $ be an outer $3$-point in $ X_0 $ with lines $ i < j < k $ (see Figure \ref{fig_lemma3pt_outer}).
	If either $ \Gamma_i=\Gamma_i' $ or $ \Gamma_j=\Gamma_j' $ in $ \widetilde{G} $ then $ \Gamma_l=\Gamma_l' $ for all $ l\in \{i,j,k\} $.
\end{lemma}

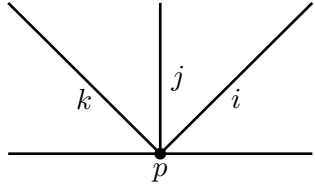
\begin{figure}[H]
	\begin{center}
		
		\definecolor{circ_col}{rgb}{0,0,0}
		\begin{tikzpicture}[x=1cm,y=1cm,scale=2]
		
		\draw [fill=circ_col] (0,0) node [below] {$ p $} circle (1pt);
		
		\draw [-, line width = 1pt] (0,0) -- (1,0);
		\draw [-, line width = 1pt] (0,0) -- (-1,0);
		
		\draw [-, line width = 1pt] (0,0) -- node [below]{$ i $} (1,1);
		\draw [-, line width = 1pt] (0,0) -- node [anchor=west]{$ j $} (0,1);
		\draw [-, line width = 1pt] (0,0) -- node [below]{$ k $} (-1,1);
		\end{tikzpicture}
		
	\end{center}
	\setlength{\abovecaptionskip}{-0.15cm}
	\caption{The vertex $ p $ in Lemma \ref{3pt-out-smallmid}.}\label{fig_lemma3pt_outer}
\end{figure}

\begin{proof}
	\uThreePointOuter{p}{i}{j}{k}{lemma3pt_outer_vert_p}
	
	If $ \Gamma_j=\Gamma_j' $ we get $ \Gamma_i=\Gamma_i' $ from \eqref{lemma3pt_outer_vert_p-3} and then $ \Gamma_k=\Gamma_k' $ from \eqref{lemma3pt_outer_vert_p-5}.
	If $ \Gamma_i=\Gamma_i' $, we get $ \Gamma_j=\Gamma_j' $ similarly to the proof of Lemma \ref{2pt-equal}, and then $ \Gamma_k=\Gamma_k' $ from \eqref{lemma3pt_outer_vert_p-5}.
\end{proof}

%+--------------------------------------+
%|                                      |
%|                CASE  U_{0,5,3}	    |
%|                                      |
%+--------------------------------------+
\subsection{$ X $ degenerates to $ U_{0,5,3} $}\label{section_computation_U_0_5_3_paper}

\begin{figure}[H]
	\begin{center}
		
		\definecolor{circ_col}{rgb}{0,0,0}
		\begin{tikzpicture}[x=1cm,y=1cm,scale=2]
		
		\draw [fill=circ_col] (0,-1) node [below] {1} circle (1pt);
		\draw [fill=circ_col] (0,0) node [anchor=east] {2} circle (1pt);
		\draw [fill=circ_col] (1,0) node [below] {3} circle (1pt);
		\draw [fill=circ_col] (2,0) node [below] {4} circle (1pt);
		\draw [fill=circ_col] (0,1) node [anchor=east] {5} circle (1pt);
		\draw [fill=circ_col] (1,1) node [above] {6} circle (1pt);
		\draw [fill=circ_col] (2,1) node [anchor=west] {7} circle (1pt);
		\draw [fill=circ_col] (0,2) node [above] {8} circle (1pt);
		
		\draw [-, line width = 1pt] (1,0) -- (2,0);
		\draw [-, line width = 1pt] (1,1) -- (0,2);
		\draw [-, line width = 1pt] (2,0) -- (2,1);
		
		\draw [-, line width = 1pt] (0,1) -- node [above]{1} (1,1);
		\draw [-, line width = 1pt] (0,0) -- node [above]{2} (1,1);
		\draw [-, line width = 1pt] (0,0) -- node [above]{3}  (1,0);
		\draw [-, line width = 1pt] (1,1) -- node [anchor=east]{4} (1,0);
		\draw [-, line width = 1pt] (1,0) -- node [above]{5} (2,1);
		
		\draw [-, line width = 1pt] (1,1) -- (2,1);
		\draw [-, line width = 1pt] (0,0) --  (0,1);
		\draw [-, line width = 1pt] (0,1) -- (0,2);
		\draw [-, line width = 1pt] (1,0) -- (0,-1);
		\draw [-, line width = 1pt] (0,0) -- (0,-1);
		\end{tikzpicture}
	\end{center}
	\setlength{\abovecaptionskip}{-0.15cm}
	\caption{The arrangement of planes $ U_{0,5,3} $.}\label{fig_computation_U_0_5_3}
\end{figure}
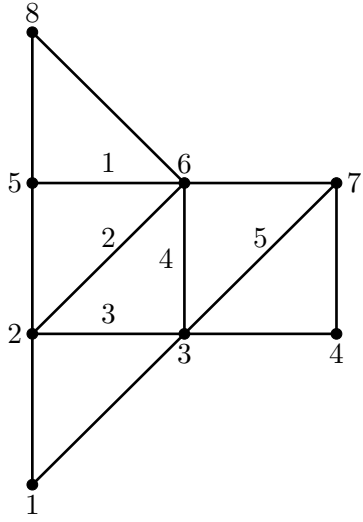

\begin{thm}\label{thm_computation_U_0_5_3_paper}
	If $ X $ degenerates to $ U_{0,5,3} $, then $\pi_1(\Xgal)$ is not trivial.
\end{thm}

\begin{proof}
	The branch curve $S_0$ in $\mathbb{CP}^2$ is an arrangement of five lines. We regenerate each vertex in turn and compute the group $G$.

	\uOnePoint{5}{1}{21-vert5}
	\uOnePoint{7}{5}{21-vert7}
	\uTwoPoint{2}{2}{3}{21-vert2}
	\uThreePointOuter{3}{3}{4}{5}{21-vert3}
	\uThreePointOuter{6}{1}{2}{4}{21-vert6}
	We also have the following parasitic and projective relations:
	\uParasit{1}{3}{21-parasit-1-3}
	\uParasit{1}{5}{21-parasit-1-5}
	\uParasit{2}{5}{21-parasit-2-5}
	\uProjRel{21-proj}{5}{4}{3}{2}{1}
	
	By Lemma \ref{3pt-out-smallmid} we get $\ug{2}=\ug{2'}$ and $\ug{4}=\ug{4'}$. Then by Lemma \ref{2pt-equal}, we get $\ug{3}=\ug{3'}$.
	
	Thus $\Ggal$ is generated by $\set{\ug{i} | i=1,\dots,5}$ modulo the following relations:
	\uGammaSq{21-simpl-4}{1}{2}{3}{4}{5}
	\ubegineq{21-simpl-5}
	\trip{1}{2} = \trip{2}{3} = \trip{2}{4}= \trip{3}{4} = \trip{4}{5}=e
	\uendeq
	\ubegineq{21-simpl-6}
	\comm{1}{3} = \comm{1}{4} = \comm{1}{5} = \comm{2}{5} = \comm{3}{5}=e.
	\uendeq
	By Lemma \ref{not-trivial}, $\pi_1(\Xgal)$ is not trivial.

\end{proof}

%+--------------------------------------+
%|                                      |
%|                CASE  U_{0,6,1}  	    |
%|                                      |
%+--------------------------------------+
\subsection{$ X $ degenerates to $ U_{0,6,1} $}\label{section_computation_U_0_6_1_paper}

\begin{figure}[H]
	\begin{center}
		
		\definecolor{circ_col}{rgb}{0,0,0}
		\begin{tikzpicture}[x=1cm,y=1cm,scale=2]
		
		\draw [fill=circ_col] (0,0) node [below] {1} circle (1pt);
		\draw [fill=circ_col] (1,0) node [below] {2} circle (1pt);
		\draw [fill=circ_col] (0,1) node [anchor=east] {3} circle (1pt);
		\draw [fill=circ_col] (1,1) node [anchor=north west] {4} circle (1pt);
		\draw [fill=circ_col] (2,1) node [below] {5} circle (1pt);
		\draw [fill=circ_col] (0,2) node [above] {6} circle (1pt);
		\draw [fill=circ_col] (1,2) node [above] {7} circle (1pt);
		\draw [fill=circ_col] (2,2) node [above] {8} circle (1pt);
		
		\draw [-, line width = 1pt] (0,0) -- (1,0) -- (1,1) -- (2,1) -- (2,2) -- (1,2) -- (0,2) -- (0,1) -- (0,0);
		
		\draw [-, line width = 1pt] (0,0) -- node [below]{1} (1,1);
		\draw [-, line width = 1pt] (0,1) -- node [below]{2} (1,1);
		\draw [-, line width = 1pt] (0,2) -- node [below]{3} (1,1);
		\draw [-, line width = 1pt] (1,1) -- node [anchor=west]{4} (1,2);
		\draw [-, line width = 1pt] (1,2) -- node [anchor=west]{5} (2,1);
		\end{tikzpicture}
		
	\end{center}
	\setlength{\abovecaptionskip}{-0.15cm}
	\caption{The arrangement of planes $ U_{0,6,1} $.}\label{fig_computation_U_0_6_1}
\end{figure}
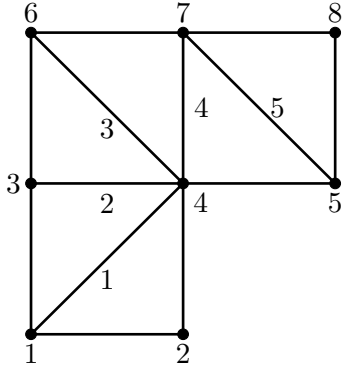

\begin{thm}\label{thm_computation_U_0_6_1_paper}
	If $ X $ degenerates to $ U_{0,6,1} $, then $\pi_1(\Xgal)$ is trivial.
\end{thm}

\begin{proof}
	The branch curve $S_0$ in $\mathbb{CP}^2$ is an arrangement of five lines. We regenerate each vertex in turn and compute the group $G$.

	\uOnePoint{1}{1}{14-vert1}
	\uOnePoint{3}{2}{14-vert3}
	\uOnePoint{5}{5}{14-vert5}
	\uOnePoint{6}{3}{14-vert6}
	\uTwoPoint{7}{4}{5}{14-vert7}
	\uFourPointOuter{4}{1}{2}{3}{4}{14-vert4}
	We also have the following parasitic and projective relations:
	\uParasit{1}{5}{14-parasit-1-5}
	\uParasit{2}{5}{14-parasit-2-5}
	\uParasit{3}{5}{14-parasit-3-5}
	\uProjRel{14-proj}{5}{4}{3}{2}{1}
	
	By Lemma \ref{2pt-equal}, we get $\ug{4}=\ug{4'}$.
	
	$\Ggal$ is thus generated by $\set{\ug{i} | i=1,\dots,5}$ with the following relations:
	\uGammaSq{14-gamma-sq}{1}{2}{3}{4}{5}
	\ubegineq{14-simpl-3}
	\trip{1}{2} = \trip{2}{3} = \trip{3}{4} = \trip{4}{5} = e,
	\uendeq
	\ubegineq{14-simpl-4}
	\comm{1}{3} = \comm{1}{4} = \comm{1}{5} = \comm{2}{4} = \comm{2}{5} = \comm{3}{5} = e.
	\uendeq
	Thus, $\Ggal \cong S_6$, therefore $\pi_1(\Xgal)$ is trivial.
	
\end{proof}

%+--------------------------------------+
%|                                      |
%|                CASE  U_{4 \cup 3,1}  |
%|                                      |
%+--------------------------------------+
\subsection{$ X $ degenerates to $ U_{4 \cup 3,1} $}\label{section_computation_U_4_cup_3_1_paper}

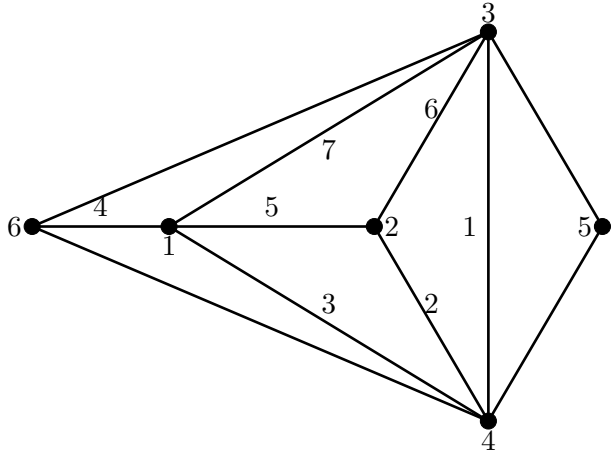
\begin{figure}[H]
	\begin{center}
		
		\definecolor{circ_col}{rgb}{0,0,0}
		\begin{tikzpicture}[x=1cm,y=1cm,scale=3]

		\draw [fill=circ_col] (-1.4, 0) node [below] {1} circle (1pt);
		\draw [fill=circ_col] (-0.5, 0) node [anchor=west] {2} circle (1pt);
		\draw [fill=circ_col] (0, 0.86) node [above] {3} circle (1pt);
		\draw [fill=circ_col] (0, -0.86) node [below] {4} circle (1pt);
		\draw [fill=circ_col] (0.5, 0) node [anchor=east] {5} circle (1pt);
		\draw [fill=circ_col] (-2, 0) node [anchor=east] {6} circle (1pt);
		
		\draw [-, line width = 1pt] (-2,0) -- (0, -0.86) -- (0.5, 0) -- (0, 0.86) -- (-2, 0);
		
		\draw [-, line width = 1pt] (-1.4,0) -- node [above] {5} (-0.5,0);
		\draw [-, line width = 1pt] (0, 0.86) -- node [above] {6} (-0.5,0);
		\draw [-, line width = 1pt] (0, -0.86) -- node [above] {2} (-0.5,0);
		\draw [-, line width = 1pt] (0, -0.86) -- node [anchor=east] {1} (0, 0.86);
		\draw [-, line width = 1pt] (-2,0) -- node [above] {4} (-1.4,0);
		\draw [-, line width = 1pt] (0, 0.86) -- node [below] {7} (-1.4,0);
		\draw [-, line width = 1pt] (0, -0.86) -- node [above] {3} (-1.4,0);
		\end{tikzpicture}
		
	\end{center}
	\setlength{\abovecaptionskip}{-0.15cm}
	\caption{The arrangement of planes $U_{4 \cup 3,1}$.}\label{fig_computation_U_4_cup_3_1}
\end{figure}

\begin{thm}\label{thm_computation_U_4_cup_3_1_paper}
	If $ X $ degenerates to $ U_{4 \cup 3,1} $, then $\pi_1(\Xgal)$ is trivial.	
\end{thm}

\begin{proof}
	The branch curve $S_0$ in $\mathbb{CP}^2$ is an arrangement of seven lines. We regenerate each vertex in turn and compute the group $G$.
	
	\uOnePoint{6}{4}{U4cup3,1-vert6}
	\uThreePointInner{2}{2}{5}{6}{U4cup3,1-vert2}
	\uThreePointOuter{3}{1}{6}{7}{U4cup3,1-vert3}
	\uThreePointOuter{4}{1}{2}{3}{U4cup3,1-vert4}
	\uFourPointInner{1}{3}{4}{5}{7}{U4cup3,1-vert1}
	We also have the following parasitic and projective relations:
	\uParasit{1}{4}{U4cup3,1-parasit-1-4}
	\uParasit{1}{5}{U4cup3,1-parasit-1-5}
	\uParasit{2}{4}{U4cup3,1-parasit-2-4}
	\uParasit{2}{7}{U4cup3,1-parasit-2-7}
	\uParasit{3}{6}{U4cup3,1-parasit-3-6}
	\uParasit{4}{6}{U4cup3,1-parasit-4-6}
	\uProjRel{U4cup3,1-proj}{7}{6}{5}{4}{3}{2}{1}
	
	Substituting \eqref{U4cup3,1-vert6} and equating  \eqref{U4cup3,1-vert1-9} and \eqref{U4cup3,1-vert1-10}, we get $\ug{5}=\ug{5'}$.
	By Lemma \ref{3pt-in-bigmid}, we have $\ug{2}=\ug{2'}$ and $\ug{6}=\ug{6'}$. Then by Lemma \ref{3pt-out-smallmid}, we get $\ug{1}=\ug{1'}, \ug{3}=\ug{3'}$ and $\ug{7}=\ug{7'}$.
	
	$\Ggal$ is thus generated by $\set{\ug{i} | i=1,\dots,7}$ with the following relations:
	\uGammaSq{U4cup3,1-gamma-sq}{1}{2}{3}{4}{5}{6}{7}
	%\ubegineq{U4cup3,1-simpl-3}
	\begin{align}	
	&&\trip{1}{2} = \trip{1}{6} = \trip{2}{3} = \trip{2}{5} = \trip{2}{6} = \trip{3}{4}= \trip{3}{5} = \\
	&&\trip{4}{7} = \trip{5}{6}= \trip{5}{7} = \trip{6}{7} =  e \nonumber
	\end{align}	
	%\uendeq
	\ubegineq{U4cup3,1-simpl-4}
	\begin{split}	
	\comm{1}{3} = \comm{1}{4} = \comm{1}{5} = \comm{1}{7} = \comm{2}{4} &= \\
	= \comm{2}{7} = \comm{3}{6} = \comm{3}{7} = \comm{4}{5} = \comm{4}{6} &= e
	\end{split}
	\uendeq
	\ubegineq{U4cup3,1-simpl-5}
	\ug{6}= \ug{2}\ug{5}\ug{2}
	\uendeq
	\ubegineq{U4cup3,1-simpl-6}
	\ug{3}\ug{4}\ug{3}=\ug{7}\ug{5}\ug{7}.
	\uendeq
	By Remark 2.8 in \cite{RTV}, we get also the relations
	\ubegineq{U4cup3,1-simpl-7}
	[\ug{1},\ug{2}\ug{6}\ug{2}]=[\ug{2},\ug{3}\ug{5}\ug{3}]=[\ug{6},\ug{7}\ug{5}\ug{7}] = e.
	\uendeq
	By Theorem 2.3 in \cite{RTV} we have $\Ggal \cong S_6$, therefore $\pi_1(\Xgal)$ is trivial.

\end{proof}

\section{Results}\label{sec:results}

In this Section we present our findings.
We have considered the following invariants of $ \Xgal $: its fundamental group $ \pi_1(\Xgal) $, its Chern numbers $ c_1^2, c_2 $, and its signature $ \chi:=\frac{1}{3}(c_1^2-2c_2) $.

\subsection{The fundamental group of the Galois cover}

We begin with the fundamental group, for which we were able to deduce, with the exception of a single case, whether it is trivial.

\begin{thm}
	The group $\pi_1(\Xgal)$ is not trivial for $U_{0,5,1}$, $U_{0,5,2}$, $U_{0,5,3}$, $U_{0,5,4}$, $U_{0,5,5}$, $U_{0,6,2}$, $U_{0,6,3}$, $U_{3,5}$. The group is trivial in all the other cases, except possibly $ U_{4,2} $.
\end{thm}

The case of $ U_{4,2} $ is especially challenging because there are no 1-points, there are no 3-valent vertices in the dual graph not attached to a cycle, and there is no inner 3-point.
Thus it is the only case for which the techniques presented in the current work are not able to deduce whether $ \pi_1(\Gal{X}) $ is trivial.
We believe this question can be interesting to those seeking to develop additional algebraic tools in the study of Coxeter groups.

Note that the technique used in \cite{A-R-T} to find the group $ \pi_1(\Gal{X}) $ explicitly is not applicable in higher degrees.
We can therefore formulate the following question.
\begin{Qst}
	What are the isomorphism classes of the groups $ \pi_1(\Gal{X}) $, where $ X $ degenerates to one of $U_{0,5,1}$, $U_{0,5,2}$, $U_{0,5,3}$, $U_{0,5,4}$, $U_{0,5,5}$, $U_{0,6,2}$, $U_{0,6,3}$, or $U_{3,5}$?
\end{Qst}

In all those cases we can show that $ \pi_1(\Gal{X}) $ is normally generated by 1 or 2 elements in $ \widetilde{G} $ (namely the fork relations).
Whenever this group is normally generated by a unique element, based results in smaller degrees, we believe all the conjugations of this element commute and thus formulate the following conjecture:
\begin{conjeceture}
	The group $ \pi_1(\Gal{X}) $ is free abelian group whenever $ X $ can be degenerated to one of $U_{0,5,1}$, $U_{0,5,2}$, $U_{0,5,3}$, $U_{0,5,5}$, $U_{0,6,2}$, $U_{0,6,3}$, or $U_{3,5}$.
\end{conjeceture}

\subsection{Chern numbers and signature}

The Chern numbers and the signature of the Galois cover are additional important topological  invariants.
To compute them we use the following proposition.

\begin{prop}\cite[Proposition\, 0.2]{MoTe87} \label{prop_chern}
	The Chern classes of $ \Xgal $ are as follows:
	\begin{enumerate}
		\item  $c_1^2(\Xgal)=\frac{n!}{4} (m-6)^2,$
		\item $c_2(\Xgal)=n! (3-m+\frac{d}{4}+\frac{\mu}{2}+\frac{\rho}{6}),$
		
		where $n=\deg f, m=\deg S$, $\mu=$ number of branch points in $S$, $d=$ number of nodes in $S$,
		and $\rho=$ number of cusps in $S$.
	\end{enumerate}
\end{prop}

We denote by $\chi(\Xgal)$ the signature of the Galois cover, and we compute it by the formula
$$\chi(\Xgal) = \frac{1}{3}({c_1^2}-2{c_2}).$$

Using Proposition \ref{prop_chern} and the computations appearing in \cite[Appendix B]{6degree}, we can obtain the Chern numbers of all the Galois covers we consider.

\begin{thm}
	The Chern numbers and signatures of the Galois covers of the surfaces appear in Table \ref{table_all_invariants}.
	
	Moreover, all the signatures of degree 6 surfaces with a degeneration that have planar representation are negative and have the form $ -\frac{a}{3}\cdot 6! $ where $ a\in \{1,\dots,7\} $.
\end{thm}

\begin{proof}
	This is an immediate application of Proposition \ref{prop_chern}.
	For the values of $ n, m, \mu, d$ and $ \rho $ that appear in this proposition, see \cite[Appendix B]{6degree}.
\end{proof}

\subsection{Summary of results}

We attach the following table that includes the invariant computed in the paper: Chern numbers, signature and information whether the fundamental group of the Galois cover is trivial or not.

\begin{table}[h]
	\begin{center}
	\begin{tabular}{| c | c | c | c | c |}
		\hline
		$ X_0 $ & $ c_1^2 $ & $ c_2 $ & $ \chi $ & $ \pi_1(\Xgal) $ \\ \hline \hline
		$U_{0,4}$ & $4 \cdot 6!$ & $4 \cdot 6!$ & $-\frac{4}{3} \cdot 6!$ & trivial \\ \hline
		$U_{0,5,1}, U_{0,5,2}, U_{0,5,3}, U_{0,6,2}, U_{0,6,3}$ &	$ 4 \cdot 6!$ & $4 \cdot 6!$& $ -\frac{4}{3} \cdot 6!$ & not trivial \\		\hline
		$U_{0,5,4}$ &	$ 4 \cdot 6!$ & $3 \cdot 6!$ & $ -\frac{2}{3} \cdot 6!$ & not trivial \\ \hline
		$U_{0,5,5}$ & $4 \cdot 6!$ & $ \frac{7}{2} \cdot 6!$ & $ - 6!$ & not trivial \\		\hline
		$U_{0,5,6}, U_{0,5,8}$ & $4 \cdot 6!$ & $ \frac{9}{2} \cdot 6!$ & $- \frac{5}{3} \cdot 6!$ & trivial \\	\hline
		$U_{0,5,7}, U_{0,6,1}$ & $4 \cdot 6!$ & $5 \cdot 6!$ & $- 2 \cdot 6!$ & trivial \\ \hline
		$U_{0,7}$ & $4 \cdot 6!$ & $\frac{11}{2} \cdot 6!$ & $- \frac{7}{3} \cdot 6!$ & trivial \\ \hline
		$U_{3,1}, U_{3,3}, U_{4,3}$ & $9 \cdot 6!$ & $\frac{13}{2} \cdot 6!$ & $- \frac{4}{3} \cdot 6!$ & trivial \\ \hline
		$U_{3,5}$ & $9 \cdot 6!$ & $\frac{13}{2} \cdot 6!$ & $- \frac{4}{3} \cdot 6!$ & not trivial \\ \hline
		$U_{3,2}, U_{5}$ & $9 \cdot 6!$ & $6 \cdot 6!$ & $- 6!$ & trivial \\ \hline
		$U_{3,4}$ & $9 \cdot 6!$ & $\frac{15}{2}\cdot 6!$ & $- 2 \cdot 6!$ & trivial \\ \hline
		$U_{3,6}$ & $9 \cdot 6!$ & $7 \cdot 6!$ & $- \frac{5}{3} \cdot 6!$ & trivial \\ \hline
		$U_{3 \cup 3}$ & $16 \cdot 6!$ & $11 \cdot 6!$ & $- 2 \cdot 6!$ & trivial \\ \hline
		$U_{4,1}$ & $9 \cdot 6!$ & $\frac{11}{2} \cdot 6!$ & $- \frac{2}{3} \cdot 6!$ & trivial \\ \hline
		$U_{4,2}$ & $9 \cdot 6!$ & $5 \cdot 6!$ & $- \frac{1}{3} \cdot 6!$ & open question \\ \hline
		$U_{4 \cup 3,1}, U_{5 \cup 3}$ & $16 \cdot 6!$ & $\frac{19}{2} \cdot 6!$ & $- 6!$ & trivial \\ \hline
		$U_{4 \cup 3,2}, U_{4 \cup 4}$ & $16 \cdot 6!$ & $9 \cdot 6!$ & $- \frac{2}{3} \cdot 6!$ & trivial \\ \hline
		$U_{6}$ & $9 \cdot 6!$ & $ 7 \cdot 6!$ & $- \frac{5}{3} \cdot 6!$ & trivial \\ \hline
	\end{tabular}
\end{center}
\caption{Table of invariants computed in the paper.}\label{table_all_invariants}
\end{table}

%+-------------------------------------------+
%|						                     |
%|    		BIBLIOGRAPHY		             |
%|						                     |
%+-------------------------------------------+
\newpage
\cleardoublepage
\phantomsection

\newpage

\appendix

\section{Classification of the cases}\label{sec:classification}

We are grateful to Guo Zhiming for proving the completeness of the following classification.

\begin{thm}\label{thm:classification}
	There are exactly 29 non-isomorphic sextic degenerations that have a planar representation (see Definition \ref{def_planar_representation}).
\end{thm}

%I  give a  change here.
\begin{proof}
	In order to classify all planar representations, it is equivalent to classify all combinations of six plane triangles with common edges.  First, we consider a special kind of n-point configuration, which does not change by gluing other triangles. We denote the n-point configuration with these kinds of n-points by $D_n$ (corresponding to degeneration of n-degree surface). Note that the smallest $D_n$ is $D_3$. First, we classify sextic degenerations by gluing $D_k$ with $6-k$ triangles.
	
	\begin{itemize}
		\item The largest  $D_n$ in the configuration is $D_6$. There is only one possible configuration (see Figure \ref{fig_cases_U_6}), which we denote by $ U_6 $.
		
		\begin{figure}[H]
			\begin{center}
				

				\caption{$U_6$.}\label{fig_cases_U_6}
			\end{center}
		\end{figure}
	\end{itemize}
	\begin{itemize}
		\item  The largest  $D_n$ in the configuration is $D_5$. In this case, there are only 2-points, except the 5-point, which cannot be changed. Except for the general type (the 5-point quintic  degeneration and a plane), only $D_3$ can be inserted. There is only one possible configuration in each case (see Figure \ref{fig_cases_U_5}), which we denote by $ U_5 $ and $ U_{5\cup 3} $.
		
		\begin{figure}[H]
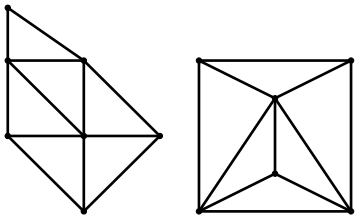

			\begin{center}

				%
				
				\caption{$U_5$ (left) and $U_{5\cup 3}$ (right).}\label{fig_cases_U_5}
			\end{center}
		\end{figure}
	\end{itemize}
	\begin{itemize}
		\item  The largest  $D_n$ in the configuration is $D_4$. In this case, there are two cases; one is the configuration of a degeneration of a union of $D_4$ with itself,  the other one is a union of  $D_4$ with $D_3$. There is only one possible degeneration of the first case (see Figure \ref{fig_cases_U_4_cup_4}), which we denote by $ U_{4\cup4} $.
		
		\begin{figure}[H]
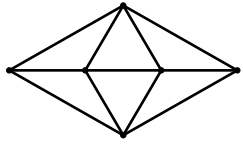

			
			\begin{center}
				
				%
				
			\end{center}
			\caption{$U_{4\cup 4}$. }\label{fig_cases_U_4_cup_4}
		\end{figure}

		The second case consists of exactly one $D_3$. So there are two possible configurations in this case (see Figure \ref{fig_cases_U_4_cup_3}), which we denote $ U_{4\cup3, 1}, U_{4\cup 3, 2} $.
		
		\begin{figure}[H]
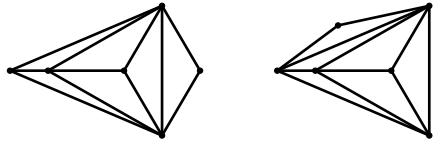

			
			\begin{center}

				%
				
			\end{center}
			\caption{$U_{4\cup 3, 1} $ (left) and $U_{4\cup 3, 2}$ (right). }\label{fig_cases_U_4_cup_3}
		\end{figure}
		
		There are also three possible degenerations of a general type (see Figure \ref{fig_cases_U_4}), which we denote as $ U_{4,1},U_{4,2} $ and $ U_{4,3} $.
		
		\begin{figure}[H]
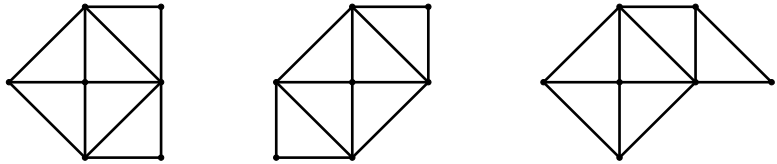

			
			\begin{center}
				
				%
				
			\end{center}
			\caption{From left to right: $U_{4,1}$, $ U_{4,2} $, and $U_{4,3} $. }\label{fig_cases_U_4}
		\end{figure}
	\end{itemize}
	\begin{itemize}
		\item  The largest  $D_n$ in the configuration is $D_3$. There is only one possible degeneration of $D_3$ with  $D_3$ (see Figure \ref{fig_cases_U_3_cup_3}), which we denote $ U_{3\cup 3} $.
		
		\begin{figure}[H]
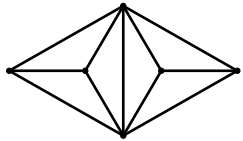

			\begin{center}
				
				%
	
			\end{center}
			\caption{$ U_{3 \cup 3}$. }\label{fig_cases_U_3_cup_3}
		\end{figure}
		
		Next we consider  $D_3$ and a plane. It is not difficult to see that there are six possible cases (see Figure \ref{fig_cases_U_3_deg}), which we denote $U_{3,i},i\in\{1,\dots,6\}$.
		
		\begin{figure}[H]
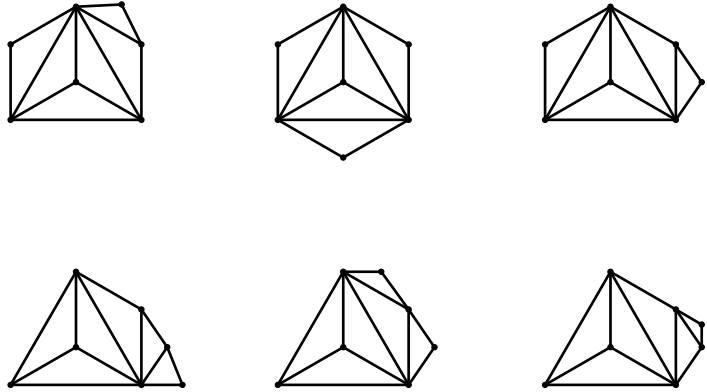

			\begin{center}

				%
				
			\caption{Top row (from left to right): $U_{3,1}$, $U_{3,2}$, and $U_{3,3}$. \\
				Bottom row (from left to right): $U_{3,4}$, $U_{3,5}$, and $U_{3,6}$.}\label{fig_cases_U_3_deg}
			\end{center}
		\end{figure}

	\end{itemize}
	\begin{itemize}
		\item Now we have the general type, which can be classified from the largest n-point.
		
		If $n=7$, there is only one possible configuration of a degeneration (see Figure \ref{fig_cases_U_0_7}), denoted $ U_{0, 7} $.
		
		\begin{figure}[H]
			
			\begin{center}
				
				%
			\end{center}
			\caption{$ U_{0,7} $.}\label{fig_cases_U_0_7}
		\end{figure}
		
		If $n=6$, there are three possible configurations of degenerations (see Figure \ref{fig_cases_U_0_6}), denoted as $ U_{0,6,1} $, $ U_{0,6,2} $, and $ U_{0,6,3} $.
		
		\begin{figure}[H]
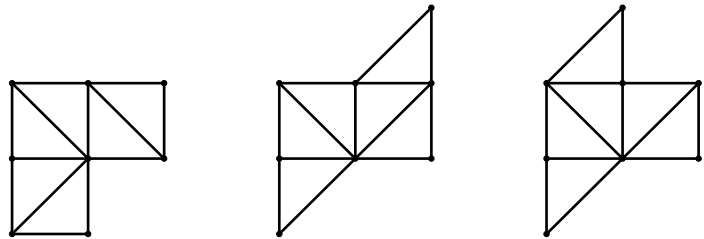

			\begin{center}
				
				%
			\end{center}
			\caption{From left to right: $ U_{0,6,1} $, $ U_{0,6,2} $, and $ U_{0,6,3} $. }\label{fig_cases_U_0_6}
		\end{figure}
		
		If $n=5$, there are eight possible configurations of degenerations (see Figures \ref{fig_cases_U_0_5_top}, \ref{fig_cases_U_0_5_middle}, and \ref{fig_cases_U_0_5_bottom}), denoted $ U_{0,5,i} $ for $ i=1,\dots,8 $ .
		
		\begin{figure}[H]
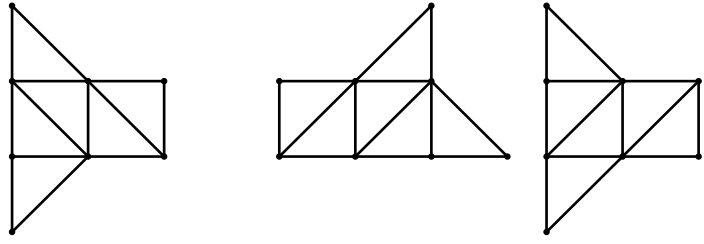

			
			\begin{center}
				
				\definecolor{circ_col}{rgb}{0,0,0}
				%
				
			\end{center}
			\caption{From left to right: $U_{0,5,1}$, $ U_{0,5,2}$, and $U_{0,5,3}$.}\label{fig_cases_U_0_5_top}
		\end{figure}
		
		\begin{figure}[H]
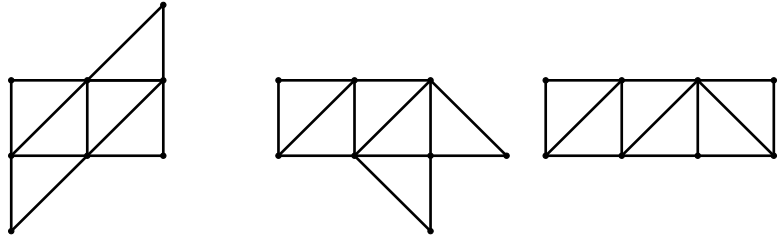

			
			\begin{center}

				%
				
			\end{center}
			\caption{From left to right: $U_{0,5,4}$, $ U_{0,5,5}$, and $U_{0,5,6}$. }\label{fig_cases_U_0_5_middle}
		\end{figure}
		
		\begin{figure}[H]
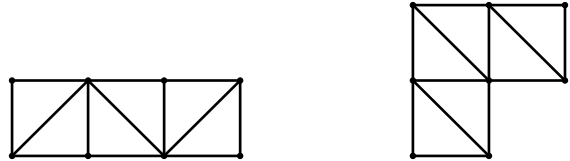

			\begin{center}
				
				%
				
			\end{center}
			\caption{$ U_{0,5,7} $ (left) and $ U_{0,5,8} $ (right).}\label{fig_cases_U_0_5_bottom}
		\end{figure}
		
		If $n=4$, there is only one  possible configuration of a degeneration (see Figure $ \ref{fig_cases_U_0_4} $), denoted $ U_{0,4} $.
		
		\begin{figure}[H]
			\begin{center}
				
				\definecolor{circ_col}{rgb}{0,0,0}
				%
			\end{center}
			\caption{$ U_{0,4} $.}\label{fig_cases_U_0_4}
		\end{figure}
	\end{itemize}
\end{proof}

\section{Full calculations}\label{sec_full_computations}

In this appendix we provide the full computations we have made.
See Section \ref{section:method} for the notations.

%+--------------------------------------+
%|                                      |
%|                CASE  U_0_4  		    |
%|                                      |
%+--------------------------------------+

\subsection{{$ X $ degenerates to $ U_{0,4}$}} \label{section_computation_U_0_4}

\begin{figure}[H]
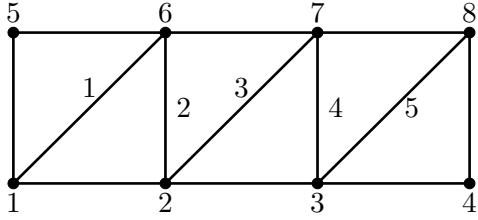

	\begin{center}
		
		\definecolor{circ_col}{rgb}{0,0,0}
		%
		
	\end{center}
	\setlength{\abovecaptionskip}{-0.15cm}
	\caption{The arrangement of planes $ U_{0,4}$.}\label{fig_computation_U_0_4}
\end{figure}

\begin{thm}\label{thm_computation_U_0_4}
	If $ X $ degenerates to $ U_{0,4} $, then $\pi_1(\Xgal)$ is trivial.
\end{thm}

\begin{proof}
	
	The branch curve $S_0$ in $\mathbb{CP}^2$ is an arrangement of five lines. We regenerate each vertex in turn and compute the group $G$.
	
	\uOnePoint{1}{1}{U04-vert1}
	\uOnePoint{8}{5}{U04-vert5}
	\uTwoPoint{2}{2}{3}{U04-vert2}
	\uTwoPoint{3}{4}{5}{U04-vert3}
	\uTwoPoint{6}{1}{2}{U04-vert6}
	\uTwoPoint{7}{3}{4}{U04-vert7}
	We also have the following parasitic and projective relations:
	\uParasit{1}{3}{U04-parasit-1-3}
	\uParasit{1}{4}{U04-parasit-1-4}
	\uParasit{1}{5}{U04-parasit-1-5}
	\uParasit{2}{4}{U04-parasit-2-4}
	\uParasit{2}{5}{U04-parasit-2-5}
	\uParasit{3}{5}{U04-parasit-3-5}
	\uProjRel{2-proj}{5}{4}{3}{2}{1}
	
	By Lemma \ref{2pt-equal} we get in $\Ggal$ that $\ug{2}=\ug{2'}, \ug{3'}  = \ug{3}$ and $\ug{4} = \ug{4'}$.
	
	Thus $\Ggal$ is generated by $\set{\ug{i} | i=1,\dots,5}$ modulo the following relations:
	\uGammaSq{U04-simpl-4}{1}{2}{3}{4}{5}
	\ubegineq{U04-simpl-5}
	\trip{1}{2} = \trip{2}{3} = \trip{3}{4} = \trip{4}{5}=e
	\uendeq
	\ubegineq{U04-simpl-6}
	\comm{1}{3} = \comm{1}{4} = \comm{1}{5} = \comm{2}{4} = \comm{2}{5} = \comm{3}{5}=e.
	\uendeq
	The generators of $\Ggal$ satisfy the same relations as the transpositions in $S_6$, therefore $\Ggal \cong S_6 $ and $\pi_1(\Xgal)$ is trivial.

\end{proof}

\uChernSummary{6}{24}{12}{10}{4}{4}{-\frac{4}{3}\cdot}

%+--------------------------------------+
%|
%|                CASE  U_{0,5,1}  		|
%|
%+--------------------------------------+

\subsection{{$ X $ degenerates to $ U_{0,5,1} $}}\label{section_computation_U_0_5_1}

\begin{figure}[H]
	\begin{center}
		
		\definecolor{circ_col}{rgb}{0,0,0}
		\begin{tikzpicture}[x=1cm,y=1cm,scale=2]
		
		\draw [fill=circ_col] (0,0) node [anchor=east] {1} circle (1pt);
		\draw [fill=circ_col] (0,1) node [anchor=east] {2} circle (1pt);
		\draw [fill=circ_col] (1,1) node [below] {3} circle (1pt);
		\draw [fill=circ_col] (2,1) node [below] {4} circle (1pt);
		\draw [fill=circ_col] (0,2) node [anchor=east] {5} circle (1pt);
		\draw [fill=circ_col] (1,2) node [above] {6} circle (1pt);
		\draw [fill=circ_col] (2,2) node [above] {7} circle (1pt);
		\draw [fill=circ_col] (0,3) node [above] {8} circle (1pt);
		
		\draw [-, line width = 1pt] (0,0) -- (1,1) -- (2,1) -- (2,2) -- (1,2) -- (0,3) -- (0,2) -- (0,1) -- (0,0);
		
		\draw [-, line width = 1pt] (0,1) -- node [below]{1} (1,1);
		\draw [-, line width = 1pt] (0,2) -- node [below]{2} (1,1);
		\draw [-, line width = 1pt] (0,2) -- node [above]{3} (1,2);
		\draw [-, line width = 1pt] (1,1) -- node [anchor=west]{4} (1,2);
		\draw [-, line width = 1pt] (1,2) -- node [anchor=west]{5} (2,1);
		\end{tikzpicture}
		
	\end{center}
	\setlength{\abovecaptionskip}{-0.15cm}
	\caption{The arrangement of planes $ U_{0,5,1} $.}\label{fig_computation_U_0_5_1}
\end{figure}

\begin{thm}\label{thm_computation_U_0_5_1}
	If $ X $ degenerates to $ U_{0,5,1} $, then $\pi_1(\Xgal)$ is not trivial.
\end{thm}

\begin{proof}
	The branch curve $S_0$ in $\mathbb{CP}^2$ is an arrangement of five lines. We regenerate each vertex in turn and compute the group $G$.

\uOnePoint{2}{1}{U051-vert2}
\uOnePoint{4}{5}{U051-vert4}
\uTwoPoint{5}{2}{3}{U051-vert5}
\uThreePointOuter{3}{1}{2}{4}{U051-vert3}
\uThreePointOuter{6}{3}{4}{5}{U051-vert6}
We also have the following parasitic and projective relations:
	\uParasit{1}{3}{U051-parasit-1-3}
	\uParasit{1}{5}{U051-parasit-1-5}
	\uParasit{2}{5}{U051-parasit-2-5}
	\uProjRel{U051-proj}{5}{4}{3}{2}{1}
	
By Lemma \ref{3pt-out-smallmid} we get $\ug{2}=\ug{2'}$ and $\ug{4}=\ug{4'}$. Then by Lemma \ref{2pt-equal}, $\ug{3}=\ug{3'}$.

Thus $\Ggal$ is generated by $\set{\ug{i} | i=1,\dots,5}$ modulo the following relations:
	\uGammaSq{U051-simpl-4}{1}{2}{3}{4}{5}
	\ubegineq{U051-simpl-5}
	\trip{1}{2} = \trip{2}{3} = \trip{2}{4}= \trip{3}{4} = \trip{4}{5}=e
	\uendeq
	\ubegineq{U051-simpl-6}
	\comm{1}{3} = \comm{1}{4} = \comm{1}{5} = \comm{2}{5} = \comm{3}{5}=e.
	\uendeq
	By Lemma \ref{not-trivial}, $\pi_1(\Xgal)$ is not trivial.

\end{proof}

\uChernSummary{7}{20}{15}{10}{4}{4}{-\frac{4}{3}\cdot}

%+--------------------------------------+
% |                                   	|
% |               CASE U_{0,5,2}       	|
% |                                     |
%+--------------------------------------+
\subsection{{$ X $ degenerates to $ U_{0,5,2}$}}\label{section_computation_U_0_5_2}

\begin{figure}[H]
	\begin{center}
		
		\definecolor{circ_col}{rgb}{0,0,0}
		\begin{tikzpicture}[x=1cm,y=1cm,scale=2]
		
		\draw [fill=circ_col] (0,0) node [below] {1} circle (1pt);
		\draw [fill=circ_col] (1,0) node [below] {2} circle (1pt);
		\draw [fill=circ_col] (2,0) node [below] {3} circle (1pt);
		\draw [fill=circ_col] (3,0) node [below] {4} circle (1pt);
		\draw [fill=circ_col] (0,1) node [above] {5} circle (1pt);
		\draw [fill=circ_col] (1,1) node [above] {6} circle (1pt);
		\draw [fill=circ_col] (2,1) node [anchor=south west] {7} circle (1pt);
		\draw [fill=circ_col] (2,2) node [above] {8} circle (1pt);

		\draw [-, line width = 1pt] (0,0) -- (1,0);
		\draw [-, line width = 1pt] (1,0) -- (2,0);
		\draw [-, line width = 1pt] (2,0) -- (3,0);
		\draw [-, line width = 1pt] (0,1) -- (1,1);
		
		\draw [-, line width = 1pt] (1,1) -- node [above]{3} (2,1);
		\draw [-, line width = 1pt] (0,0) -- node [above]{1} (1,1);
		\draw [-, line width = 1pt] (1,1) -- node [anchor=west]{2} (1,0);
		\draw [-, line width = 1pt] (1,0) -- node [above]{4} (2,1);
		\draw [-, line width = 1pt] (2,1) -- node [anchor=west]{5} (2,0);
		
		\draw [-, line width = 1pt] (0,0) --  (0,1);
		\draw [-, line width = 1pt] (2,1) -- (3,0);
		\draw [-, line width = 1pt] (1,1) -- (2,2);
		\draw [-, line width = 1pt] (2,1) -- (2,2);
		\end{tikzpicture}
		
	\end{center}
	\setlength{\abovecaptionskip}{-0.15cm}
	\caption{The arrangement of planes $ U_{0,5,2}$.}\label{fig_computation_U_0_5_2}
\end{figure}

\begin{thm}\label{thm_computation_U_0_5_2}
	If $ X $ degenerates to $ U_{0,5,2} $, then $\pi_1(\Xgal)$ is not trivial.
\end{thm}

\begin{proof}
	
	The branch curve $S_0$ in $\mathbb{CP}^2$ is an arrangement of five lines. We regenerate each vertex in turn and compute the group $G$.

	\uOnePoint{1}{1}{4-vert1}
	\uOnePoint{3}{5}{4-vert3}
	\uTwoPoint{2}{2}{4}{4-vert2}
	\uThreePointOuter{6}{1}{2}{3}{4-vert6}
	\uThreePointOuter{7}{3}{4}{5}{4-vert7}
	We also have the following parasitic and projective relations:
	\uParasit{1}{4}{4-parasit-1-4}
	\uParasit{1}{5}{4-parasit-1-5}
	\uParasit{2}{5}{4-parasit-2-5}
	\uProjRel{4-proj}{5}{4}{3}{2}{1}

By Lemma \ref{3pt-out-smallmid} we get $\ug{2}=\ug{2'}$ and $\ug{3}=\ug{3'}$. Then by Lemma \ref{2pt-equal}, $\ug{4}=\ug{4'}$.

Thus $\Ggal$ is generated by $\set{\ug{i} | i=1,\dots,5}$ modulo the following relations:
	\uGammaSq{4-simpl-4}{1}{2}{3}{4}{5}
	\ubegineq{4-simpl-5}
	\trip{1}{2} = \trip{2}{3} = \trip{2}{4}= \trip{3}{4} = \trip{4}{5}=e
	\uendeq
	\ubegineq{4-simpl-6}
	\comm{1}{3} = \comm{1}{4} = \comm{1}{5} = \comm{2}{5} = \comm{3}{5}=e.
	\uendeq
	By Lemma \ref{not-trivial}, $\pi_1(\Xgal)$ is not trivial.

\end{proof}

	\uChernSummary{7}{20}{15}{10}{4}{4}{-\frac{4}{3}\cdot}

%+--------------------------------------+
%|                                      |
%|                CASE  U_{0,5,3}	    |
%|                                      |
%+--------------------------------------+
\subsection{{$ X $ degenerates to $ U_{0,5,3} $}}\label{section_computation_U_0_5_3}

\begin{figure}[H]
	\begin{center}
		
		\definecolor{circ_col}{rgb}{0,0,0}
		\begin{tikzpicture}[x=1cm,y=1cm,scale=2]
		
		\draw [fill=circ_col] (0,-1) node [below] {1} circle (1pt);
		\draw [fill=circ_col] (0,0) node [anchor=east] {2} circle (1pt);
		\draw [fill=circ_col] (1,0) node [below] {3} circle (1pt);
		\draw [fill=circ_col] (2,0) node [below] {4} circle (1pt);
		\draw [fill=circ_col] (0,1) node [anchor=east] {5} circle (1pt);
		\draw [fill=circ_col] (1,1) node [above] {6} circle (1pt);
		\draw [fill=circ_col] (2,1) node [anchor=west] {7} circle (1pt);
		\draw [fill=circ_col] (0,2) node [above] {8} circle (1pt);
		
		\draw [-, line width = 1pt] (1,0) -- (2,0);
		\draw [-, line width = 1pt] (1,1) -- (0,2);
		\draw [-, line width = 1pt] (2,0) -- (2,1);
		
		\draw [-, line width = 1pt] (0,1) -- node [above]{1} (1,1);
		\draw [-, line width = 1pt] (0,0) -- node [above]{2} (1,1);
		\draw [-, line width = 1pt] (0,0) -- node [above]{3}  (1,0);
		\draw [-, line width = 1pt] (1,1) -- node [anchor=east]{4} (1,0);
		\draw [-, line width = 1pt] (1,0) -- node [above]{5} (2,1);
		
		\draw [-, line width = 1pt] (1,1) -- (2,1);
		\draw [-, line width = 1pt] (0,0) --  (0,1);
		\draw [-, line width = 1pt] (0,1) -- (0,2);
		\draw [-, line width = 1pt] (1,0) -- (0,-1);
		\draw [-, line width = 1pt] (0,0) -- (0,-1);
		\end{tikzpicture}
	\end{center}
	\setlength{\abovecaptionskip}{-0.15cm}
	\caption{The arrangement of planes $ U_{0,5,3} $.}\label{fig_computation_U_0_5_3_appendix}
\end{figure}

\begin{thm}\label{thm_computation_U_0_5_3}
	If $ X $ degenerates to $ U_{0,5,3} $, then $\pi_1(\Xgal)$ is not trivial.
\end{thm}

\begin{proof}
	See Theorem \ref{thm_computation_U_0_5_3_paper}.

\end{proof}

	\uChernSummary{7}{20}{15}{10}{4}{4}{-\frac{4}{3}\cdot}

%+--------------------------------------+
%|                                      |
%|                CASE  U_{0,5,4}  		|
%|                                      |
%+--------------------------------------+

\subsection{{$ X $ degenerates to $ U_{0,5,4} $}}\label{section_computation_U_0_5_4}

\begin{figure}[H]
	\begin{center}
		
		\definecolor{circ_col}{rgb}{0,0,0}
		\begin{tikzpicture}[x=1cm,y=1cm,scale=2]
		
		\draw [fill=circ_col] (0,1) node [anchor=east] {1} circle (1pt);
		\draw [fill=circ_col] (0,0) node [anchor=east] {2} circle (1pt);
		\draw [fill=circ_col] (1,1) node [below] {3} circle (1pt);
		\draw [fill=circ_col] (2,1) node [below] {4} circle (1pt);
		\draw [fill=circ_col] (0,2) node [anchor=east] {5} circle (1pt);
		\draw [fill=circ_col] (1,2) node [above] {6} circle (1pt);
		\draw [fill=circ_col] (2,3) node [above] {7} circle (1pt);
		\draw [fill=circ_col] (2,2) node [anchor=west] {8} circle (1pt);

		\draw [-, line width = 1pt] (0,0) -- (1,1) -- (2,1) -- (2,2) -- (1,2) -- (2,3) -- (2,2);
		\draw [-, line width = 1pt] (0,2) -- (1,2);
		\draw [-, line width = 1pt] (0,2) -- (0,1) -- (0,0);
		
		\draw [-, line width = 1pt] (0,1) -- node [above]{1} (1,2);
		\draw [-, line width = 1pt] (0,1) -- node [above]{2} (1,1);
		\draw [-, line width = 1pt] (1,1) -- node [anchor=west]{3} (1,2);
		\draw [-, line width = 1pt] (1,2) -- node [above]{4} (2,2);
		\draw [-, line width = 1pt] (1,1) -- node [anchor=west]{5} (2,2);
		\end{tikzpicture}
		
	\end{center}
	\setlength{\abovecaptionskip}{-0.15cm}
	\caption{The arrangement of planes $ U_{0,5,4} $.}\label{fig_computation_U_0_5_4}
\end{figure}

\begin{thm}\label{thm_computation_U_0_5_4}
	If $ X $ degenerates to $ U_{0,5,4} $, then $\pi_1(\Xgal)$ is not trivial.
\end{thm}

\begin{proof}
	The branch curve $S_0$ in $\mathbb{CP}^2$ is an arrangement of five lines. We regenerate each vertex in turn and compute the group $G$.

	\uTwoPoint{1}{1}{2}{20-vert1}
     \uTwoPoint{8}{4}{5}{20-vert8}
	\uThreePointOuter{3}{2}{3}{5}{20-vert3}
	\uThreePointOuter{6}{1}{3}{4}{20-vert6}
	We also have the following parasitic and projective relations:
	\uParasit{1}{5}{20-parasit-1-5}
	\uParasit{2}{4}{20-parasit-2-4}
	\uProjRel{20-proj}{5}{4}{3}{2}{1}

We compute the quotient $\Ggal/\langle{\Gamma_i=\Gamma_i^{'}} \rangle$:
	\uGammaSq{20-quot-1}{1}{2}{3}{4}{5}
	\ubegineq{20-quot-2}
	\trip{1}{2} = \trip{1}{3} = \trip{2}{3} = \trip{3}{4} =  \trip{3}{5} = \trip{4}{5} = e
	\uendeq
	\ubegineq{20-quot-3}
	\comm{1}{4} = \comm{1}{5} = \comm{2}{4} = \comm{2}{5} = e.
	\uendeq
By Lemma \ref{not-trivial}, $\pi_1(\Xgal)$ is not trivial.

\end{proof}

\uChernSummary{6}{16}{18}{10}{4}{3}{-\frac{2}{3}\cdot}

%+--------------------------------------+
% |                                	    |
% |            CASE U_{0,5,5}           |
% |                                     |
%+--------------------------------------+
\subsection{{$ X $ degenerates to $ U_{0,5,5} $}}\label{section_computation_U_0_5_5}

\begin{figure}[H]
	\begin{center}
		
		\definecolor{circ_col}{rgb}{0,0,0}
		\begin{tikzpicture}[x=1cm,y=1cm,scale=2]
		
		\draw [fill=circ_col] (0,0) node [below] {1} circle (1pt);
		\draw [fill=circ_col] (1,0) node [below] {2} circle (1pt);
		\draw [fill=circ_col] (2,0) node [anchor=north west] {4} circle (1pt);
		\draw [fill=circ_col] (3,0) node [below] {5} circle (1pt);
		\draw [fill=circ_col] (0,1) node [above] {6} circle (1pt);
		\draw [fill=circ_col] (1,1) node [above] {7} circle (1pt);
		\draw [fill=circ_col] (2,1) node [above] {8} circle (1pt);
		\draw [fill=circ_col] (2,-1) node [below] {3} circle (1pt);
		
		\draw [-, line width = 1pt] (0,0) -- (1,0);
		\draw [-, line width = 1pt] (1,1) -- (2,1);
		\draw [-, line width = 1pt] (2,0) -- (3,0);
		\draw [-, line width = 1pt] (0,1) -- (1,1);

		\draw [-, line width = 1pt] (1,0) -- node [above]{4} (2,0);
		\draw [-, line width = 1pt] (0,0) -- node [above]{1} (1,1);
		\draw [-, line width = 1pt] (1,1) -- node [anchor=east]{2} (1,0);
		\draw [-, line width = 1pt] (1,0) -- node [above]{3} (2,1);
		\draw [-, line width = 1pt] (2,1) -- node [anchor=east]{5} (2,0);
		
		\draw [-, line width = 1pt] (0,0) --  (0,1);
		\draw [-, line width = 1pt] (2,1) -- (3,0);
		\draw [-, line width = 1pt] (1,0) -- (2,-1);
		\draw [-, line width = 1pt] (2,0) -- (2,-1);
		\end{tikzpicture}
		
	\end{center}
	\setlength{\abovecaptionskip}{-0.15cm}
	\caption{The arrangement of planes $ U_{0,5,5} $.}\label{fig_computation_U_0_5_5}
\end{figure}

\begin{thm}\label{thm_computation_U_0_5_5}
	If $ X $ degenerates to $ U_{0,5,5} $, then $\pi_1(\Xgal)$ is not trivial.
\end{thm}

\begin{proof}
	The branch curve $S_0$ in $\mathbb{CP}^2$ is an arrangement of five lines. We regenerate each vertex in turn and compute the group $G$.

	\uOnePoint{1}{1}{3-vert1}
	\uTwoPoint{4}{4}{5}{3-vert4}
	\uTwoPoint{7}{1}{2}{3-vert7}
	\uTwoPoint{8}{3}{5}{3-vert8}
    \uThreePointOuter{2}{2}{3}{4}{3-vert2}
	We also have the following parasitic and projective relations:
	\uParasit{1}{3}{3-parasit-1-3}
	\uParasit{1}{4}{3-parasit-1-4}
	\uParasit{1}{5}{3-parasit-1-5}
	\uParasit{2}{5}{3-parasit-2-5}
	\uProjRel{3-proj}{5}{4}{3}{2}{1}

By Lemma \ref{2pt-equal}, $\ug{2}=\ug{2'}$. Then by Lemma \ref{3pt-out-smallmid} we get $\ug{3}=\ug{3'}$ and $\ug{4}=\ug{4'}$. Again by Lemma \ref{2pt-equal}, we get $\ug{5}=\ug{5'}$.

Thus $\Ggal$ is generated by $\set{\ug{i} | i=1,\dots,5}$ modulo the following relations:
	\uGammaSq{3-simpl-4}{1}{2}{3}{4}{5}
	\ubegineq{3-simpl-5}
	\trip{1}{2} = \trip{2}{3} = \trip{3}{4}= \trip{3}{5} = \trip{4}{5}=e
	\uendeq
	\ubegineq{3-simpl-6}
	\comm{1}{3} = \comm{1}{4} = \comm{1}{5} = \comm{2}{4} = \comm{2}{5}=e.
	\uendeq
	By Lemma \ref{not-trivial}, $\pi_1(\Xgal)$ is not trivial.

\end{proof}

	\uChernSummary{6}{20}{15}{10}{4}{\frac{7}{2}}{-}

%+--------------------------------------+
%|                                      |
%|             CASE U_0_5_6             |
%|                                      |
%+--------------------------------------+

\subsection{{$ X $ degenerates to $U_{0,5,6}$}}\label{section_computation_U_0_5_6}

\begin{figure}[H]
	\begin{center}
		
		\definecolor{circ_col}{rgb}{0,0,0}
		\begin{tikzpicture}[x=1cm,y=1cm,scale=2]
		
		\draw [fill=circ_col] (0,0) node [below] {1} circle (1pt);
		\draw [fill=circ_col] (1,0) node [below] {2} circle (1pt);
		\draw [fill=circ_col] (2,0) node [below] {3} circle (1pt);
		\draw [fill=circ_col] (3,0) node [below] {4} circle (1pt);
		\draw [fill=circ_col] (0,1) node [above] {5} circle (1pt);
		\draw [fill=circ_col] (1,1) node [above] {6} circle (1pt);
		\draw [fill=circ_col] (2,1) node [above] {7} circle (1pt);
		\draw [fill=circ_col] (3,1) node [above] {8} circle (1pt);

		\draw [-, line width = 1pt] (0,0) -- (1,0);
		\draw [-, line width = 1pt] (1,0) -- (2,0);
		\draw [-, line width = 1pt] (2,0) -- (3,0);
		\draw [-, line width = 1pt] (0,1) -- (1,1);
		
		\draw [-, line width = 1pt] (0,0) -- node [above]{1} (1,1);
		\draw [-, line width = 1pt] (1,1) -- node [anchor=west]{2} (1,0);
		\draw [-, line width = 1pt] (1,0) -- node [above]{3} (2,1);
		\draw [-, line width = 1pt] (2,0) -- node [anchor=west]{4} (2,1);
		\draw [-, line width = 1pt] (2,1) -- node [anchor=west]{5} (3,0);
		
		\draw [-, line width = 1pt] (0,0) --  (0,1);
		\draw [-, line width = 1pt] (2,1) -- (1,1);
		\draw [-, line width = 1pt] (2,1) -- (3,1);
		\draw [-, line width = 1pt] (3,0) -- (3,1);
		\end{tikzpicture}
		
	\end{center}
	\setlength{\abovecaptionskip}{-0.15cm}
	\caption{The arrangement of planes $U_{0,5,6}$.}\label{fig_computation_U_0_5_6}
\end{figure}

\begin{thm}\label{thm_computation_U_0_5_6}	
	If $ X $ degenerates to $ U_{0,5,6} $, then $\pi_1(\Xgal)$ is trivial.
\end{thm}

\begin{proof}
	
	The branch curve $S_0$ in $\mathbb{CP}^2$ is an arrangement of five lines. We regenerate each vertex in turn and compute the group $G$.

	\uOnePoint{1}{1}{1a-vert1}
     \uOnePoint{3}{4}{1a-vert3}
	\uOnePoint{4}{5}{1a-vert4}
	\uTwoPoint{2}{2}{3}{1a-vert2}
	\uTwoPoint{6}{1}{2}{1a-vert6}
	\uThreePointOuter{7}{3}{4}{5}{1a-vert7}
	We also have the following parasitic and projective relations:
	\uParasit{1}{3}{1a-parasit-1-3}
	\uParasit{1}{4}{1a-parasit-1-4}
	\uParasit{1}{5}{1a-parasit-1-5}
	\uParasit{2}{4}{1a-parasit-2-4}
	\uParasit{2}{5}{1a-parasit-2-5}
	\uProjRel{1a-proj}{5}{4}{3}{2}{1}

By Lemma \ref{2pt-equal}, we get $\ug{2}=\ug{2'}$ and $\ug{3}=\ug{3'}$.

$\Ggal$ is thus generated by $\set{\ug{i} | i=1,\dots,5}$ with the following relations:
	\uGammaSq{1a-gamma-sq}{1}{2}{3}{4}{5}
	\ubegineq{1a-simpl-3}
	\trip{1}{2} = \trip{2}{3} = \trip{3}{4} = \trip{4}{5} = e
	\uendeq
	\ubegineq{1a-simpl-4}
	\comm{1}{3} = \comm{1}{4} = \comm{1}{5} = \comm{2}{4} = \comm{2}{5} = \comm{3}{5} = e.
	\uendeq
Thus, $\Ggal \cong S_6$, therefore $\pi_1(\Xgal)$ is trivial.

\end{proof}

\uChernSummary{7}{24}{12}{10}{4}{\frac{9}{2}}{-\frac{5}{3}\cdot}

%+--------------------------------------+
%|                                      |
%|               CASE U_0_5_7           |
%|                                      |
%+--------------------------------------+

\subsection{{$ X $ degenerates to $U_{0,5,7}$}}\label{section_computation_U_0_5_7}

\begin{figure}[H]
	\begin{center}
		
		\definecolor{circ_col}{rgb}{0,0,0}
		\begin{tikzpicture}[x=1cm,y=1cm,scale=2]
		
		\draw [fill=circ_col] (0,0) node [below] {1} circle (1pt);
		\draw [fill=circ_col] (1,0) node [below] {2} circle (1pt);
		\draw [fill=circ_col] (2,0) node [below] {3} circle (1pt);
		\draw [fill=circ_col] (3,0) node [below] {4} circle (1pt);
		\draw [fill=circ_col] (0,1) node [above] {5} circle (1pt);
		\draw [fill=circ_col] (1,1) node [above] {6} circle (1pt);
		\draw [fill=circ_col] (2,1) node [above] {7} circle (1pt);
		\draw [fill=circ_col] (3,1) node [above] {8} circle (1pt);

		\draw [-, line width = 1pt] (0,0) -- (1,0);
		\draw [-, line width = 1pt] (1,0) -- (2,0);
		\draw [-, line width = 1pt] (2,0) -- (3,0);
		\draw [-, line width = 1pt] (0,1) -- (1,1);
		
		\draw [-, line width = 1pt] (0,0) -- node [above]{1} (1,1);
		\draw [-, line width = 1pt] (1,1) -- node [anchor=west]{2} (1,0);
		\draw [-, line width = 1pt] (1,1) -- node [above]{3} (2,0);
		\draw [-, line width = 1pt] (2,0) -- node [anchor=west]{4} (2,1);
		\draw [-, line width = 1pt] (2,0) -- node [anchor=west]{5} (3,1);
		
		\draw [-, line width = 1pt] (0,0) --  (0,1);
		\draw [-, line width = 1pt] (2,1) -- (1,1);
		\draw [-, line width = 1pt] (2,1) -- (3,1);
		\draw [-, line width = 1pt] (3,0) -- (3,1);
		\end{tikzpicture}
		
	\end{center}
	\setlength{\abovecaptionskip}{-0.15cm}
	\caption{The arrangement of planes $ U_{0,5,7}$.}\label{fig_computation_U_0_5_7}
\end{figure}

\begin{thm}\label{thm_computation_U_0_5_7}
	If $ X $ degenerates to $ U_{0,5,7} $, then $\pi_1(\Xgal)$ is trivial.
\end{thm}

\begin{proof}
	
	The branch curve $S_0$ in $\mathbb{CP}^2$ is an arrangement of five lines. We regenerate each vertex in turn and compute the group $G$.

	\uOnePoint{1}{1}{1d-vert1}
	\uOnePoint{2}{2}{1d-vert2}
	\uOnePoint{7}{4}{1d-vert7}
	\uOnePoint{8}{5}{1d-vert8}
	\uThreePointOuter{3}{3}{4}{5}{1d-vert3}
	\uThreePointOuter{6}{1}{2}{3}{1d-vert6}
	We also have the following parasitic and projective relations:
	\uParasit{1}{4}{1d-parasit-1-4}
	\uParasit{1}{5}{1d-parasit-1-5}
	\uParasit{2}{4}{1d-parasit-2-4}
	\uParasit{2}{5}{1d-parasit-2-5}
	\uProjRel{1d-proj}{5}{4}{3}{2}{1}

	By Lemma \ref{3pt-out-smallmid}, $\ug{3} = \ug{3'}$.
	
$\Ggal$ is thus generated by $\set{\ug{i} | i=1,\dots,5}$ with the following relations:
	\uGammaSq{1d-gamma-sq}{1}{2}{3}{4}{5}
	\ubegineq{1d-simpl-3}
	\trip{1}{2} = \trip{2}{3} = \trip{3}{4} = \trip{4}{5} = e
	\uendeq
	\ubegineq{1d-simpl-4}
	\comm{1}{3} = \comm{1}{4} = \comm{1}{5} = \comm{2}{4} = \comm{2}{5} = \comm{3}{5} = e.
	\uendeq
Thus, $\Ggal \cong S_6$, therefore $\pi_1(\Xgal)$ is trivial.

\end{proof}

\uChernSummary{8}{24}{12}{10}{4}{5}{-2\cdot}

%+--------------------------------------+
%|                                      |
%|                CASE  U_{0,5,8}	    |
%|                                      |
%+--------------------------------------+

\subsection{{$ X $ degenerates to $ U_{0,5,8} $}} \label{section_computation_U_0_5_8}

\begin{figure}[H]
	\begin{center}
		
		\definecolor{circ_col}{rgb}{0,0,0}
		\begin{tikzpicture}[x=1cm,y=1cm,scale=2]
		
		\draw [fill=circ_col] (0,0) node [below] {1} circle (1pt);
		\draw [fill=circ_col] (1,0) node [below] {2} circle (1pt);
		\draw [fill=circ_col] (0,1) node [anchor=east] {3} circle (1pt);
		\draw [fill=circ_col] (1,1) node [anchor=north west] {4} circle (1pt);
		\draw [fill=circ_col] (2,1) node [below] {5} circle (1pt);
		\draw [fill=circ_col] (0,2) node [above] {6} circle (1pt);
		\draw [fill=circ_col] (1,2) node [above] {7} circle (1pt);
		\draw [fill=circ_col] (2,2) node [above] {8} circle (1pt);
		
		\draw [-, line width = 1pt] (0,0) -- (1,0) -- (1,1) -- (2,1) -- (2,2) -- (1,2) -- (0,2) -- (0,1) -- (0,0);
		
		\draw [-, line width = 1pt] (1,0) -- node [below]{1} (0,1);
		\draw [-, line width = 1pt] (0,1) -- node [below]{2} (1,1);
		\draw [-, line width = 1pt] (0,2) -- node [above]{3} (1,1);
		\draw [-, line width = 1pt] (1,1) -- node [anchor=west]{4} (1,2);
		\draw [-, line width = 1pt] (1,2) -- node [anchor=west]{5} (2,1);
		\end{tikzpicture}
		
	\end{center}
	\setlength{\abovecaptionskip}{-0.15cm}
	\caption{The arrangement of planes $ U_{0,5,8} $.}\label{fig_computation_U_0_5_8}
\end{figure}

\begin{thm}\label{thm_computation_U_0_5_8}
	If $ X $ degenerates to $ U_{0,5,8} $, then $\pi_1(\Xgal)$ is trivial.
\end{thm}

\begin{proof}
	The branch curve $S_0$ in $\mathbb{CP}^2$ is an arrangement of five lines. We regenerate each vertex in turn and compute the group $G$.
	
	\uOnePoint{2}{1}{13-vert2}
\uOnePoint{5}{5}{13-vert5}
	\uOnePoint{6}{3}{13-vert6}
	\uTwoPoint{3}{1}{2}{13-vert3}
\uTwoPoint{7}{4}{5}{13-vert7}
	\uThreePointOuter{4}{2}{3}{4}{13-vert4}
	We also have the following parasitic and projective relations:
	\uParasit{1}{3}{13-parasit-1-3}
	\uParasit{1}{4}{13-parasit-1-4}
	\uParasit{1}{5}{13-parasit-1-5}
	\uParasit{2}{5}{13-parasit-2-5}
	\uParasit{3}{5}{13-parasit-3-5}
	\uProjRel{13-proj}{5}{4}{3}{2}{1}

By Lemma \ref{2pt-equal}, we get $\ug{2}=\ug{2'}$ and $\ug{4}=\ug{4'}$.
	
$\Ggal$ is thus generated by $\set{\ug{i} | i=1,\dots,5}$ with the following relations:
	\uGammaSq{13-gamma-sq}{1}{2}{3}{4}{5}
	\ubegineq{13-simpl-3}
	\trip{1}{2} = \trip{2}{3} = \trip{3}{4} = \trip{4}{5} = e
	\uendeq
	\ubegineq{13-simpl-4}
	\comm{1}{3} = \comm{1}{4} = \comm{1}{5} = \comm{2}{4} = \comm{2}{5} = \comm{3}{5} = e.
	\uendeq
Thus, $\Ggal \cong S_6$, therefore $\pi_1(\Xgal)$ is trivial.

\end{proof}

\uChernSummary{7}{24}{12}{10}{4}{\frac{9}{2}}{-\frac{5}{3}\cdot}

%+--------------------------------------+
%|                                      |
%|                CASE  U_{0,6,1}  	    |
%|                                      |
%+--------------------------------------+
\subsection{{$ X $ degenerates to $ U_{0,6,1} $}}\label{section_computation_U_0_6_1}

\begin{figure}[H]
	\begin{center}
		
		\definecolor{circ_col}{rgb}{0,0,0}
		\begin{tikzpicture}[x=1cm,y=1cm,scale=2]
		
		\draw [fill=circ_col] (0,0) node [below] {1} circle (1pt);
		\draw [fill=circ_col] (1,0) node [below] {2} circle (1pt);
		\draw [fill=circ_col] (0,1) node [anchor=east] {3} circle (1pt);
		\draw [fill=circ_col] (1,1) node [anchor=north west] {4} circle (1pt);
		\draw [fill=circ_col] (2,1) node [below] {5} circle (1pt);
		\draw [fill=circ_col] (0,2) node [above] {6} circle (1pt);
		\draw [fill=circ_col] (1,2) node [above] {7} circle (1pt);
		\draw [fill=circ_col] (2,2) node [above] {8} circle (1pt);
		
		\draw [-, line width = 1pt] (0,0) -- (1,0) -- (1,1) -- (2,1) -- (2,2) -- (1,2) -- (0,2) -- (0,1) -- (0,0);
		
		\draw [-, line width = 1pt] (0,0) -- node [below]{1} (1,1);
		\draw [-, line width = 1pt] (0,1) -- node [below]{2} (1,1);
		\draw [-, line width = 1pt] (0,2) -- node [below]{3} (1,1);
		\draw [-, line width = 1pt] (1,1) -- node [anchor=west]{4} (1,2);
		\draw [-, line width = 1pt] (1,2) -- node [anchor=west]{5} (2,1);
		\end{tikzpicture}
		
	\end{center}
	\setlength{\abovecaptionskip}{-0.15cm}
	\caption{The arrangement of planes $ U_{0,6,1} $.}\label{fig_computation_U_0_6_1_appendix}
\end{figure}

\begin{thm}\label{thm_computation_U_0_6_1}
	If $ X $ degenerates to $ U_{0,6,1} $, then $\pi_1(\Xgal)$ is trivial.
\end{thm}

\begin{proof}
	See Theorem \ref{thm_computation_U_0_6_1_paper}.

\end{proof}

	\uChernSummary{8}{24}{12}{10}{4}{5}{-2\cdot}

%+--------------------------------------+
%|                                      |
%|               CASE  U_{0,6,2}	    |
%|                                      |
%+--------------------------------------+
\subsection{{$ X $ degenerates to $ U_{0,6,2} $}}\label{section_computation_U_0_6_2}

\begin{figure}[H]
	\begin{center}
		
		\definecolor{circ_col}{rgb}{0,0,0}
		\begin{tikzpicture}[x=1cm,y=1cm,scale=2]
		
		\draw [fill=circ_col] (0,0) node [anchor=east] {1} circle (1pt);
		\draw [fill=circ_col] (2,1) node [anchor=west] {2} circle (1pt);
		\draw [fill=circ_col] (0,1) node [anchor=east] {3} circle (1pt);
		\draw [fill=circ_col] (1,1) node [below] {4} circle (1pt);
		\draw [fill=circ_col] (2,2) node [anchor=west] {5} circle (1pt);
		\draw [fill=circ_col] (0,2) node [above] {6} circle (1pt);
		\draw [fill=circ_col] (1,2) node [above] {7} circle (1pt);
		\draw [fill=circ_col] (2,3) node [above] {8} circle (1pt);
		
		\draw [-, line width = 1pt] (0,0) -- (1,1) -- (2,1) -- (2,2) -- (2,3) -- (1,2) -- (0,2) -- (0,1) -- (0,0);
		
		\draw [-, line width = 1pt] (0,1) -- node [below]{1} (1,1);
		\draw [-, line width = 1pt] (0,2) -- node [below]{2} (1,1);
		\draw [-, line width = 1pt] (1,2) -- node [anchor=east]{3} (1,1);
		\draw [-, line width = 1pt] (2,2) -- node [below]{4} (1,1);
		\draw [-, line width = 1pt] (1,2) -- node [below]{5} (2,2);
		\end{tikzpicture}
		
	\end{center}
	\setlength{\abovecaptionskip}{-0.15cm}
	\caption{The arrangement of planes $ U_{0,6,2} $.}\label{fig_computation_U_0_6_2_appendix}
\end{figure}

\begin{thm}
	If $ X $ degenerates to $ U_{0,6,2} $,  then $\pi_1(\Xgal)$ is not trivial.
\end{thm}

\begin{proof}
The branch curve $S_0$ in $\mathbb{CP}^2$ is an arrangement of five lines. We regenerate each vertex in turn and compute the group $G$.
	
	\uOnePoint{3}{1}{15-vert3}
	\uOnePoint{6}{2}{15-vert6}
	\uTwoPoint{5}{4}{5}{15-vert5}
    \uTwoPoint{7}{3}{5}{15-vert7}
	\uFourPointOuter{4}{1}{2}{3}{4}{15-vert4}
	We also have the following parasitic and projective relations:
	\uParasit{1}{5}{15-parasit-1-5}
	\uParasit{2}{5}{15-parasit-2-5}
	\uProjRel{15-proj}{5}{4}{3}{2}{1}
	
Substituting \eqref{15-vert3} and \eqref{15-vert6} in \eqref{15-vert4-12}, we get $\ug{3}=\ug{3'}$.
By Lemma \ref{2pt-equal} , we have $\ug{5}=\ug{5'}$ and $\ug{4}=\ug{4'}$.

$\Ggal$ is thus generated by $\set{\ug{i} | i=1,\dots,5}$ with the following relations:
	\uGammaSq{15-gamma-sq}{1}{2}{3}{4}{5}
	\ubegineq{15-simpl-3}
	\trip{1}{2} = \trip{2}{3} = \trip{3}{4} = \trip{3}{5} = \trip{4}{5} = e
	\uendeq
	\ubegineq{15-simpl-4}
	\comm{1}{3} = \comm{1}{4} = \comm{1}{5} = \comm{2}{4} = \comm{2}{5} = e.
	\uendeq
By Lemma \ref{not-trivial}, $\pi_1(\Xgal)$ is not trivial.

\end{proof}

\uChernSummary{7}{20}{15}{10}{4}{4}{-\frac{4}{3}\cdot}

%+--------------------------------------+
%|                              	    |
%|                CASE  U_0_6_3		    |
%|                                     	|
%+--------------------------------------+
\subsection{{$ X $ degenerates to $U_{0,6,3}$}}\label{section_computation_U_0_6_3}

\begin{figure}[H]
	\begin{center}
		
		\definecolor{circ_col}{rgb}{0,0,0}
		\begin{tikzpicture}[x=1cm,y=1cm,scale=2]
		
		\draw [fill=circ_col] (0,0) node [anchor=east] {1} circle (1pt);
		\draw [fill=circ_col] (2,1) node [anchor=west] {2} circle (1pt);
		\draw [fill=circ_col] (0,1) node [anchor=east] {3} circle (1pt);
		\draw [fill=circ_col] (1,1) node [below] {4} circle (1pt);
		\draw [fill=circ_col] (2,2) node [anchor=west] {5} circle (1pt);
		\draw [fill=circ_col] (0,2) node [above] {6} circle (1pt);
		\draw [fill=circ_col] (1,2) node [anchor=south west] {7} circle (1pt);
		\draw [fill=circ_col] (1,3) node [above] {8} circle (1pt);
		
		\draw [-, line width = 1pt] (0,0) -- (1,1) -- (2,1) -- (2,2) -- (1,2) -- (1,3) -- (0,2) -- (0,1) -- (0,0);
		
		\draw [-, line width = 1pt] (0,1) -- node [below]{1} (1,1);
		\draw [-, line width = 1pt] (0,2) -- node [below]{2} (1,1);
		\draw [-, line width = 1pt] (1,2) -- node [anchor=east]{3} (1,1);
		\draw [-, line width = 1pt] (2,2) -- node [below]{4} (1,1);
		\draw [-, line width = 1pt] (0,2) -- node [below]{5} (1,2);
		\end{tikzpicture}
		
	\end{center}
	\setlength{\abovecaptionskip}{-0.15cm}
	\caption{The arrangement of planes $ U_{0,6,3} $.}\label{fig_computation_U_0_6_3_appendix}
\end{figure}

\begin{thm}
	If $ X $ degenerates to $ U_{0,6,3} $,  then $\pi_1(\Xgal)$ is not trivial.
\end{thm}

\begin{proof}
	The branch curve $S_0$ in $\mathbb{CP}^2$ is an arrangement of five lines. We regenerate each vertex in turn and compute the group $G$.
	
	\uOnePoint{3}{1}{U063-vert3}
	\uOnePoint{5}{4}{U063-vert5}
	\uTwoPoint{6}{2}{5}{U063-vert6}
    \uTwoPoint{7}{3}{5}{U063-vert7}
	\uFourPointOuter{4}{1}{2}{3}{4}{U063-vert4}
	We also have the following parasitic and projective relations:
	\uParasit{1}{5}{U063-parasit-1-5}
	\uParasit{4}{5}{U063-parasit-4-5}
	\uProjRel{U063-proj}{5}{4}{3}{2}{1}
	
Using \eqref{U063-vert5} and \eqref{U063-vert4-2} in \eqref{U063-vert4-8}, we get $\ug{3}=\ug{3'}$.
We use also  \eqref{U063-vert3} and \eqref{U063-vert4-9} in \eqref{U063-vert4-10}, we get $\ug{2}=\ug{2'}$.
By Lemma \ref{2pt-equal} , we have $\ug{5}=\ug{5'}$.

$\Ggal$ is thus generated by $\set{\ug{i} | i=1,\dots,5}$ with the following relations:
	\uGammaSq{U063-gamma-sq}{1}{2}{3}{4}{5}
	\ubegineq{U063-simpl-3}
	\trip{1}{2} = \trip{2}{3} = \trip{2}{5} = \trip{3}{4} = \trip{3}{5} = e,
	\uendeq
	\ubegineq{U063-simpl-4}
	\comm{1}{3} = \comm{1}{4} = \comm{1}{5} = \comm{2}{4} = \comm{4}{5} = e.
	\uendeq
By Lemma \ref{not-trivial}, $\pi_1(\Xgal)$ is not trivial.	

\end{proof}

\uChernSummary{7}{20}{15}{10}{4}{4}{-\frac{4}{3}\cdot}

%+--------------------------------------+
%|                              	    |
%|                CASE  U_0_7		    |
%|                                     	|
%+--------------------------------------+
\subsection{{$ X $ degenerates to $U_{0,7}$}}\label{section_computation_U_0_7}

\begin{figure}[H]
	\begin{center}
		
		\definecolor{circ_col}{rgb}{0,0,0}
		\begin{tikzpicture}[x=1cm,y=1cm,scale=2]
		
		\draw [fill=circ_col] (0,0) node [below] {1} circle (1pt);
		\draw [fill=circ_col] (1,-1) node [anchor=west] {2} circle (1pt);
		\draw [fill=circ_col] (1,0) node [anchor=west] {3} circle (1pt);
		\draw [fill=circ_col] (1,1) node [anchor=west] {4} circle (1pt);
		\draw [fill=circ_col] (0,1) node [above] {5} circle (1pt);
		\draw [fill=circ_col] (-1,1) node [anchor=east] {6} circle (1pt);
		\draw [fill=circ_col] (-1,0) node [anchor=east] {7} circle (1pt);
		\draw [fill=circ_col] (-1,-1) node [anchor=east] {8} circle (1pt);
		
		\draw [-, line width=1pt] (0,0) -- (1,-1) -- (1,0) -- (1,1) -- (0,1) -- (-1,1) -- (-1,0) -- (-1,-1) -- (0,0);
		
		\draw [-, line width = 1pt] (0,0) -- node [below]{1} (1,0);
		\draw [-, line width = 1pt] (0,0) -- node [above]{2} (1,1);
		\draw [-, line width = 1pt] (0,0) -- node [anchor=east]{3} (0,1);
		\draw [-, line width = 1pt] (0,0) -- node [above]{4} (-1,1);
		\draw [-, line width = 1pt] (0,0) -- node [below]{5} (-1,0);

		\end{tikzpicture}
		
	\end{center}
	\setlength{\abovecaptionskip}{-0.15cm}
	\caption{The arrangement of planes $ U_{0,7} $.}\label{fig_computation_U_0_7}
\end{figure}

\begin{thm}\label{thm_computation_U_0_7}
If $ X $ degenerates to $ U_{0,7} $, then $\pi_1(\Xgal)$ is trivial.
\end{thm}

\begin{proof}
	The branch curve $S_0$ in $\mathbb{CP}^2$ is an arrangement of five lines. We regenerate each vertex in turn and compute the group $G$.
	
   \uOnePoint{3}{1}{U07-vert3}
   \uOnePoint{4}{2}{U07-vert4}
   \uOnePoint{5}{3}{U07-vert5}
   \uOnePoint{6}{4}{U07-vert6}
   \uOnePoint{7}{5}{U07-vert7}
Vertex 1 is an outer $5$-point, the braid monodromy corresponding to it is:

\begin{align*}\label{}
\widetilde{\Delta}_{1} =  &Z_{4\; 4', 5}^3 \cdot (Z_{1\; 1', 5}^2)^{Z_{2\; 2', 5}^{-2} Z_{3\; 3', 5}^{-2}} \cdot
(Z_{2\; 2', 5}^2)^{Z_{3\; 3', 5}^{-2}} \cdot Z_{3\; 3', 5}^2 \cdot {\bar{Z}}_{1\; 1', 5'}^2 \cdot {\bar{Z}}_{2\; 2', 5'}^2
\cdot {\bar{Z}}_{3\; 3', 5'}^2  \cdot \\
& (Z_{5\; 5'})^{Z_{4\; 4', 5}^2} \cdot (Z_{3\; 3', 4}^3)^{Z_{4\; 4', 5}^2}
\cdot (Z_{1\; 1', 4}^2)^{Z_{2\; 2', 4}^{-2} Z_{4\; 4', 5}^2} \cdot (Z_{2\; 2', 4}^2)^{Z_{4\; 4', 5}^2} \cdot ({\bar{Z}}_{1\; 1', 4'}^2)^{Z_{4\; 4', 5}^2}   \cdot \\
& ({\bar{Z}}_{2\; 2', 4'}^2)^{Z_{4\; 4', 5}^2} \cdot  (Z_{4\; 4'})^{Z_{3\; 3', 4}^2 Z_{4\; 4', 5}^2} \cdot Z_{1', 2 \; 2'}^3 \cdot
(Z_{1\; 1'})^{Z_{1', 2\; 2'}^2} \cdot  \\
& (Z_{2\; 2', 3}^3)^{Z_{1', 2\; 2'}^2 Z_{3\; 3', 4}^2 Z_{4\; 4', 5}^2} \cdot  (Z_{3\; 3'})^{Z_{2\; 2', 3}^2 Z_{1', 2\; 2'}^2 Z_{3\; 3', 4}^2 Z_{4\; 4', 5}^2} \cdot (Z_{1\; 1', 3\; 3'}^2)^{Z_{3\; 3' , 4}^2 Z_{4\; 4', 5}^2}.
\end{align*}
It gives rise to the following relations:
	\begin{equation}\label{five1}
	\trip{4}{5} = \trip{4'}{5} = \langle \ug{4}^{-1}\ug{4'}\ug{4}, \ug{5} \rangle =e
	\end{equation}
	\begin{equation}\label{five2}
	[\ug{3'}\ug{3}\ug{2'}\ug{2}\ug{1}\ug{2}^{-1}\ug{2'}^{-1}\ug{3}^{-1}\ug{3'}^{-1}, \ug{5}] = [\ug{3'}\ug{3}\ug{2'}\ug{2}\ug{1'}\ug{2}^{-1}\ug{2'}^{-1}\ug{3}^{-1}\ug{3'}^{-1}, \ug{5}] = e
	\end{equation}
\begin{equation}\label{five3}
	[\ug{3'}\ug{3}\ug{2}\ug{3}^{-1}\ug{3'}^{-1}, \ug{5}] =[\ug{3'}\ug{3}\ug{2'}\ug{3}^{-1}\ug{3'}^{-1}, \ug{5}] = e
	\end{equation}
	\begin{equation}\label{five4}
	\comm{3}{5} = \comm{3'}{5} = e
	\end{equation}
	\begin{equation}\label{five5}
 \begin{split}
	[\ug{4'}\ug{4}\ug{3'}\ug{3}\ug{2'}\ug{2}\ug{1}\ug{2}^{-1}\ug{2'}^{-1}\ug{3}^{-1}\ug{3'}^{-1}\ug{4}^{-1}\ug{4'}^{-1}, \ug{5}^{-1}\ug{5'}\ug{5} ] = 	\\ [\ug{4'}\ug{4}\ug{3'}\ug{3}\ug{2'}\ug{2}\ug{1'}\ug{2}^{-1}\ug{2'}^{-1}\ug{3}^{-1}\ug{3'}^{-1}\ug{4}^{-1}\ug{4'}^{-1}, \ug{5}^{-1}\ug{5'}\ug{5} ] = e
\end{split}
	\end{equation}
\begin{equation}\label{five6}
	[\ug{4'}\ug{4}\ug{3'}\ug{3}\ug{2}\ug{3}^{-1}\ug{3'}^{-1}\ug{4}^{-1}\ug{4'}^{-1}, \ug{5}^{-1}\ug{5'}\ug{5} ] = 	 [\ug{4'}\ug{4}\ug{3'}\ug{3}\ug{2'}\ug{3}^{-1}\ug{3'}^{-1}\ug{4}^{-1}\ug{4'}^{-1}, \ug{5}^{-1}\ug{5'}\ug{5} ] = e
	\end{equation}
\begin{equation}\label{five7}
	[\ug{4'}\ug{4}\ug{3}\ug{4}^{-1}\ug{4'}^{-1}, \ug{5}^{-1}\ug{5'}\ug{5} ] = 	 [\ug{4'}\ug{4}\ug{3'}\ug{4}^{-1}\ug{4'}^{-1}, \ug{5}^{-1}\ug{5'}\ug{5} ] = e
	\end{equation}
    \begin{equation}\label{five8}
	\ug{5}\ug{4'}\ug{4}\ug{5}\ug{4}^{-1}\ug{4'}^{-1}\ug{5}^{-1} = \ug{5'}
	\end{equation}
	\begin{equation}\label{five9}
	\langle \ug{3}, \ug{5}\ug{4}\ug{5}^{-1} \rangle = \langle \ug{3'}, \ug{5}\ug{4}\ug{5}^{-1} \rangle =
 \langle \ug{3}^{-1}\ug{3'}\ug{3}, \ug{5}\ug{4}\ug{5}^{-1} \rangle = e
	\end{equation}
	\begin{equation}\label{five10}
	[\ug{2'}\ug{2}\ug{1}\ug{2}^{-1}\ug{2'}^{-1}, \ug{5}\ug{4}\ug{5}^{-1}] = [\ug{2'}\ug{2}\ug{1'}\ug{2}^{-1}\ug{2'}^{-1}, \ug{5}\ug{4}\ug{5}^{-1}] = e
	\end{equation}
\begin{equation}\label{five11}
	[\ug{2}, \ug{5}\ug{4}\ug{5}^{-1}] = [\ug{2'}, \ug{5}\ug{4}\ug{5}^{-1}] = e
	\end{equation}
    \begin{equation}\label{five12}
	[\ug{3'}\ug{3}\ug{2'}\ug{2}\ug{1}\ug{2}^{-1}\ug{2'}^{-1}\ug{3}^{-1}\ug{3'}^{-1}, \ug{5}\ug{4}^{-1}\ug{4'}\ug{4}\ug{5}^{-1} ] = 	 [\ug{3'}\ug{3}\ug{2'}\ug{2}\ug{1'}\ug{2}^{-1}\ug{2'}^{-1}\ug{3}^{-1}\ug{3'}^{-1}, \ug{5}\ug{4}^{-1}\ug{4'}\ug{4}\ug{5}^{-1} ] =  e
	\end{equation}
     \begin{equation}\label{five13}
	[\ug{3'}\ug{3}\ug{2}\ug{3}^{-1}\ug{3'}^{-1}, \ug{5}\ug{4}^{-1}\ug{4'}\ug{4}\ug{5}^{-1} ] = 	 [\ug{3'}\ug{3}\ug{2'}\ug{3}^{-1}\ug{3'}^{-1}, \ug{5}\ug{4}^{-1}\ug{4'}\ug{4}\ug{5}^{-1} ] =  e
	\end{equation}
     \begin{equation}\label{five14}
	 \ug{3}^{-1}\ug{3'}^{-1} \ug{5}\ug{4}^{-1}\ug{4'}\ug{4}\ug{5}^{-1} \ug{3'}\ug{3} = \ug{5}\ug{4}\ug{5}^{-1} 	
     \end{equation}
      \begin{equation}\label{five15}
	\langle \ug{1'}, \ug{2} \rangle = \langle \ug{1'}, \ug{2'} \rangle =
	 \langle \ug{1'}, \ug{2}^{-1}\ug{2'}\ug{2} \rangle = e
	\end{equation}
\begin{equation}\label{five16}
	 \ug{1} = \ug{2'}\ug{2}\ug{1'}\ug{2}^{-1}\ug{2'}^{-1} 	
     \end{equation}
 \begin{equation}\label{five17}
     \begin{split}
\langle \ug{2'}\ug{2}\ug{1'}\ug{2}\ug{1'}^{-1}\ug{2}^{-1}\ug{2'}^{-1},\ug{5}\ug{4}\ug{5}^{-1}\ug{3}\ug{5}\ug{4}^{-1}\ug{5}^{-1} \rangle =
\langle \ug{2'}\ug{2}\ug{1'}\ug{2'}\ug{1'}^{-1}\ug{2}^{-1}\ug{2'}^{-1},\ug{5}\ug{4}\ug{5}^{-1}\ug{3}\ug{5}\ug{4}^{-1}\ug{5}^{-1} \rangle  =\\
 \langle \ug{2'}\ug{2}\ug{1'}\ug{2}^{-1}\ug{2'}\ug{2}\ug{1'}^{-1}\ug{2}^{-1}\ug{2'}^{-1},\ug{5}\ug{4}\ug{5}^{-1}\ug{3}\ug{5}\ug{4}^{-1}\ug{5}^{-1} \rangle  =	e
\end{split}
	\end{equation}
\begin{equation}\label{five18}
\begin{split}
	\ug{2'}\ug{2}\ug{1'}\ug{2}^{-1}\ug{2'}^{-1}\ug{1'}^{-1}\ug{2}^{-1}\ug{2'}^{-1} \ug{5}\ug{4}\ug{5}^{-1}\ug{3}^{-1}\ug{3'} \ug{3}\ug{5}\ug{4}^{-1}\ug{5}^{-1}\ug{2'}\ug{2}\ug{1'}\ug{2'}\ug{2}\ug{1'}^{-1}\ug{2}^{-1}\ug{2'}^{-1} = \\
 \ug{5}\ug{4}\ug{5}^{-1}\ug{3} \ug{5}\ug{4}^{-1}\ug{5}^{-1}
 \end{split}
\end{equation}
 \begin{equation}\label{five19}
     \begin{split}
	[\ug{1}, \ug{5}\ug{4}\ug{5}^{-1}\ug{3}\ug{5}\ug{4}^{-1}\ug{5}^{-1} ] = 	 [\ug{1'},\ug{5}\ug{4}\ug{5}^{-1}\ug{3}\ug{5}\ug{4}^{-1}\ug{5}^{-1} ] =\\
[\ug{1}, \ug{5}\ug{4}\ug{5}^{-1}\ug{3'}\ug{5}\ug{4}^{-1}\ug{5}^{-1} ] =
[\ug{1'}, \ug{5}\ug{4}\ug{5}^{-1}\ug{3'}\ug{5}\ug{4}^{-1}\ug{5}^{-1} ] = e.
\end{split}
	\end{equation}
   We also have the following projective relation:
	\uProjRel{U07-proj}{5}{4}{3}{2}{1}

By \eqref{U07-vert3}-\eqref{U07-vert7}, $\Ggal$ is generated by $\set{\ug{i} | i=1,\dots,5}$ with the following relations:
    \uGammaSq{U07-gamma-sq}{1}{2}{3}{4}{5}
	\ubegineq{U07-simpl-3}
	\trip{1}{2} = \trip{2}{3} = \trip{3}{4} = \trip{4}{5} = e
	\uendeq
	\ubegineq{U07-simpl-4}
	\comm{1}{3} = \comm{1}{4} = \comm{1}{5} = \comm{2}{4} =  \comm{2}{5} = \comm{3}{5} = e.
	\uendeq
Thus, $\Ggal \cong S_6$, therefore $\pi_1(\Xgal)$ is trivial.

\end{proof}

\uChernSummary{9}{24}{12}{10}{4}{\frac{11}{2}}{-\frac{7}{3}\cdot}

%+--------------------------------------+
%|                              	    |
%|                CASE  U_3_1  		    |
%|                                     	|
%+--------------------------------------+
\subsection{{$ X $ degenerates to $U_{3,1}$}}\label{section_computation_U_3_1}

\begin{figure}[H]
	\begin{center}
		
		\definecolor{circ_col}{rgb}{0,0,0}
		\begin{tikzpicture}[x=1cm,y=1cm,scale=2]
		
		\draw [fill=circ_col] (0.6,1.03) node [anchor=west] {1} circle (1pt);
		\draw [fill=circ_col] (-0.86, -0.5) node [below] {2} circle (1pt);
		\draw [fill=circ_col] (0.86, -0.5) node [below] {3} circle (1pt);
		\draw [fill=circ_col] (0,0) node [below] {4} circle (1pt);
		\draw [fill=circ_col] (-0.86, 0.5) node [above] {5} circle (1pt);
		\draw [fill=circ_col] (0.86, 0.5) node [above] {6} circle (1pt);
		\draw [fill=circ_col] (0,1) node [above] {7} circle (1pt);
		
		\draw [-, line width = 1pt] (-0.86, -0.5) -- (0.86, -0.5) -- (0.86, 0.5)  -- (0.6,1.03)-- (0,1) -- (-0.86, 0.5) -- (-0.86, -0.5);
		
		\draw [-, line width = 1pt] (-0.86,-0.5) -- node [above] {4} (0, 1);
		\draw [-, line width = 1pt] (0,0) -- node [anchor=west] {3} (0, 1);
		\draw [-, line width = 1pt] (0.86,-0.5) -- node [anchor=west] {2} (0, 1);
		\draw [-, line width = 1pt] (0,0) -- node [above] {5} (-0.86, -0.5);
		\draw [-, line width = 1pt] (0,0) -- node [above] {6} (0.86, -0.5);
		\draw [-, line width = 1pt] (0.86,0.5) -- node [below] {1} (0,1);
		
		\end{tikzpicture}
	\end{center}
	\setlength{\abovecaptionskip}{-0.15cm}
	\caption{The arrangement of planes $U_{3,1}$.}\label{fig_computation_U_3_1}
\end{figure}

\begin{thm}\label{thm_computation_U_3_1}
If $ X $ degenerates to $ U_{3,1} $, then $\pi_1(\Xgal)$ is trivial.
\end{thm}

\begin{proof}
	The branch curve $S_0$ in $\mathbb{CP}^2$ is an arrangement of six lines. We regenerate each vertex in turn and compute the group $G$.
	
	\uOnePoint{6}{1}{U31-vert6}
	\uTwoPoint{2}{4}{5}{U31-vert2}
	\uTwoPoint{3}{2}{6}{U31-vert3}
	\uThreePointInner{4}{3}{5}{6}{U31-vert4}
	\uFourPointOuter{7}{1}{2}{3}{4}{U31-vert7}
	We also have the following parasitic and projective relations:
	\uParasit{1}{5}{U31-parasit-1-5}
	\uParasit{1}{6}{U31-parasit-1-6}
	\uParasit{2}{5}{U31-parasit-2-5}
	\uParasit{4}{6}{U31-parasit-4-6}
	\uProjRel{U31-proj}{6}{5}{4}{3}{2}{1}

Using \eqref{U31-vert6} and \eqref{U31-vert7-9} in \eqref{U31-vert7-10}, we get $\ug{2}=\ug{2'}$.
By Lemma \ref{2pt-equal}, we get $\ug{6}=\ug{6'}$. Then by Lemma \ref{3pt-in-bigmid}, we have  $\ug{3}=\ug{3'}, \ug{5}=\ug{5'}$. And again by Lemma \ref{2pt-equal} we get $\ug{4}=\ug{4'}$.
	
$\Ggal$ is thus generated by $\set{\ug{i} | i=1,\dots,6}$ with the following relations:
\uGammaSq{U31-gamma-sq}{1}{2}{3}{4}{5}{6}
	\ubegineq{U31-simpl-3}
	\trip{1}{2} = \trip{2}{3} = \trip{2}{6} = \trip{3}{4} = \trip{3}{5} = \trip{3}{6} = \trip{4}{5} = \trip{5}{6} = e
	\uendeq
	\ubegineq{U31-simpl-4}
	\comm{1}{3} = \comm{1}{4} = \comm{1}{5} = \comm{1}{6} = \comm{2}{4} =  \comm{2}{5} = \comm{4}{6} = e
	\uendeq
	\ubegineq{U31-simpl-5}
	\ug{6}=\ug{3}\ug{5}\ug{3}.
	\uendeq
By Remark 2.8 in \cite{RTV}, we get also the relations
\ubegineq{U31-simpl-6}
	[\ug{2},\ug{3}\ug{6}\ug{3}]=[\ug{4},\ug{3}\ug{5}\ug{3}] = e.
	\uendeq
By Theorem 2.3 in \cite{RTV} we have $\Ggal \cong S_6$, therefore $\pi_1(\Xgal)$ is trivial.

\end{proof}

\uChernSummary{10}{28}{21}{12}{9}{\frac{13}{2}}{-\frac{4}{3}\cdot}

%+--------------------------------------+
%|                              	    |
%|                CASE  U_3_2  		    |
%|                                    	|
%+--------------------------------------+

\subsection{{$ X $ degenerates to $ U_{3,2} $}}\label{section_computation_U_3_2}

\begin{figure}[H]
	\begin{center}
		
		\definecolor{circ_col}{rgb}{0,0,0}
		\begin{tikzpicture}[x=1cm,y=1cm,scale=2]
		
		\draw [fill=circ_col] (0,-1) node [below] {1} circle (1pt);
		\draw [fill=circ_col] (-0.86, -0.5) node [below] {2} circle (1pt);
		\draw [fill=circ_col] (0.86, -0.5) node [below] {3} circle (1pt);
		\draw [fill=circ_col] (0,0) node [below] {4} circle (1pt);
		\draw [fill=circ_col] (-0.86, 0.5) node [above] {5} circle (1pt);
		\draw [fill=circ_col] (0.86, 0.5) node [above] {6} circle (1pt);
		\draw [fill=circ_col] (0,1) node [above] {7} circle (1pt);
		
		\draw [-, line width = 1pt] (-0.86, -0.5) -- (0,-1) -- (0.86, -0.5) -- (0.86, 0.5) -- (0,1) -- (-0.86, 0.5) -- (-0.86, -0.5);
		
		\draw [-, line width = 1pt] (-0.86,-0.5) -- node [above] {1} (0, 1);
		\draw [-, line width = 1pt] (0,0) -- node [anchor=west] {2} (0, 1);
		\draw [-, line width = 1pt] (0.86,-0.5) -- node [anchor=west] {3} (0, 1);
		\draw [-, line width = 1pt] (0,0) -- node [above] {4} (-0.86, -0.5);
		\draw [-, line width = 1pt] (0,0) -- node [above] {5} (0.86, -0.5);
		\draw [-, line width = 1pt] (-0.86,-0.5) -- node [below] {6} (0.86, -0.5);
		
		\end{tikzpicture}
	\end{center}
	\setlength{\abovecaptionskip}{-0.15cm}
	\caption{The arrangement of planes $ U_{3,2} $.}\label{fig_computation_U_3_2}
\end{figure}

\begin{thm}\label{thm_computation_U_3_2}
	If $ X $ degenerates to $ U_{3,2} $, then $\pi_1(\Xgal)$ is trivial.
\end{thm}

\begin{proof}

	The branch curve $S_0$ in $\mathbb{CP}^2$ is an arrangement of six lines. We regenerate each vertex in turn and compute the group $G$.
	
	\uThreePointOuter{2}{1}{4}{6}{24-vert2}
	\uThreePointOuter{3}{3}{5}{6}{24-vert3}
	\uThreePointInner{4}{2}{4}{5}{24-vert4}
	\uThreePointOuter{7}{1}{2}{3}{24-vert7}
	We also have the following parasitic and projective relations:
	\uParasit{1}{5}{24-parasit-1-5}
	\uParasit{2}{6}{24-parasit-2-6}
	\uParasit{3}{4}{24-parasit-3-4}
	\uProjRel{24-proj}{6}{5}{4}{3}{2}{1}

Equating \eqref{24-vert4-3} and \eqref{24-vert4-5} and using \eqref{24-vert4-1} and \eqref{24-vert4-2}, we get  $\ug{2}=\ug{2'}$.
By Lemma \ref{3pt-out-smallmid}, we get $\ug{1}=\ug{1'}, \ug{3}=\ug{3'}, \ug{4}=\ug{4'}, \ug{5}=\ug{5'}$ and $\ug{6}=\ug{6'}$.
	
$\Ggal$ is thus generated by $\set{\ug{i} | i=1,\dots,6}$ with the following relations:
	\uGammaSq{24-gamma-sq}{1}{2}{3}{4}{5}{6}
	\ubegineq{24-simpl-3}
	\begin{split}
		\trip{1}{2} = \trip{1}{4} = \trip{2}{3} = \trip{2}{4} = \trip{2}{5} & = \\
	= \trip{3}{5} = \trip{4}{5} = \trip{4}{6} = \trip{5}{6} &= e
	\end{split}
	\uendeq
	\ubegineq{24-simpl-4}
	\comm{1}{3} = \comm{1}{5} = \comm{1}{6} = \comm{2}{6} = \comm{3}{4} = \comm{3}{6} = e
	\uendeq
	\ubegineq{24-simpl-5}
	\ug{5}=\ug{2}\ug{4}\ug{2}.
	\uendeq
By Remark 2.8 in \cite{RTV}, we get also the relations
\ubegineq{24-simpl-6}
	[\ug{1},\ug{2}\ug{4}\ug{2}]=[\ug{3},\ug{2}\ug{5}\ug{2}]=[\ug{6},\ug{4}\ug{5}\ug{4}] = e.
	\uendeq
By Theorem 2.3 in \cite{RTV} we have $\Ggal \cong S_6$, therefore $\pi_1(\Xgal)$ is trivial.

\end{proof}

\uChernSummary{10}{24}{24}{12}{9}{6}{-}

%+--------------------------------------+
%|                              	    |
%|                CASE  U_{3,3}		    |
%|                                     	|
%+--------------------------------------+

\subsection{{$ X $ degenerates to $ U_{3,3} $}}\label{section_computation_U_3_3}

\begin{figure}[H]
	\begin{center}
		
		\definecolor{circ_col}{rgb}{0,0,0}
		\begin{tikzpicture}[x=1cm,y=1cm,scale=2]
		
		\draw [fill=circ_col] (1.2,0) node [anchor=west] {1} circle (1pt);
		\draw [fill=circ_col] (-0.86, -0.5) node [below] {2} circle (1pt);
		\draw [fill=circ_col] (0.86, -0.5) node [below] {3} circle (1pt);
		\draw [fill=circ_col] (0,0) node [below] {4} circle (1pt);
		\draw [fill=circ_col] (-0.86, 0.5) node [above] {5} circle (1pt);
		\draw [fill=circ_col] (0.86, 0.5) node [above] {6} circle (1pt);
		\draw [fill=circ_col] (0,1) node [above] {7} circle (1pt);
		
		\draw [-, line width = 1pt] (-0.86, -0.5) -- (0.86, -0.5) -- (1.2,0) -- (0.86, 0.5) -- (0,1) -- (-0.86, 0.5) -- (-0.86, -0.5);
		
		\draw [-, line width = 1pt] (-0.86,-0.5) -- node [above] {6} (0, 1);
		\draw [-, line width = 1pt] (0,0) -- node [anchor=west] {3} (0, 1);
		\draw [-, line width = 1pt] (0.86,-0.5) -- node [anchor=west] {2} (0, 1);
		\draw [-, line width = 1pt] (0,0) -- node [above] {4} (-0.86, -0.5);
		\draw [-, line width = 1pt] (0,0) -- node [above] {5} (0.86, -0.5);
		\draw [-, line width = 1pt] (0.86,0.5) -- node [anchor=east] {1} (0.86, -0.5);
		
		\end{tikzpicture}
	\end{center}
	\setlength{\abovecaptionskip}{-0.15cm}
	\caption{The arrangement of planes $ U_{3,3} $.}\label{fig_computation_U_3_3}
\end{figure}

\begin{thm}\label{thm_computation_U_3_3}
	
	If $ X $ degenerates to $ U_{3,3} $, then $\pi_1(\Xgal)$ is trivial.
\end{thm}

\begin{proof}
	The branch curve $S_0$ in $\mathbb{CP}^2$ is an arrangement of six lines. We regenerate each vertex in turn and compute the group $G$.
	
    \uOnePoint{6}{1}{25-vert6}
	\uTwoPoint{2}{4}{6}{25-vert2}
	\uThreePointOuter{3}{1}{2}{5}{25-vert3}
	\uThreePointInner{4}{3}{4}{5}{25-vert4}
	\uThreePointOuter{7}{2}{3}{6}{25-vert7}
	We also have the following parasitic and projective relations:
	\uParasit{1}{3}{25-parasit-1-3}
	\uParasit{1}{4}{25-parasit-1-4}
	\uParasit{1}{6}{25-parasit-1-6}
	\uParasit{2}{4}{25-parasit-2-4}
	\uParasit{5}{6}{25-parasit-5-6}
	\uProjRel{25-proj}{6}{5}{4}{3}{2}{1}

By Lemma \ref{3pt-out-smallmid}, we get $\ug{2}=\ug{2'}, \ug{5}=\ug{5'}$ and $\ug{3}=\ug{3'}, \ug{6}=\ug{6'}$.
Then by Lemma \ref{2pt-equal}, we get $\ug{4}=\ug{4'}$.

$\Ggal$ is thus generated by $\set{\ug{i} | i=1,\dots,6}$ with the following relations:
	\uGammaSq{25-gamma-sq}{1}{2}{3}{4}{5}{6}
	\ubegineq{25-simpl-3}
	\trip{1}{2} = \trip{2}{3} = \trip{2}{5} = \trip{3}{4} = \trip{3}{5} = \trip{3}{6} = \trip{4}{5} = \trip{4}{6}= e
	\uendeq
	\ubegineq{25-simpl-4}
	\comm{1}{3} = \comm{1}{4} = \comm{1}{5} = \comm{1}{6} = \comm{2}{4} = \comm{2}{6} = \comm{5}{6} = e
	\uendeq
	\ubegineq{25-simpl-5}
	\ug{5}=\ug{3}\ug{4}\ug{3}.
	\uendeq
By Remark 2.8 in \cite{RTV}, we get also the relations
\ubegineq{25-simpl-6}
	[\ug{6},\ug{3}\ug{4}\ug{3}] = e.
	\uendeq
By Theorem 2.3 in \cite{RTV} we have $\Ggal \cong S_6$, therefore $\pi_1(\Xgal)$ is trivial.

\end{proof}

\uChernSummary{10}{28}{21}{12}{9}{\frac{13}{2}}{-\frac{4}{3}\cdot}

%+--------------------------------------+
%|                             		    |
%|                CASE  U_3_4  		    |
%|                                 	    |
%+--------------------------------------+

\subsection{{$ X $ degenerates to  $ U_{3,4} $}}\label{section_computation_U_3_4}

\begin{figure}[H]
	\begin{center}
		
		\definecolor{circ_col}{rgb}{0,0,0}
		\begin{tikzpicture}[x=1cm,y=1cm,scale=2]
		
		\draw [fill=circ_col] (0,1) node [above] {1} circle (1pt);
		\draw [fill=circ_col] (-0.86, -0.5) node [below] {2} circle (1pt);
		\draw [fill=circ_col] (0.86, -0.5) node [below] {3} circle (1pt);
		\draw [fill=circ_col] (0,0) node [below] {4} circle (1pt);
		\draw [fill=circ_col] (0.86, 0.5) node [above] {5} circle (1pt);
		\draw [fill=circ_col] (1.2,0) node [anchor=west] {6} circle (1pt);
		\draw [fill=circ_col] (1.4, -0.5) node [below] {7} circle (1pt);
		
		\draw [-, line width = 1pt] (-0.86, -0.5) -- (0.86, -0.5) -- (1.4, -0.5) -- (1.2,0) -- (0.86, 0.5) -- (0,1) -- (-0.86, -0.5);
		
		\draw [-, line width = 1pt] (-0.86, -0.5) -- node [below] {6} (0,0);
		\draw [-, line width = 1pt] (0,1) -- node [anchor=east] {5} (0,0);
		\draw [-, line width = 1pt] (0.86, -0.5) -- node [below] {4} (0,0);
		\draw [-, line width = 1pt] (0.86, -0.5) -- node [below] {3} (0,1);
		\draw [-, line width = 1pt] (0.86, -0.5) -- node [anchor=west] {2} (0.86, 0.5);
		\draw [-, line width = 1pt] (0.86, -0.5) -- node [below] {1} (1.2,0);
		
		\end{tikzpicture}
	\end{center}
	\setlength{\abovecaptionskip}{-0.15cm}
	\caption{The arrangement of planes $ U_{3,4} $.}\label{fig_computation_U_3_4}
\end{figure}

\begin{thm}\label{thm_computation_U_3_4}
	If $ X $ degenerates to $ U_{3,4} $, then $\pi_1(\Xgal)$ is trivial.
\end{thm}

\begin{proof}
	The branch curve $S_0$ in $\mathbb{CP}^2$ is an arrangement of six lines. We regenerate each vertex in turn and compute the group $G$.

	\uOnePoint{2}{6}{U34-vert2}
	\uOnePoint{5}{2}{U34-vert5}
	\uOnePoint{6}{1}{U34-vert6}
	\uTwoPoint{1}{3}{5}{U34-vert1}
	\uThreePointInner{4}{4}{5}{6}{U34-vert4}
	\uFourPointOuter{3}{1}{2}{3}{4}{U34-vert3}
	We also have the following parasitic and projective relations:
	\uParasit{1}{5}{U34-parasit-1-5}
	\uParasit{1}{6}{U34-parasit-1-6}
	\uParasit{2}{5}{U34-parasit-2-5}
	\uParasit{2}{6}{U34-parasit-2-6}
	\uParasit{3}{6}{U34-parasit-3-6}
	\uProjRel{U34-proj}{6}{5}{4}{3}{2}{1}

By Lemma \ref{3pt-in-bigmid}, we get $\ug{4}=\ug{4'}, \ug{5}=\ug{5'}$. Then by Lemma \ref{2pt-equal}, we have $\ug{3}=\ug{3'}$.

$\Ggal$ is thus generated by $\set{\ug{i} | i=1,\dots,6}$ with the following relations:
	\uGammaSq{U34-gamma-sq}{1}{2}{3}{4}{5}{6}
	\ubegineq{U34-simpl-3}
	\trip{1}{2} = \trip{2}{3} = \trip{3}{4} = \trip{3}{5} = \trip{4}{5} = \trip{4}{6} = \trip{5}{6} = e
	\uendeq
	\ubegineq{U34-simpl-4}
	\comm{1}{3} = \comm{1}{4} = \comm{1}{5} = \comm{1}{6} = \comm{2}{4} = \comm{2}{5} = \comm{2}{6} = \comm{3}{6} =e
	\uendeq
	\ubegineq{U34-simpl-5}
	\ug{6}=\ug{4}\ug{5}\ug{4}.
	\uendeq
By Remark 2.8 in \cite{RTV}, we get also the relation
\ubegineq{U34-simpl-6}
	[\ug{3},\ug{4}\ug{5}\ug{4}] = e.
	\uendeq
By Theorem 2.3 in \cite{RTV} we have $\Ggal \cong S_6$, therefore $\pi_1(\Xgal)$ is trivial.

\end{proof}

\uChernSummary{11}{32}{18}{12}{9}{\frac{15}{2}}{-2\cdot}

%+--------------------------------------+
%|                             		    |
%|                CASE  U_3_5  		    |
%|                                 	    |
%+--------------------------------------+

\subsection{{$ X $ degenerates to  $ U_{3,5} $}}\label{section_computation_U_3_5}

\begin{figure}[H]
	\begin{center}
		
		\definecolor{circ_col}{rgb}{0,0,0}
		\begin{tikzpicture}[x=1cm,y=1cm,scale=2]
		
		\draw [fill=circ_col] (0,1) node [above] {1} circle (1pt);
		\draw [fill=circ_col] (-0.86, -0.5) node [below] {2} circle (1pt);
		\draw [fill=circ_col] (0.86, -0.5) node [below] {3} circle (1pt);
		\draw [fill=circ_col] (0,0) node [below] {4} circle (1pt);
		\draw [fill=circ_col] (0.86, 0.5) node [above] {5} circle (1pt);
		\draw [fill=circ_col] (1.2,0) node [anchor=west] {6} circle (1pt);
		\draw [fill=circ_col] (0.5, 1) node [above] {7} circle (1pt);
		
		\draw [-, line width = 1pt] (-0.86, -0.5) -- (0.86, -0.5) -- (1.2,0) -- (0.86, 0.5) -- (0.5, 1) -- (0,1) -- (-0.86, -0.5);
		
		\draw [-, line width = 1pt] (-0.86, -0.5) -- node [below] {2} (0,0);
		\draw [-, line width = 1pt] (0,1) -- node [anchor=east] {1} (0,0);
		\draw [-, line width = 1pt] (0.86, -0.5) -- node [below] {3} (0,0);
		\draw [-, line width = 1pt] (0.86, -0.5) -- node [below] {4} (0,1);
		\draw [-, line width = 1pt] (0.86, -0.5) -- node [anchor=east] {5} (0.86, 0.5);
		\draw [-, line width = 1pt] (0.86, 0.5) -- node [below] {6} (0,1);
		
		\end{tikzpicture}
	\end{center}
	\setlength{\abovecaptionskip}{-0.15cm}
	\caption{The arrangement of planes $ U_{3,5} $.}\label{fig_computation_U_3_5}
\end{figure}

\begin{thm}\label{thm_computation_U_3_5}
	If $ X $ degenerates to $ U_{3,5} $, then $\pi_1(\Xgal)$ is not trivial.
\end{thm}

\begin{proof}
	The branch curve $S_0$ in $\mathbb{CP}^2$ is an arrangement of six lines. We regenerate each vertex in turn and compute the group $G$.
	
	\uOnePoint{2}{2}{U35-vert2}
	\uTwoPoint{5}{5}{6}{U35-vert5}
	\uThreePointOuter{1}{1}{4}{6}{U35-vert1}
	\uThreePointOuter{3}{3}{4}{5}{U35-vert3}
	\uThreePointInner{4}{1}{2}{3}{U35-vert4}
	We also have the following parasitic and projective relations:
	\uParasit{1}{5}{U35-parasit-1-5}
	\uParasit{2}{4}{U35-parasit-2-4}
	\uParasit{2}{5}{U35-parasit-2-5}
	\uParasit{2}{6}{U35-parasit-2-6}
	\uParasit{3}{6}{U35-parasit-3-6}
	\uProjRel{25-proj}{6}{5}{4}{3}{2}{1}
	
By Lemma \ref{3pt-in-bigmid}, we get  $\ug{1}=\ug{1'}, \ug{3}=\ug{3'}$. Then by Lemma \ref{3pt-out-smallmid}, we get $\ug{4}=\ug{4'},  \ug{5}=\ug{5'}$ and $\ug{6}=\ug{6'}$.

$\Ggal$ is thus generated by $\set{\ug{i} | i=1,\dots,6}$ with the following relations:
	\uGammaSq{U35-gamma-sq}{1}{2}{3}{4}{5}{6}
	\ubegineq{U35-simpl-3}
	\trip{1}{2} = \trip{1}{3} = \trip{1}{4} = \trip{2}{3} = \trip{3}{4} = \trip{4}{5} = \trip{4}{6} = \trip{5}{6}= e
	\uendeq
	\ubegineq{U35-simpl-4}
	\comm{1}{5} = \comm{1}{6} = \comm{2}{4} = \comm{2}{5} = \comm{2}{6} = \comm{3}{5} = \comm{3}{6} = e
	\uendeq
	\ubegineq{U35-simpl-5}
	\ug{3}=\ug{1}\ug{2}\ug{1}.
	\uendeq
By Lemma \ref{not-trivial}, $\pi_1(\Xgal)$ is not trivial.

\end{proof}

\uChernSummary{10}{28}{21}{12}{9}{\frac{13}{2}}{-\frac{4}{3}\cdot}

%+--------------------------------------+
%|                             		    |
%|                CASE  U_3_6  		    |
%|                                 	    |
%+--------------------------------------+

\subsection{{$ X $ degenerates to  $ U_{3,6} $}}\label{section_computation_U_3_6}

\begin{figure}[H]
	\begin{center}
		
		\definecolor{circ_col}{rgb}{0,0,0}
		\begin{tikzpicture}[x=1cm,y=1cm,scale=2]
		
		\draw [fill=circ_col] (0,1) node [above] {1} circle (1pt);
		\draw [fill=circ_col] (-0.86, -0.5) node [below] {2} circle (1pt);
		\draw [fill=circ_col] (0.86, -0.5) node [below] {3} circle (1pt);
		\draw [fill=circ_col] (0,0) node [below] {4} circle (1pt);
		\draw [fill=circ_col] (0.86, 0.5) node [above] {5} circle (1pt);
		\draw [fill=circ_col] (1.2,0) node [anchor=west] {6} circle (1pt);
		\draw [fill=circ_col] (1.2, 0.3) node [above] {7} circle (1pt);
		
		\draw [-, line width = 1pt] (-0.86, -0.5) -- (0.86, -0.5) -- (1.2,0) -- (1.2, 0.3) -- (0.86, 0.5) -- (0,1) -- (-0.86, -0.5);
		
		\draw [-, line width = 1pt] (-0.86, -0.5) -- node [below] {2} (0,0);
		\draw [-, line width = 1pt] (0,1) -- node [anchor=east] {1} (0,0);
		\draw [-, line width = 1pt] (0.86, -0.5) -- node [below] {3} (0,0);
		\draw [-, line width = 1pt] (0.86, -0.5) -- node [below] {4} (0,1);
		\draw [-, line width = 1pt] (0.86, -0.5) -- node [anchor=east] {5} (0.86, 0.5);
		\draw [-, line width = 1pt] (1.2,0) -- node [below] {6} (0.86, 0.5);
		
		\end{tikzpicture}
	\end{center}
	\setlength{\abovecaptionskip}{-0.15cm}
	\caption{The arrangement of planes $ U_{3,6} $.}\label{fig_computation_U_3_6}
\end{figure}

\begin{thm}\label{thm_computation_U_3_6}
	If $ X $ degenerates to $ U_{3,6} $, then $\pi_1(\Xgal)$ is trivial.
\end{thm}

\begin{proof}
	The branch curve $S_0$ in $\mathbb{CP}^2$ is an arrangement of six lines. We regenerate each vertex in turn and compute the group $G$.	

	\uOnePoint{2}{2}{U36-vert2}
	\uOnePoint{6}{6}{U36-vert6}
	\uTwoPoint{1}{1}{4}{U36-vert1}
	\uTwoPoint{5}{5}{6}{U36-vert5}
    \uThreePointOuter{3}{3}{4}{5}{U36-vert3}
	\uThreePointInner{4}{1}{2}{3}{U36-vert4}
	We also have the following parasitic and projective relations:
	\uParasit{1}{5}{U36-parasit-1-5}
	\uParasit{1}{6}{U36-parasit-1-6}
	\uParasit{2}{4}{U36-parasit-2-4}
	\uParasit{2}{5}{U36-parasit-2-5}
	\uParasit{2}{6}{U36-parasit-2-6}
	\uParasit{3}{6}{U36-parasit-3-6}
	\uParasit{4}{6}{U36-parasit-4-6}
	\uProjRel{25-proj}{6}{5}{4}{3}{2}{1}

By Lemma \ref{3pt-in-bigmid}, we get $\ug{1}=\ug{1'}, \ug{3}=\ug{3'}$. By Lemma \ref{2pt-equal}, we get $\ug{4}=\ug{4'}$ and   $\ug{5}=\ug{5'}$.

$\Ggal$ is thus generated by $\set{\ug{i} | i=1,\dots,6}$ with the following relations:
	\uGammaSq{U36-gamma-sq}{1}{2}{3}{4}{5}{6}
	\ubegineq{U36-simpl-3}
	\trip{1}{2} = \trip{1}{3} = \trip{1}{4} = \trip{2}{3} = \trip{3}{4} = \trip{4}{5} = \trip{5}{6}= e
	\uendeq
	\ubegineq{U36-simpl-4}
	\comm{1}{5} = \comm{1}{6} = \comm{2}{4} = \comm{2}{5} = \comm{2}{6} = \comm{3}{5} = \comm{3}{6} = \comm{4}{6} = e
	\uendeq
	\ubegineq{U36-simpl-5}
	\ug{3}=\ug{1}\ug{2}\ug{1}.
	\uendeq
By Remark 2.8 in \cite{RTV}, we get also the relation
\ubegineq{U36-simpl-6}
	[\ug{4},\ug{1}\ug{3}\ug{1}] = e.
	\uendeq
By Theorem 2.3 in \cite{RTV} we have $\Ggal \cong S_6$, therefore $\pi_1(\Xgal)$ is trivial.

\end{proof}

\uChernSummary{10}{32}{18}{12}{9}{7}{-\frac{5}{3}\cdot}

%+--------------------------------------+
%|                                      |
%|                CASE  U_{3 \cup 3}    |
%|                                      |
%+--------------------------------------+

\subsection{{$ X $ degenerates to $ U_{3 \cup 3} $}}\label{section_computation_U_3_cup_3}

\begin{figure}[H]
	\begin{center}
		
		\definecolor{circ_col}{rgb}{0,0,0}
		\begin{tikzpicture}[x=1cm,y=1cm,scale=2]
		
		\draw [fill=circ_col] (-1.5, 0) node [anchor=east] {1} circle (1pt);
		\draw [fill=circ_col] (-0.5, 0) node [anchor=west] {2} circle (1pt);
		\draw [fill=circ_col] (0, 0.86) node [above] {3} circle (1pt);
		\draw [fill=circ_col] (0, -0.86) node [below] {4} circle (1pt);
		\draw [fill=circ_col] (0.5, 0) node [anchor=east] {5} circle (1pt);
		\draw [fill=circ_col] (1.5, 0) node [anchor=west] {6} circle (1pt);
		
		\draw [-, line width = 1pt] (-1.5,0) -- (0, -0.86) -- (1.5, 0) -- (0, 0.86) -- (-1.5, 0);
		
		\draw [-, line width = 1pt] (-1.5,0) -- node [above] {2} (-0.5,0);
		\draw [-, line width = 1pt] (0, 0.86) -- node [above] {1} (-0.5,0);
		\draw [-, line width = 1pt] (0, -0.86) -- node [anchor=east] {3} (-0.5,0);
		\draw [-, line width = 1pt] (0, -0.86) -- node [anchor=east] {4} (0, 0.86);
		\draw [-, line width = 1pt] (1.5,0) -- node [above] {6} (0.5,0);
		\draw [-, line width = 1pt] (0, 0.86) -- node [above] {5} (0.5,0);
		\draw [-, line width = 1pt] (0, -0.86) -- node [anchor=west] {7} (0.5,0);
		
		\end{tikzpicture}
		
	\end{center}
	\setlength{\abovecaptionskip}{-0.15cm}
	\caption{The arrangement of planes $U_{3 \cup 3}$.}\label{fig_computation_U_3_cup_3}
\end{figure}

\begin{thm}\label{thm_computation_U_3_cup_3}
	If $ X $ degenerates to $ U_{3\cup 3} $, then $\pi_1(\Xgal)$ is trivial.
\end{thm}

\begin{proof}
	The branch curve $S_0$ in $\mathbb{CP}^2$ is an arrangement of seven lines. We regenerate each vertex in turn and compute the group $G$.

	\uOnePoint{1}{2}{U3cup3-vert1}
	\uOnePoint{6}{6}{U3cup3-vert6}
	\uThreePointInner{2}{1}{2}{3}{U3cup3-vert2}
	\uThreePointOuter{3}{1}{4}{5}{U3cup3-vert3}
	\uThreePointOuter{4}{3}{4}{7}{U3cup3-vert4}
	\uThreePointInner{5}{5}{6}{7}{U3cup3-vert5}
	We also have the following parasitic and projective relations:
	\uParasit{1}{6}{U3cup3-parasit-1-6}
	\uParasit{1}{7}{U3cup3-parasit-1-7}
	\uParasit{2}{4}{U3cup3-parasit-2-4}
	\uParasit{2}{5}{U3cup3-parasit-2-5}
	\uParasit{2}{6}{U3cup3-parasit-2-6}
	\uParasit{2}{7}{U3cup3-parasit-2-7}
	\uParasit{3}{5}{U3cup3-parasit-3-5}
	\uParasit{3}{6}{U3cup3-parasit-3-6}
	\uParasit{4}{6}{U3cup3-parasit-4-6}
	\uProjRel{25-proj}{7}{6}{5}{4}{3}{2}{1}
	
By Lemma \ref{3pt-in-bigmid}, we get $\ug{1}=\ug{1'}, \ug{3}=\ug{3'}$ and $\ug{5}=\ug{5'}, \ug{7}=\ug{7'}$. Then by Lemma
\ref{3pt-out-smallmid}, we get $\ug{4}=\ug{4'}$.

$\Ggal$ is thus generated by $\set{\ug{i} | i=1,\dots,7}$ with the following relations:
	\uGammaSq{U3cup3-gamma-sq}{1}{2}{3}{4}{5}{6}{7}
	\ubegineq{U3cup3-simpl-3}
	\begin{split}
	\trip{1}{2} = \trip{1}{3} = \trip{1}{4} = \trip{2}{3} = \trip{3}{4} &= \\
	= \trip{4}{5} = \trip{4}{7}= \trip{5}{6} = \trip{5}{7} = \trip{6}{7} &=e
	\end{split}
	\uendeq
	\ubegineq{U3cup3-simpl-4}
	\begin{split}
	\comm{1}{5} = \comm{1}{6} = \comm{1}{7} = \comm{2}{4} = \comm{2}{5} = \comm{2}{6} &= \\
= \comm{2}{7} = \comm{3}{5} = \comm{3}{6} = \comm{3}{7} = \comm{4}{6} &= e
	\end{split}
	\uendeq
	\ubegineq{U3cup3-simpl-5}
	\ug{3}=\ug{1}\ug{2}\ug{1}
	\uendeq
     \ubegineq{U3cup3-simpl-6}
	\ug{7}=\ug{5}\ug{6}\ug{5}.
	\uendeq
By Remark 2.8 in \cite{RTV}, we get also the relations
\ubegineq{U3cup3-simpl-7}
	[\ug{4},\ug{1}\ug{3}\ug{1}]=[\ug{4},\ug{5}\ug{7}\ug{5}] = e.
	\uendeq
By Theorem 2.3 in \cite{RTV} we have $\Ggal \cong S_6$, therefore $\pi_1(\Xgal)$ is trivial.

\end{proof}

\uChernSummary{14}{44}{24}{14}{16}{11}{-2\cdot}

%+--------------------------------------+
%|                               		    |
%|                CASE  U_4_1     		    |
%|                                      	|
%+--------------------------------------+

\subsection{{$ X $ degenerates to  $ U_{4,1} $}}\label{section_computation_U_4_1_appendix}
\begin{figure}[H]
	\begin{center}
		
		\definecolor{circ_col}{rgb}{0,0,0}
		\begin{tikzpicture}[x=1cm,y=1cm,scale=2]
		
		\draw [fill=circ_col] (0,-1) node [anchor=east] {1} circle (1pt);
		\draw [fill=circ_col] (1,-1) node [below] {2} circle (1pt);
		\draw [fill=circ_col] (-1,0) node [anchor=east] {3} circle (1pt);
		\draw [fill=circ_col] (0,0) node [anchor=north east] {4} circle (1pt);
		\draw [fill=circ_col] (1,0) node [anchor=west] {5} circle (1pt);
		\draw [fill=circ_col] (0,1) node [above] {6} circle (1pt);
		\draw [fill=circ_col] (1,1) node [above] {7} circle (1pt);
		
		\draw [-, line width = 1pt] (0,-1) -- (1,-1) -- (1,0) -- (1,1) -- (0,1) -- (-1,0) -- (0,-1);
		
		\draw [-, line width = 1pt] (0,-1) -- node [above]{1} (1,0);
		\draw [-, line width = 1pt] (0,-1) -- node [anchor=west]{2} (0,0);
		\draw [-, line width = 1pt] (-1,0) -- node [above]{3} (0,0);
		\draw [-, line width = 1pt] (0,0) -- node [above]{4} (1,0);
		\draw [-, line width = 1pt] (0,0) -- node [anchor=east]{5} (0,1);
		\draw [-, line width = 1pt] (0,1) -- node [above]{6} (1,0);
		\end{tikzpicture}
	\end{center}
	\setlength{\abovecaptionskip}{-0.15cm}
	\caption{The arrangement of planes $ U_{4,1} $.}\label{fig_computation_U_4_1_appendix}
\end{figure}

\begin{thm}
	If $ X $ degenerates to $ U_{4,1} $, then $\pi_1(\Xgal)$ is trivial.	
\end{thm}

\begin{proof}
The branch curve $S_0$ in $\mathbb{CP}^2$ is an arrangement of six lines. We regenerate each vertex in turn and compute the group $G$.

	\uOnePoint{3}{3}{U41-vert3}
	\uTwoPoint{1}{1}{2}{U41-vert1}
    \uTwoPoint{6}{5}{6}{U41-vert6}
	\uThreePointOuter{5}{1}{4}{6}{U41-vert5}	
	\uFourPointInner{4}{2}{3}{4}{5}{U41-vert4}
	We also have the following parasitic and projective relations:
	\uParasit{1}{3}{U41-parasit-1-3}
	\uParasit{1}{5}{U41-parasit-1-5}
	\uParasit{2}{6}{U41-parasit-2-6}
	\uParasit{3}{6}{U41-parasit-3-6}
	\uProjRel{U41-proj}{6}{5}{4}{3}{2}{1}

By  substituting \eqref{U41-vert3} and equating \eqref{U41-vert4-11} and \eqref{U41-vert4-12}, we get $\ug{4}=\ug{4'}$. By Lemma \ref{3pt-out-smallmid}, we have $\ug{1}=\ug{1'}$ and $\ug{6}=\ug{6'}$. Then by Lemma \ref{2pt-equal}, we get $\ug{2}=\ug{2'}$ and $\ug{5}=\ug{5'}$.

$\Ggal$ is thus generated by $\set{\ug{i} | i=1,\dots,6}$ with the following relations:
	\uGammaSq{U41-gamma-sq}{1}{2}{3}{4}{5}{6}
	\ubegineq{U41-simpl-3}
	\trip{1}{2} = \trip{1}{4} = \trip{2}{3} = \trip{2}{4} = \trip{3}{5} = \trip{4}{5}= \trip{4}{6} = \trip{5}{6} = e
	\uendeq
	\ubegineq{U41-simpl-4}
	\comm{1}{3} = \comm{1}{5} = \comm{1}{6} = \comm{2}{5} = \comm{2}{6} = \comm{3}{4} = \comm{3}{6} = e
	\uendeq
	\ubegineq{U41-simpl-5}
	\ug{2}\ug{3}\ug{2}=\ug{5}\ug{4}\ug{5}.
	\uendeq
By Remark 2.8 in \cite{RTV}, we get also the relations
\ubegineq{U41-simpl-6}
	[\ug{1},\ug{2}\ug{4}\ug{2}]=[\ug{6},\ug{4}\ug{5}\ug{4}] = e.
	\uendeq
By Theorem 2.3 in \cite{RTV} we have $\Ggal \cong S_6$, therefore $\pi_1(\Xgal)$ is trivial.

\end{proof}

\uChernSummary{9}{24}{24}{12}{9}{\frac{11}{2}}{-\frac{2}{3}\cdot}

%+--------------------------------------+
%|                             		    |
%|                CASE  U_4_2  		    |
%|                                 	    |
%+--------------------------------------+

\subsection{{$ X $ degenerates to $U_{4,2}$}}\label{section_computation_U_4_2}
\begin{figure}[H]
	\begin{center}
		
		\definecolor{circ_col}{rgb}{0,0,0}
		\begin{tikzpicture}[x=1cm,y=1cm,scale=2]
		
		\draw [fill=circ_col] (-1,-1) node [anchor=east] {1} circle (1pt);
		\draw [fill=circ_col] (0,-1) node [below] {2} circle (1pt);
		\draw [fill=circ_col] (-1,0) node [anchor=east] {3} circle (1pt);
		\draw [fill=circ_col] (0,0) node [anchor=north east] {4} circle (1pt);
		\draw [fill=circ_col] (1,0) node [below] {5} circle (1pt);
		\draw [fill=circ_col] (0,1) node [above] {6} circle (1pt);
		\draw [fill=circ_col] (1,1) node [above] {7} circle (1pt);
		
		\draw [-, line width = 1pt] (-1,-1) -- (0,-1) -- (1,0) -- (1,1) -- (0,1) -- (-1,0) -- (-1,-1);
		
		\draw [-, line width = 1pt] (0,-1) -- node [above]{1} (-1,0);
		\draw [-, line width = 1pt] (0,-1) -- node [anchor=west]{2} (0,0);
		\draw [-, line width = 1pt] (-1,0) -- node [above]{3} (0,0);
		\draw [-, line width = 1pt] (0,0) -- node [above]{4} (1,0);
		\draw [-, line width = 1pt] (0,0) -- node [anchor=east]{5} (0,1);
		\draw [-, line width = 1pt] (0,1) -- node [above]{6} (1,0);
		\end{tikzpicture}
	\end{center}
	\setlength{\abovecaptionskip}{-0.15cm}
	\caption{The arrangement of planes $U_{4,2}$.}\label{fig_computation_U_4_2}
\end{figure}

Consider $ X $ that degenerates to $ U_{4,2} $.
The branch curve $S_0$ in $\mathbb{CP}^2$ is an arrangement of six lines. We regenerate each vertex in turn and compute the group $G$.

\uTwoPoint{2}{1}{2}{27-vert2}
\uTwoPoint{3}{1}{3}{27-vert3}
\uTwoPoint{5}{4}{6}{27-vert5}
\uTwoPoint{6}{5}{6}{27-vert6}
\uFourPointInner{4}{2}{3}{4}{5}{27-vert4}
We also have the following parasitic and projective relations:
\uParasit{1}{4}{27-parasit-1-4}
\uParasit{1}{5}{27-parasit-1-5}
\uParasit{1}{6}{27-parasit-1-6}
\uParasit{2}{6}{27-parasit-2-6}
\uParasit{3}{6}{27-parasit-3-6}
\uProjRel{27-proj}{6}{5}{4}{3}{2}{1}

As mentioned in Section \ref{sec:results}, it is a challenging problem to determine if the kernel of the natural surjection from $ \widetilde{G} $ to $ S_6 $ is trivial or not.
This question should be addressed once the relevant algebraic tools are developed.
We note that the use of the computer aided algebra system MAGMA, have not yielded satisfactory results.

\uChernSummary{8}{24}{24}{12}{9}{5}{-\frac{1}{3}\cdot}

%+--------------------------------------+
%|                               		    |
%|                CASE  U_4_3     		    |
%|                                      	    |
%+--------------------------------------+

\subsection{{$ X $ degenerates to $ U_{4,3} $}}\label{section_computation_U_4_3}
\begin{figure}[H]
	\begin{center}
		
		\definecolor{circ_col}{rgb}{0,0,0}
		\begin{tikzpicture}[x=1cm,y=1cm,scale=2]
		
		\draw [fill=circ_col] (0,-1) node [anchor=east] {1} circle (1pt);
		\draw [fill=circ_col] (-1,0) node [anchor=east] {2} circle (1pt);
		\draw [fill=circ_col] (0,0) node [anchor=north east] {3} circle (1pt);
		\draw [fill=circ_col] (1,0) node [anchor=north west] {4} circle (1pt);
		\draw [fill=circ_col] (2,0) node [anchor=west] {5} circle (1pt);
		\draw [fill=circ_col] (0,1) node [above] {6} circle (1pt);
		\draw [fill=circ_col] (1,1) node [above] {7} circle (1pt);
		
		\draw [-, line width = 1pt] (0,-1) -- (1,0) --  (2,0) -- (1,1) -- (0,1) -- (-1,0) -- (0,-1);
		
		\draw [-, line width = 1pt] (1,0) -- node [anchor=west]{6} (1,1);
		\draw [-, line width = 1pt] (0,-1) -- node [anchor=west]{1} (0,0);
		\draw [-, line width = 1pt] (-1,0) -- node [above]{2} (0,0);
		\draw [-, line width = 1pt] (0,0) -- node [above]{3} (1,0);
		\draw [-, line width = 1pt] (0,0) -- node [anchor=east]{4} (0,1);
		\draw [-, line width = 1pt] (0,1) -- node [above]{5} (1,0);
		\end{tikzpicture}
	\end{center}
	\setlength{\abovecaptionskip}{-0.15cm}
	\caption{The arrangement of planes $ U_{4,3} $.}\label{fig_computation_U_4_3}
\end{figure}
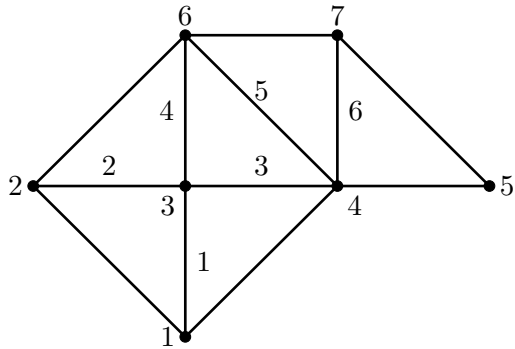

\begin{thm}\label{thm_computation_U_4_3}
If $ X $ degenerates to $ U_{4,3} $, then $\pi_1(\Xgal)$ is trivial.	
\end{thm}

\begin{proof}
The branch curve $S_0$ in $\mathbb{CP}^2$ is an arrangement of six lines. We regenerate each vertex in turn and compute the group $G$.

	\uOnePoint{1}{1}{29-vert1}
	\uOnePoint{2}{2}{29-vert2}
	\uOnePoint{7}{6}{29-vert7}
	\uTwoPoint{6}{4}{5}{29-vert6}
	\uThreePointOuter{4}{3}{5}{6}{29-vert4}
	\uFourPointInner{3}{1}{2}{3}{4}{29-vert3}
	We also have the following parasitic and projective relations:
	\uParasit{1}{5}{29-parasit-1-5}
	\uParasit{1}{6}{29-parasit-1-6}
	\uParasit{2}{5}{29-parasit-2-5}
	\uParasit{2}{6}{29-parasit-2-6}
	\uParasit{4}{6}{29-parasit-4-6}
	\uProjRel{29-proj}{6}{5}{4}{3}{2}{1}

Substituting \eqref{29-vert1} and \eqref{29-vert2} and equating  \eqref{29-vert3-9} and \eqref{29-vert3-10}, we get $\ug{3}=\ug{3'}$.
By Lemma \ref{3pt-out-smallmid}, we have $\ug{5}=\ug{5'}$. Then by Lemma \ref{2pt-equal}, we get $\ug{4}=\ug{4'}$.

$\Ggal$ is thus generated by $\set{\ug{i} | i=1,\dots,6}$ with the following relations:
	\uGammaSq{29-gamma-sq}{1}{2}{3}{4}{5}{6}
	\ubegineq{29-simpl-3}
	\trip{1}{2} = \trip{1}{3} = \trip{2}{4} = \trip{3}{4} = \trip{3}{5} = \trip{4}{5}= \trip{5}{6} = e
	\uendeq
	\ubegineq{29-simpl-4}
	\comm{1}{4} = \comm{1}{5} = \comm{1}{6} = \comm{2}{3} = \comm{2}{5} = \comm{2}{6} = \comm{3}{6} = \comm{4}{6} = e
	\uendeq
	\ubegineq{29-simpl-5}
	\ug{1}\ug{2}\ug{1}=\ug{4}\ug{3}\ug{4}.
	\uendeq
By Remark 2.8 in \cite{RTV}, we get also the relation
\ubegineq{29-simpl-6}
	[\ug{5}, \ug{4}\ug{3}\ug{4}]=e.
	\uendeq
By Theorem 2.3 in \cite{RTV} we have $\Ggal \cong S_6$, therefore $\pi_1(\Xgal)$ is trivial.

\end{proof}

\uChernSummary{10}{28}{21}{12}{9}{\frac{13}{2}}{-\frac{4}{3}\cdot}

%+--------------------------------------+
%|                                      |
%|                CASE  U_{4 \cup 3,1}  |
%|                                      |
%+--------------------------------------+
\subsection{{$ X $ degenerates to $ U_{4 \cup 3,1} $}}\label{section_computation_U_4_cup_3_1}

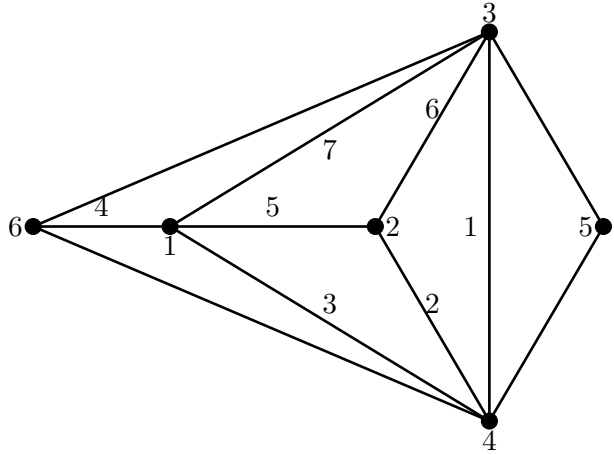
\begin{figure}[H]
	\begin{center}
		
		\definecolor{circ_col}{rgb}{0,0,0}
		\begin{tikzpicture}[x=1cm,y=1cm,scale=3]

		\draw [fill=circ_col] (-1.4, 0) node [below] {1} circle (1pt);
		\draw [fill=circ_col] (-0.5, 0) node [anchor=west] {2} circle (1pt);
		\draw [fill=circ_col] (0, 0.86) node [above] {3} circle (1pt);
		\draw [fill=circ_col] (0, -0.86) node [below] {4} circle (1pt);
		\draw [fill=circ_col] (0.5, 0) node [anchor=east] {5} circle (1pt);
		\draw [fill=circ_col] (-2, 0) node [anchor=east] {6} circle (1pt);
		
		\draw [-, line width = 1pt] (-2,0) -- (0, -0.86) -- (0.5, 0) -- (0, 0.86) -- (-2, 0);
		
		\draw [-, line width = 1pt] (-1.4,0) -- node [above] {5} (-0.5,0);
		\draw [-, line width = 1pt] (0, 0.86) -- node [above] {6} (-0.5,0);
		\draw [-, line width = 1pt] (0, -0.86) -- node [above] {2} (-0.5,0);
		\draw [-, line width = 1pt] (0, -0.86) -- node [anchor=east] {1} (0, 0.86);
		\draw [-, line width = 1pt] (-2,0) -- node [above] {4} (-1.4,0);
		\draw [-, line width = 1pt] (0, 0.86) -- node [below] {7} (-1.4,0);
		\draw [-, line width = 1pt] (0, -0.86) -- node [above] {3} (-1.4,0);
		\end{tikzpicture}
		
	\end{center}
	\setlength{\abovecaptionskip}{-0.15cm}
	\caption{The arrangement of planes $U_{4 \cup 3,1}$.}\label{fig_computation_U_4_cup_3_1_appendix}
\end{figure}

\begin{thm}\label{thm_computation_U_4_cup_3_1}
If $ X $ degenerates to $ U_{4 \cup 3,1} $, then $\pi_1(\Xgal)$ is trivial.	
\end{thm}

\begin{proof}
	
	See Theorem \ref{thm_computation_U_4_cup_3_1_paper}.

\end{proof}

\uChernSummary{13}{36}{30}{14}{16}{\frac{19}{2}}{-}

%+--------------------------------------+
%|                                      |
%|                CASE  U_{4 \cup 3,2}  |
%|                                      |
%+--------------------------------------+
\subsection{{$ X $ degenerates to $ U_{4 \cup 3,2} $}}\label{section_computation_U_4_cup_3_2}

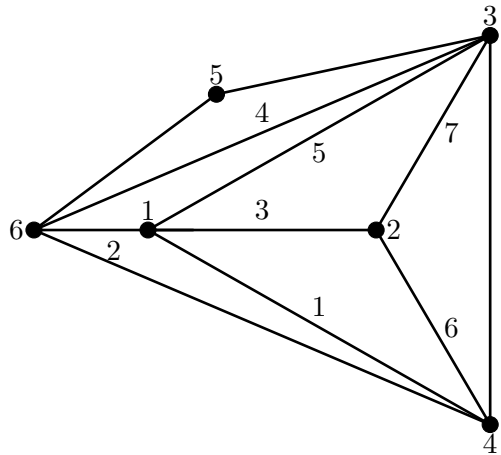
\begin{figure}[H]
	\begin{center}
		
		\definecolor{circ_col}{rgb}{0,0,0}
		\begin{tikzpicture}[x=1cm,y=1cm,scale=3]

		\draw [fill=circ_col] (-1.5, 0) node [above] {1} circle (1pt);
		\draw [fill=circ_col] (-0.5, 0) node [anchor=west] {2} circle (1pt);
		\draw [fill=circ_col] (0, 0.86) node [above] {3} circle (1pt);
		\draw [fill=circ_col] (0, -0.86) node [below] {4} circle (1pt);
		\draw [fill=circ_col] (-1.2, 0.6) node [above] {5} circle (1pt);
		\draw [fill=circ_col] (-2, 0) node [anchor=east] {6} circle (1pt);
		
		\draw [-, line width = 1pt] (-2,0) -- (0, -0.86) -- (0, 0.86) -- (-1.2, 0.6) -- (-2, 0);
		
		\draw [-, line width = 1pt] (-1.5,0) -- node [above] {3} (-0.5,0);
		\draw [-, line width = 1pt] (0, 0.86) -- node [anchor=west] {7} (-0.5,0);
		\draw [-, line width = 1pt] (0, -0.86) -- node [anchor=west] {6} (-0.5,0);
		\draw [-, line width = 1pt] (-2, 0) -- node [above] {4} (0, 0.86);
		\draw [-, line width = 1pt] (-1.3,0) -- node [below] {2} (-2,0);
		\draw [-, line width = 1pt] (0, 0.86) -- node [below] {5} (-1.5,0);
		\draw [-, line width = 1pt] (0, -0.86) -- node [above] {1} (-1.5,0);
		\end{tikzpicture}
		
	\end{center}
	\setlength{\abovecaptionskip}{-0.15cm}
	\caption{The arrangement of planes $U_{4 \cup 3,2}$.}\label{fig_computation_U_4_cup_3_2}
\end{figure}

\begin{thm}\label{thm_computation_U_4_cup_3_2}
If $ X $ degenerates to $ U_{4 \cup 3,2} $, then $\pi_1(\Xgal)$ is trivial.	
\end{thm}

\begin{proof}
The branch curve $S_0$ in $\mathbb{CP}^2$ is an arrangement of seven lines. We regenerate each vertex in turn and compute the group $G$.

	\uTwoPoint{4}{1}{6}{U4cup3-2-vert4}
	\uTwoPoint{6}{2}{4}{U4cup3-2-vert6}
    \uThreePointInner{2}{3}{6}{7}{U4cup3-2-vert2}
	\uThreePointOuter{3}{4}{5}{7}{U4cup3-2-vert3}
	\uFourPointInner{1}{1}{2}{3}{5}{U4cup3-2-vert1}
	We also have the following parasitic and projective relations:
	\uParasit{1}{4}{U4cup3-2-parasit-1-4}
	\uParasit{1}{7}{U4cup3-2-parasit-1-7}
	\uParasit{2}{6}{U4cup3-2-parasit-2-6}
	\uParasit{2}{7}{U4cup3-2-parasit-2-7}
	\uParasit{3}{4}{U4cup3-2-parasit-3-4}
    \uParasit{4}{6}{U4cup3-2-parasit-4-6}
    \uParasit{5}{6}{U4cup3-2-parasit-5-6}
	\uProjRel{U4cup3-2-proj}{7}{6}{5}{4}{3}{2}{1}

By Lemma \ref{3pt-in-bigmid}, we have $\ug{3}=\ug{3'}$. Using this result and equating \eqref{U4cup3-2-vert1-9} and \eqref{U4cup3-2-vert1-10}, we get $\ug{2}=\ug{2'}$. By Lemma \ref{2pt-equal}, we have $\ug{4}=\ug{4'}$. Then by Lemma \ref{3pt-out-smallmid}, we get $\ug{5}=\ug{5'}$ and $\ug{7}=\ug{7'}$. Using Lemma \ref{3pt-in-bigmid}, we get $\ug{6}=\ug{6'}$. And again by Lemma \ref{2pt-equal} we have $\ug{1}=\ug{1'}$.

$\Ggal$ is thus generated by $\set{\ug{i} | i=1,\dots,7}$ with the following relations:
	\uGammaSq{U4cup3-2-gamma-sq}{1}{2}{3}{4}{5}{6}{7}
\begin{align}	
\trip{1}{2} = \trip{1}{3} = \trip{1}{6} = \trip{2}{4} = \trip{2}{5} = \trip{3}{5}&=  \\
=\trip{3}{6} =\trip{3}{7} = \trip{4}{5}= \trip{5}{7} = \trip{6}{7} &=  e \nonumber
\end{align}	
	\begin{align}\label{U4cup3-2-simpl-4}
	\comm{1}{4} = \comm{1}{5} = \comm{1}{7} = \comm{2}{3} = \comm{2}{6} &= \\ =\comm{2}{7} = \comm{3}{4} = \comm{4}{6} = \comm{4}{7} = \comm{5}{6} &= e \nonumber
	\end{align}
     \ubegineq{U4cup3-2-simpl-5}
	\ug{7}= \ug{3}\ug{6}\ug{3}
	\uendeq
	\ubegineq{U4cup3-2-simpl-6}
	\ug{1}\ug{2}\ug{1}=\ug{5}\ug{3}\ug{5}.
	\uendeq
By Remark 2.8 in \cite{RTV}, we get also the relations
    \ubegineq{U4cup3-2-simpl-7}
	[\ug{4},\ug{2}\ug{5}\ug{2}]= [\ug{6},\ug{1}\ug{3}\ug{1}] = [\ug{7},\ug{5}\ug{3}\ug{5}] = e.
	\uendeq
By Theorem 2.3 in \cite{RTV} we have $\Ggal \cong S_6$, therefore $\pi_1(\Xgal)$ is trivial.

\end{proof}

\uChernSummary{12}{36}{30}{14}{16}{9}{-\frac{2}{3}\cdot}

%+--------------------------------------+
%|                                      |
%|                CASE  U_{4 \cup 4}    |
%|                                      |
%+--------------------------------------+

\subsection{{$ X $ degenerates to $ U_{4 \cup 4} $}}\label{section_computation_U_4_cup_4}

\begin{figure}[H]
	\begin{center}
		
		\definecolor{circ_col}{rgb}{0,0,0}
		\begin{tikzpicture}[x=1cm,y=1cm,scale=2]
		
		\draw [fill=circ_col] (-1.5, 0) node [anchor=east] {1} circle (1pt);
		\draw [fill=circ_col] (-0.5, 0) node [below] {2} circle (1pt);
		\draw [fill=circ_col] (0, 0.86) node [above] {3} circle (1pt);
		\draw [fill=circ_col] (0, -0.86) node [below] {4} circle (1pt);
		\draw [fill=circ_col] (0.5, 0) node [below] {5} circle (1pt);
		\draw [fill=circ_col] (1.5, 0) node [anchor=west] {6} circle (1pt);
		
		\draw [-, line width = 1pt] (-1.5,0) -- (0, -0.86) -- (1.5, 0) -- (0, 0.86) -- (-1.5, 0);
		
		\draw [-, line width = 1pt] (-1.5,0) -- node [above] {2} (-0.5,0);
		\draw [-, line width = 1pt] (0, 0.86) -- node [above] {5} (-0.5,0);
		\draw [-, line width = 1pt] (0, -0.86) -- node [above] {1} (-0.5,0);
		\draw [-, line width = 1pt] (-0.5, 0) -- node [above] {4} (0.5, 0);
		\draw [-, line width = 1pt] (1.5,0) -- node [above] {6} (0.5,0);
		\draw [-, line width = 1pt] (0, 0.86) -- node [above] {7} (0.5,0);
		\draw [-, line width = 1pt] (0, -0.86) -- node [above] {3} (0.5,0);
		
		\end{tikzpicture}
		
	\end{center}
	\setlength{\abovecaptionskip}{-0.15cm}
	\caption{The arrangement of planes $U_{4 \cup 4}$.}\label{fig_computation_U_4_cup_4}
\end{figure}
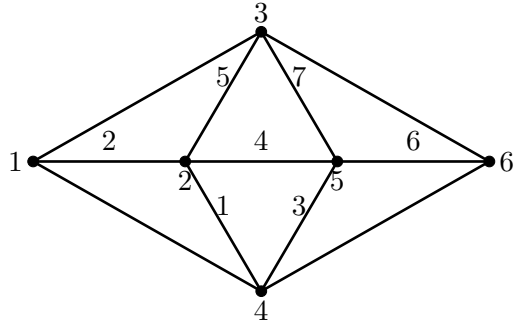

\begin{thm}\label{thm_computation_U_4_cup_4}
	If $ X $ degenerates to $ U_{4\cup 4} $, then $ \pi_1(\Xgal) $ is trivial.
\end{thm}

\begin{proof}
The branch curve $S_0$ in $\mathbb{CP}^2$ is an arrangement of seven lines. We regenerate each vertex in turn and compute the group $G$.

	\uOnePoint{1}{2}{U4cup4-vert1}
	\uOnePoint{6}{6}{U4cup4-vert6}
	\uTwoPoint{3}{5}{7}{U4cup4-vert3}
	\uTwoPoint{4}{1}{3}{U4cup4-vert4}
	\uFourPointInner{2}{1}{2}{4}{5}{U4cup4-vert2}
	\uFourPointInner{5}{3}{4}{6}{7}{U4cup4-vert5}
	We also have the following parasitic and projective relations:
	\uParasit{1}{6}{U4cup4-parasit-1-6}
	\uParasit{1}{7}{U4cup4-parasit-1-7}
	\uParasit{2}{3}{U4cup4-parasit-2-3}
	\uParasit{2}{6}{U4cup4-parasit-2-6}
	\uParasit{2}{7}{U4cup4-parasit-2-7}
	\uParasit{3}{5}{U4cup4-parasit-3-5}
	\uParasit{5}{6}{U4cup4-parasit-5-6}
	\uProjRel{29-proj}{7}{6}{5}{4}{3}{2}{1}

Substituting \eqref{U4cup4-vert1} and equating  \eqref{U4cup4-vert2-9} and \eqref{U4cup4-vert2-10}, we get $\ug{4}=\ug{4'}$.

Using computer algebra system (we used MAGMA), we can express all the $ \Gamma_i' $'s by $ \{ \Gamma_j | j=1,\dots,7 \} $ and get all the defining relations of $ S_6 $.
Then it is a straight forward verification that all the relations in $ G $ hold in $ S_6 $ so $ G/\langle \Gamma_i^2 \rangle \cong S_6 $, and $ \pi_1(\Xgal) $ is trivial.

\end{proof}

\uChernSummary{12}{36}{30}{14}{16}{9}{-\frac{2}{3}\cdot}

%+--------------------------------------+
%|                                      |
%|                CASE  U_5      	    |
%|                                      |
%+--------------------------------------+

\subsection{{$ X $ degenerates to $ U_5 $}}\label{section_computation_U_5}

\begin{figure}[H]
	\begin{center}
		
		\definecolor{circ_col}{rgb}{0,0,0}
		\begin{tikzpicture}[x=1cm,y=1cm,scale=2]
		
		\draw [fill=circ_col] (1,0) node [below] {1} circle (1pt);
		\draw [fill=circ_col] (0,1) node [below] {2} circle (1pt);
		\draw [fill=circ_col] (1,1) node [anchor=north west] {3} circle (1pt);
		\draw [fill=circ_col] (2,1) node [anchor=west] {4} circle (1pt);
		\draw [fill=circ_col] (0,2) node [anchor=east] {5} circle (1pt);
		\draw [fill=circ_col] (1,2) node [above] {6} circle (1pt);
		\draw [fill=circ_col] (0,2.7) node [above] {7} circle (1pt);
		
		\draw [-, line width = 1pt] (1,0) -- (2,1) -- (1,2) -- (0,2.7) -- (0,2) -- (0,1) -- (1,0);
		
		\draw [-, line width = 1pt] (1,1) -- node [anchor=east]{1} (1,0);
		\draw [-, line width = 1pt] (1,1) -- node [anchor=west]{5} (1,2);
		\draw [-, line width = 1pt] (1,1) -- node [below]{2} (0,1);
		\draw [-, line width = 1pt] (1,1) -- node [above]{3} (2,1);
		\draw [-, line width = 1pt] (1,1) -- node [below]{4} (0,2);
		\draw [-, line width = 1pt] (1,2) -- node [below]{6} (0,2);
		\end{tikzpicture}
		
	\end{center}
	\setlength{\abovecaptionskip}{-0.15cm}
	\caption{The arrangement of planes $U_5$.}\label{fig_computation_U_5}
\end{figure}
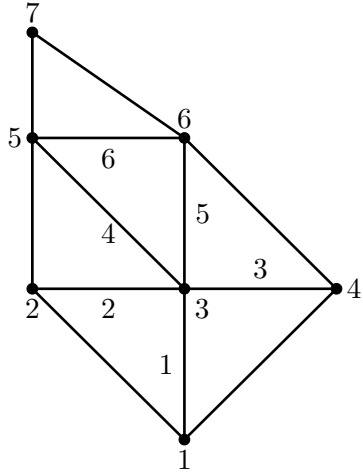

\begin{thm}\label{thm_computation_U_5}
	If $ X $ degenerates to $ U_{5}$, then $\pi_1(\Xgal)$ is trivial.
\end{thm}

\begin{proof}
	The branch curve $S_0$ in $\mathbb{CP}^2$ is an arrangement of six lines. We regenerate each vertex in turn and compute the group $G$.

	\uOnePoint{1}{1}{15-vert1}
	\uOnePoint{2}{2}{15-vert2}
	\uOnePoint{4}{3}{15-vert4}
	\uTwoPoint{5}{4}{6}{15-vert5}
	\uTwoPoint{6}{5}{6}{15-vert6}
	\uFivePointInner{3}{1}{2}{3}{4}{5}{15-vert3}
	We also have the following parasitic and projective relations:
	\uParasit{1}{6}{15-parasit-1-6}
	\uParasit{2}{6}{15-parasit-2-6}
	\uParasit{3}{6}{15-parasit-3-6}
	\uProjRel{15-proj}{6}{5}{4}{3}{2}{1}

Using \eqref{15-vert1}, \eqref{15-vert2}, \eqref{15-vert4}, \eqref{15-vert3-8}, \eqref{15-vert3-14} and equating \eqref{15-vert3-9} and \eqref{15-vert3-15}, we get $\ug{5}=\ug{5'}$. By Lemma \ref{2pt-equal}, $\ug{6}=\ug{6'}$ and $\ug{4}=\ug{4'}$.

$\Ggal$ is thus generated by $\set{\ug{i} | i=1,\dots,6}$ with the following relations:
	\uGammaSq{15-gamma-sq}{1}{2}{3}{4}{5}{6}
	\ubegineq{15-simpl-3}	
\trip{1}{2} = \trip{1}{3} = \trip{2}{4} = \trip{3}{5} = \trip{4}{5} = \trip{4}{6}= \trip{5}{6} =  e
\uendeq
	\ubegineq{15-simpl-4}
	\comm{1}{4} = \comm{1}{5} = \comm{1}{6} = \comm{2}{3} = \comm{2}{5} = \comm{2}{6} = \comm{3}{4} = \comm{3}{6} = e
	\uendeq
     \ubegineq{15-simpl-5}
	\ug{1}\ug{2}\ug{1}=\ug{4}\ug{5}\ug{3}\ug{5}\ug{4}.
	\uendeq
By Remark 2.8 in \cite{RTV}, we get also the relation
    \ubegineq{15-simpl-7}
	[\ug{6},\ug{4}\ug{5}\ug{4}] = e.
	\uendeq
By Theorem 2.3 in \cite{RTV} we have $\Ggal \cong S_6$, therefore $\pi_1(\Xgal)$ is trivial.

\end{proof}

\uChernSummary{10}{24}{24}{12}{9}{6}{-}

%+--------------------------------------+
%|                               		    |
%|                CASE  U_5_cup_3     		    |
%|                                      	    |
%+--------------------------------------+

\subsection{{$ X $ degenerates to $ U_{5\cup 3} $}}\label{section_computation_U_5_cup_3}

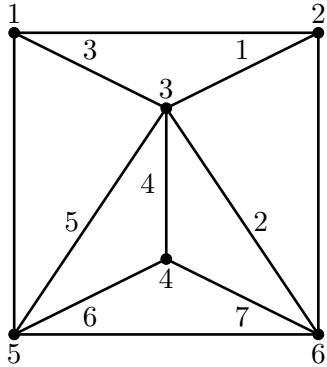
\begin{figure}[H]
	\begin{center}
		
		\definecolor{circ_col}{rgb}{0,0,0}
		\begin{tikzpicture}[x=1cm,y=1cm,scale=2]
		\draw [fill=circ_col] (-1, 1) node [above] {1} circle (1pt);
		\draw [fill=circ_col] (1, 1) node [above] {2} circle (1pt);
		\draw [fill=circ_col] (0,0.5) node [above] {3} circle (1pt);
		\draw [fill=circ_col] (0,-0.5) node [below] {4} circle (1pt);
		\draw [fill=circ_col] (-1, -1) node [below] {5} circle (1pt);
		\draw [fill=circ_col] (1, -1) node [below] {6} circle (1pt);
		
		\draw [-, line width=1pt] (-1,-1) -- (-1, 1) -- (1,1) -- (1,-1) -- (-1,-1);
		
		\draw [-, line width = 1pt] (0,0.5) -- node [above] {3} (-1,1);
		\draw [-, line width = 1pt] (0,0.5) -- node [above] {1} (1,1);
		\draw [-, line width = 1pt] (0,0.5) -- node [anchor=east] {4} (0, -0.5);
		\draw [-, line width = 1pt] (0,0.5) -- node [anchor=east] {5} (-1, -1);
		\draw [-, line width = 1pt] (0,0.5) -- node [anchor=west] {2} (1, -1);
		\draw [-, line width = 1pt] (0,-0.5) -- node [below] {6} (-1, -1);
		\draw [-, line width = 1pt] (0,-0.5) -- node [below] {7} (1, -1);
		\end{tikzpicture}
	\end{center}
	\setlength{\abovecaptionskip}{-0.15cm}
	\caption{The arrangement of planes $U_{5\cup 3}$.}\label{fig_computation_U_5_cup_3}
\end{figure}

\begin{thm}\label{thm_computation_U_5_cup_3}
If $ X $ degenerates to $ U_{5\cup 3} $, then $\pi_1(\Xgal)$ is trivial.
\end{thm}

\begin{proof}
The branch curve $S_0$ in $\mathbb{CP}^2$ is an arrangement of seven lines. We regenerate each vertex in turn and compute the group $G$.
	
	\uOnePoint{1}{3}{U5cup3-vert1}
	\uOnePoint{2}{1}{U5cup3-vert2}
	\uTwoPoint{5}{5}{6}{U5cup3-vert5}
	\uTwoPoint{6}{2}{7}{U5cup3-vert6}
	\uThreePointInner{4}{4}{6}{7}{U5cup3-vert4}
	\uFivePointInner{3}{1}{2}{3}{4}{5}{U5cup3-vert3}
	We also have the following parasitic and projective relations:
	\uParasit{1}{6}{U5cup3-parasit-1-6}
	\uParasit{1}{7}{U5cup3-parasit-1-7}
	\uParasit{2}{6}{U5cup3-parasit-2-6}
	\uParasit{3}{6}{U5cup3-parasit-3-6}
	\uParasit{3}{7}{U5cup3-parasit-3-7}
	\uParasit{5}{7}{U5cup3-parasit-5-7}
	\uProjRel{U5cup3-proj}{7}{6}{5}{4}{3}{2}{1}
	
Substituting \eqref{U5cup3-vert1}, \eqref{U5cup3-vert2} and Equating \eqref{U5cup3-vert3-9} and \eqref{U5cup3-vert3-10}, we get
$\ug{2}=\ug{2'}$. By Lemma \ref{2pt-equal}, $\ug{7}=\ug{7'}$. Then by Lemma \eqref{3pt-in-bigmid} we get $\ug{4}=\ug{4'}$ and $\ug{6}=\ug{6'}$. Then again by Lemma \eqref{2pt-equal} we have $\ug{5}=\ug{5'}$.

$\Ggal$ is thus generated by $\set{\ug{i} | i=1,\dots,7}$ with the following relations:
	\uGammaSq{U5cup3-gamma-sq}{1}{2}{3}{4}{5}{6}{7}
	\ubegineq{U5cup3-simpl-3}	
\begin{split}
\trip{1}{2} = \trip{1}{3} = \trip{2}{4} = \trip{2}{7} = \trip{3}{5} & =\\
= \trip{4}{5}= \trip{4}{6} = \trip{4}{7} =\trip{5}{6} =\trip{6}{7} & = e
\end{split}
\uendeq
	\ubegineq{U5cup3-simpl-4}
	\begin{split}
	\comm{1}{4} = \comm{1}{5} = \comm{1}{6} = \comm{1}{7} = \comm{2}{3} = \comm{2}{5} &=\\
	=\comm{2}{6} = \comm{3}{4} = \comm{3}{6} = \comm{3}{7} =  \comm{5}{7} &= e
	\end{split}
	\uendeq
 \ubegineq{U5cup3-simpl-5}
	\ug{7}=\ug{4}\ug{6}\ug{4}
	\uendeq
     \ubegineq{U5cup3-simpl-6}
	\ug{1}\ug{2}\ug{1}=\ug{4}\ug{5}\ug{3}\ug{5}\ug{4}.
	\uendeq
By Remark 2.8 in \cite{RTV}, we get also the relations
    \ubegineq{U5cup3-simpl-7}
	[\ug{2},\ug{4}\ug{7}\ug{4}] = [\ug{5},\ug{4}\ug{6}\ug{4}] = e.
	\uendeq
By Theorem 2.3 in \cite{RTV} we have $\Ggal \cong S_6$, therefore $\pi_1(\Xgal)$ is trivial.

\end{proof}

\uChernSummary{13}{36}{30}{14}{16}{\frac{19}{2}}{-}

%+--------------------------------------+
%|                               		    |
%|                CASE  U_6     		    |
%|                                      	    |
%+--------------------------------------+

\subsection{{$ X $ degenerates to $ U_6 $}}\label{section_computation_U_6}

\begin{figure}[H]
	\begin{center}
		
		\definecolor{circ_col}{rgb}{0,0,0}
		\begin{tikzpicture}[x=1cm,y=1cm,scale=2]
		
		\draw [fill=circ_col] (0,0) node [above] {1} circle (1pt);
		\draw [fill=circ_col] (1,0) circle (1pt);
		\draw [fill=circ_col] (-1,0) circle (1pt);
		\draw [fill=circ_col] (0.5,0.7) circle (1pt);
		\draw [fill=circ_col] (-0.5,0.7) circle (1pt);
		\draw [fill=circ_col] (0.5,-0.7) circle (1pt);
		\draw [fill=circ_col] (-0.5,-0.7) circle (1pt);
		
		\draw [-, line width = 1pt] (1,0) -- (0.5,0.7) -- (-0.5,0.7) -- (-1,0) -- (-0.5, -0.7) -- (0.5, -0.7) -- (1,0);
		
		\draw [-, line width = 1pt] (0,0) -- (1,0);
		\draw [-, line width = 1pt] (0,0) -- (0.5,0.7);
		\draw [-, line width = 1pt] (0,0) -- (0.5,-0.7);
		\draw [-, line width = 1pt] (0,0) -- (-1,0);
		\draw [-, line width = 1pt] (0,0) -- (-0.5,0.7);
		\draw [-, line width = 1pt] (0,0) -- (-0.5,-0.7);
		
		\end{tikzpicture}
	\end{center}
	\setlength{\abovecaptionskip}{-0.15cm}
	\caption{The arrangement of planes $U_6$.}\label{fig_computation_U_6}
\end{figure}
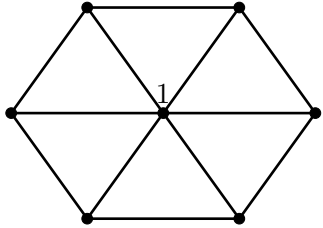

\begin{thm}\label{thm_computation_U_6}
	If $ X $ degenerates to $ U_6 $, then $ \pi_1(\Xgal) $ is trivial.
\end{thm}

\begin{proof}
	See \cite[Theorem 6]{ZAPP}.

\end{proof}

\uChernSummary{12}{24}{24}{12}{9}{7}{-\frac{5}{3}}

\end{document}